\documentclass[11pt]{amsart}

\newlength{\myhmargin}  
\usepackage[textheight=574pt, textwidth=445pt, marginparwidth=50pt, centering]{geometry}

\usepackage{amsmath,amssymb,amsthm,amsfonts,amscd,flafter,
graphicx,verbatim,pinlabel,mathrsfs}
\usepackage[all]{xy}
\usepackage{epstopdf}
\usepackage[colorlinks=false]{hyperref}
\usepackage[all]{hypcap}
\AtBeginDocument{\addtocontents{toc}{\protect\setlength{\parskip}{0pt}}}
\DeclareMathAlphabet\euscript{U}{eus}{m}{n}
\SetMathAlphabet\euscript{bold}{U}{eus}{b}{n}

\title{Khovanov homology detects the trefoils}

\author[John A. Baldwin]{John A. Baldwin}
\address{Department of Mathematics \\ Boston College}
\email{john.baldwin@bc.edu}

\author[Steven Sivek]{Steven Sivek}
\address{Department of Mathematics\\Imperial College London}
\email{s.sivek@imperial.ac.uk}

\thanks{JAB was supported by NSF Grant DMS-1406383 and NSF CAREER Grant DMS-1454865.}

\def\C{{\mathbb{C}}}

\def\ZZ{{\mathbb{Z}}}

\newcommand\sminus-

\newcommand\ssm{\smallsetminus}
\newcommand\cB{\mathcal{B}}
\newcommand\cC{\mathcal{C}}
\newcommand\cG{\mathcal{G}}
\newcommand\pt{\mathrm{pt}}

\newcommand\Z{\mathbb{Z}}

\newcommand\data{\mathscr{D}}
\newcommand\SFH{\mathit{SFH}}

\newcommand\CF{\mathit{\widehat{CF}}}
\newcommand\HFKm{\mathit{HFK^-}}
\newcommand\CFKm{\mathit{CFK^-}}
\newcommand\HFK{\mathit{\widehat{HFK}}}

\newcommand\EH{\mathit{EH}}
\newcommand\Kh{\mathit{Kh}}
\newcommand\Khr{\mathit{Khr}}

\newcommand\Inat{\mathit{I^\natural}}
\newcommand\SHI{\mathit{SHI}}
\newcommand\KHI{\mathit{KHI}}

\newcommand\cinvt{\theta}
\newcommand\linvt{\mathscr{L}}
\newcommand\kinvt{\mathscr{T}} 

\newcommand{\id}{\operatorname{id}}

\newcommand{\longcomment}[2]{#2}

\DeclareFontFamily{U}{mathx}{\hyphenchar\font45}
\DeclareFontShape{U}{mathx}{m}{n}{
      <5> <6> <7> <8> <9> <10>
      <10.95> <12> <14.4> <17.28> <20.74> <24.88>
      mathx10
      }{}
\DeclareSymbolFont{mathx}{U}{mathx}{m}{n}
\DeclareFontSubstitution{U}{mathx}{m}{n}
\DeclareMathAccent{\widecheck}{0}{mathx}{"71}
\newcommand{\HMto}{\widecheck{\mathit{HM}}}

\longcomment{
    \RequirePackage{rotating}                   
    \def\HMto{%
       \setbox0=\hbox{$\widehat{\mathit{HM}}$}
       \setbox1=\hbox{$\mathit{HM}$}
       \dimen0=1.1\ht0
       \advance\dimen0 by 1.17\ht1
       \smash{\mskip2mu\raise\dimen0\rlap{%
          \begin{turn}{180}
              {$\widehat{\phantom{\mathit{HM}}}$}
           \end{turn}} \mskip-2mu    
                \mathit{HM}
    }{\vphantom{\widehat{\mathit{HM}}}}{}}
}
    
    \newcommand*\oline[1]{%
  \vbox{%
    \hrule height 0.35pt
    \kern0.1ex
    \hbox{%
      \kern-0.0em
      \ifmmode#1\else\ensuremath{#1}\fi
      \kern-0.1em
    }
  }
}

\newtheorem{theorem}{Theorem}[section]
\newtheorem{lemma}[theorem]{Lemma}

\newtheorem{conjecture}[theorem]{Conjecture}
\newtheorem{corollary}[theorem]{Corollary}
\newtheorem{proposition}[theorem]{Proposition}
\newtheorem{question}[theorem]{Question}

\theoremstyle{definition}
\newtheorem{definition}[theorem]{Definition}
\newtheorem{notation}[theorem]{Notation}

\newtheorem{remark}[theorem]{Remark}

\newtheorem{example}[theorem]{Example}

\makeatletter
\newtheorem*{rep@thm}{\rep@title}
\newcommand{\newreptheorem}[2]{%
\newenvironment{rep#1}[1][0,0]{%
\def\rep@title{#2##1}%
\begin{rep@thm}}%
{\end{rep@thm}}}
\makeatother
\newreptheorem{theorem}{}

\makeatother
\usepackage{lipsum}

\begin{document}
\begin{abstract} 
We prove that Khovanov homology  detects  the trefoils. Our proof incorporates an array of ideas in Floer homology and contact geometry. It uses open books; the contact invariants  we defined in the instanton Floer setting; a    bypass exact triangle in   sutured instanton homology, proven here; and Kronheimer and Mrowka's spectral sequence relating Khovanov homology with singular instanton knot homology. As a byproduct, we also strengthen a result of Kronheimer and Mrowka on  $SU(2)$ representations of the knot group.
 \end{abstract}

\maketitle

\section{Introduction}
\label{sec:intro}

Khovanov homology assigns to a knot $K\subset S^3$ a bigraded abelian group \[\Kh(K)=\bigoplus_{i,j}\Kh^{i,j}(K)\] whose graded Euler characteristic recovers the  Jones polynomial of $K$. 
In their landmark paper \cite{km-khovanov}, Kronheimer and Mrowka proved that Khovanov homology detects the unknot, answering a weaker version of the famous open question below.
\begin{question} Does the Jones polynomial detect the unknot?
\end{question}
The question below is perhaps even more difficult.

\begin{question}\label{ques:jones-trefoil} Does the Jones polynomial detect the trefoils?
\end{question}
The  goal of this paper is to  prove that Khovanov homology detects the right- and left-handed trefoils, $T_+$ and $T_-$, answering a weaker version of Question \ref{ques:jones-trefoil}. 

Recall that $\Kh(T_+)$ and $\Kh(T_-)$ are both isomorphic to $\Z^4\oplus\Z/2\Z$   but are supported in different bigradings. Our main result is the following.


\begin{theorem}\label{thm:kh-detects-trefoil}
$\Kh(K)\cong \Z^4\oplus\Z/2\Z$ if and only if $K$ is a trefoil.
\end{theorem}

As a bigraded theory, Khovanov homology therefore detects each of $T_+$ and $T_-$.  One should not expect similar results for other knots in general, since for example Khovanov homology does not distinguish the knots $10_{22}$ and $10_{35}$ from each other.

Like Kronheimer and Mrowka's unknot detection result,  Theorem \ref{thm:kh-detects-trefoil} relies on a relationship between Khovanov homology and instanton Floer homology. More surprising is that our proof also hinges fundamentally on ideas from contact geometry. Essential tools  include the  invariant  of contact 3-manifolds with boundary we  defined  in the instanton Floer setting \cite{bs-shi};  our naturality result for sutured instanton homology  \cite{bs-naturality}; and an instanton Floer version of Honda's bypass exact triangle, established here. 

We describe below how our main Theorem \ref{thm:kh-detects-trefoil} follows from a certain result, Theorem \ref{thm:khi-fibered}, in the instanton Floer setting. We then explain how  the latter theorem can be used to strengthen a result of Kronheimer and Mrowka on  $SU(2)$ representations of the knot group. Finally, we outline both the ideas which motivated our approach to Theorem \ref{thm:khi-fibered} and the proof itself, and along the way we state a bypass exact triangle for instanton Floer homology.


\subsection{Trefoils and reduced Khovanov homology} We first note that Theorem \ref{thm:kh-detects-trefoil}   follows  from the detection result below for  reduced Khovanov homology $\Khr$. 

\begin{theorem}\label{thm:khr-detects-trefoil}
$\dim_\Z\Khr(K)=3$ if and only if $K$ is a trefoil.
\end{theorem}

To see how Theorem \ref{thm:kh-detects-trefoil} follows, let us suppose 
 $\Kh(K)\cong \Z^4\oplus\Z/2\Z.$ Then \[\Kh(K;\Z/2\Z)\cong (\Z/2\Z)^6\] by the Universal Coefficient Theorem. Recall the general facts that\begin{enumerate}
\item $\Kh(K)$ and $\Khr(K)$ fit into an exact triangle \[ \xymatrix@C=-15pt@R=15pt{
\Kh(K) \ar[rr] & & \Khr(K) \ar[dl] \\
& \Khr(K);  \ar[ul] & \\
} \] 
\item  $\Kh(K;\Z/2\Z)\cong \Khr(K;\Z/2\Z)\oplus \Khr(K;\Z/2\Z)$. 
\end{enumerate} The first implies that $\dim_\Z\Khr(K)\geq 2$ while the second implies that \[\Khr(K;\Z/2\Z) \cong(\Z/2\Z)^3.\] These together force   $\dim_\Z\Khr(K)=3$ by  another application of the UCT.  Therefore, $K$ is a trefoil by Theorem \ref{thm:khr-detects-trefoil}. We describe below how Theorem \ref{thm:khr-detects-trefoil} follows from Theorem \ref{thm:khi-fibered}.

\subsection{Trefoils and instanton Floer homology}

To prove that Khovanov homology detects the unknot, Kronheimer and Mrowka established in \cite{km-khovanov} a spectral sequence relating  Khovanov homology and singular instanton knot  homology, the latter of which assigns to a knot $K\subset Y$ an abelian group $\Inat(Y,K)$. In particular, they proved that \begin{equation*}\label{eqn:kh-i} \dim_\Z \Khr(K) \geq \dim_\Z \Inat(S^3,K).\end{equation*} Kronheimer and Mrowka moreover showed that the right side is odd and greater than one for nontrivial knots. Theorem \ref{thm:khr-detects-trefoil} therefore follows immediately from the  result below.

\begin{theorem}\label{thm:skhi-detects-trefoil}
If $\dim_\Z \Inat(S^3,K)=3$ then $K$ is a trefoil.
\end{theorem}

We prove Theorem \ref{thm:skhi-detects-trefoil} using yet another knot invariant. The instanton knot Floer homology of a knot $K\subset Y$  is a $\C$-module defined in \cite{km-excision} as the sutured instanton  homology of the knot complement with two oppositely oriented meridional sutures, \[\KHI(Y,K):=\SHI(Y(K),\Gamma_\mu):=\SHI(Y\ssm\nu(K),\mu\cup -\mu)\]
where $\mu$ is an oriented meridian.  It is related to singular instanton knot homology as follows \cite[Proposition 1.4]{km-khovanov}, \begin{equation*}\label{eqn:khi-i}\KHI(Y,K) \cong \Inat(Y,K)\otimes_\Z \C.\end{equation*} 
Theorem \ref{thm:skhi-detects-trefoil} therefore follows immediately from the result below.

\begin{theorem}\label{thm:khi-detects-trefoil}
If $\dim_\C \KHI(S^3,K)=3$ then $K$ is a trefoil.
\end{theorem}

Our proof of Theorem \ref{thm:khi-detects-trefoil} makes use of  some additional structure on $\KHI$. Namely, if $\Sigma$ is a Seifert surface for $K$ then $\KHI(Y,K)$ may be endowed with a symmetric Alexander grading, \[\KHI(Y,K)=\bigoplus_{i=-g(\Sigma)}^{g(\Sigma)}\KHI(Y,K,[\Sigma],i), \] where \[\KHI(Y,K,[\Sigma],i)\cong\KHI(Y,K,[\Sigma],-i) \text{ for all i}.\]  This grading depends only on the relative homology class of the surface in $H_2(Y,K)$. We will omit this class from the notation when it is unambiguous, as when $Y=S^3$. Kronheimer and Mrowka proved in \cite{km-excision} that if $K$ is fibered with fiber $\Sigma$ then \[\KHI(Y,K,[\Sigma],g(\Sigma))\cong \C.\] Moreover, they showed \cite{km-excision,km-alexander} that the  Alexander grading completely detects genus and fiberedness when $Y=S^3$. Specifically,  \begin{align}
\label{eqn:genus}&\KHI(S^3,K,g(K))\neq 0 \textrm{ and } \KHI(S^3,K,i)=0 \textrm{ for } i>g(K)\\
\label{eqn:genusfibered}&\KHI(S^3,K,g(K))\cong \C \textrm{ if and only if } K \textrm{ is fibered},\end{align} exactly as in Heegaard knot Floer homology.

We claim that  Theorem \ref{thm:khi-detects-trefoil} (and therefore each preceding theorem) follows  from the result below, which states that the instanton knot Floer homology of a fibered knot is nontrivial in the next-to-top Alexander grading.

\begin{theorem}\label{thm:khi-fibered}
Suppose $K$ is a genus $g>0$ fibered knot in $Y\not\cong \#^{2g}(S^1\times S^2)$ with fiber $\Sigma$. Then $\KHI(Y,K,[\Sigma],g{-}1)\neq 0$.
\end{theorem}

To see how Theorem \ref{thm:khi-detects-trefoil} follows, let us
 suppose  that \[\dim_\C \KHI(S^3,K)=3.\] Then $\KHI(S^3,K)$ is supported in Alexander gradings $0$ and $\pm g(K)$ by  symmetry  and  genus detection \eqref{eqn:genus}. Note that $g(K)\geq 1$ since  $K$ is otherwise the unknot and \[\dim_\C \KHI(S^3,K)=1,\] a contradiction. So we have that \[\KHI(S^3,K,i)\cong \begin{cases}
\C,&i=g(K),\\
\C,&i=0,\\
\C,&i=-g(K).
\end{cases}\] The fiberedness detection \eqref{eqn:genusfibered} therefore implies that $K$ is fibered. But Theorem \ref{thm:khi-fibered} then forces $g(K)=1$. We conclude that $K$ is a genus one fibered knot. It follows that $K$ is either a trefoil or the figure eight, but $\KHI$ of the latter is 5-dimensional, so $K$ is a trefoil.

\begin{remark} \label{rem:alexander-fibered-gap}
The nonvanishing guaranteed by Theorem~\ref{thm:khi-fibered} is a feature of instanton knot homology which is not present in the Alexander polynomial.  For example, the knot $10_{161}$ is fibered of genus 3, with Alexander polynomial
\[ \Delta_{10_{161}}(t) = t^3 - 2t + 3 - 2t^{-1} + t^{-3}. \]
Theorem~\ref{thm:khi-fibered} says that $\KHI(S^3,10_{161},2)$ is nonzero even though the coefficient of $t^2$ in $\Delta_{10_{161}}(t)$ vanishes.
\end{remark}

In summary, we have shown that Theorem \ref{thm:khi-fibered} implies all of the other results above including  that Khovanov homology detects the trefoils. The bulk of this paper is therefore devoted to proving Theorem \ref{thm:khi-fibered}. Before outlining its proof in detail below, we describe an application  of Theorem \ref{thm:skhi-detects-trefoil} to $SU(2)$ representations of the knot group.

\subsection{Trefoils and   $SU(2)$ representions}
\label{ssec:trefoil-reps}


Given a knot $K$ in the 3-sphere, consider the representation variety \[\mathscr{R}(K,\mathbf{i}) = \{\rho:\pi_1(S^3\ssm K)\to SU(2)\mid \rho(\mu)=\mathbf{i}\},\] where $\mu$ is a chosen meridian and \[\mathbf{i}=\left[\begin{array}{cc}
i&0\\
0&-i
\end{array}\right].\] Recall that the representation variety of a trefoil $T$ is given by \[\mathscr{R}(T,\mathbf{i}) \cong \{*\} \sqcup S^1,\] where   $*$ is the    reducible homomorphism in $\mathscr{R}(T,\mathbf{i})$ and  $S^1$ is the unique conjugacy class of irreducibles.  We conjecture that $\mathscr{R}(K,\mathbf{i})$ detects the trefoil.


\begin{conjecture}
$\mathscr{R}(K,\mathbf{i}) \cong \{*\} \sqcup S^1$ if and only if $K$ is a trefoil.
\end{conjecture}

We prove this conjecture modulo an assumption of nondegeneracy, using Theorem \ref{thm:khi-detects-trefoil} together with the relationship between $\mathscr{R}(K,\mathbf{i})$ and  $\KHI$ described in \cite[Section 7.6]{km-excision} and \cite[Section 4.2]{km-alexander}. 

The rough idea is that points in $\mathscr{R}(K,\mathbf{i})$ should correspond to  critical points of the Chern-Simons functional whose Morse-Bott homology computes $\KHI(S^3,K)$; the reducible corresponds to a single critical point while conjugacy classes of irreducibles ought to correspond to circles of critical points. In other words, the reducible should contribute 1 generator and each  class of irreducibles should contribute 2 generators (generators of the homology of the corresponding circle of critical points) to a chain complex which computes $\KHI(S^3,K)$. This heuristic holds true as long as   the circles of critical points corresponding to irreducibles are nondegenerate in the Morse-Bott sense. 
 Thus, if $n(K)$ is   the number of conjugacy classes of irreducibles and the corresponding circles of critical points are nondegenerate then \[\dim_\C\KHI(S^3,K)\leq 1+2n(K).\] Theorem \ref{thm:khi-detects-trefoil} therefore implies the following.

\begin{theorem}
\label{thm:rep-detects-trefoil}
Suppose  there is one conjugacy class of irreducible homomorphisms in $\mathscr{R}(K,\mathbf{i})$. If these homomorphisms are nondegenerate, then $K$ is a trefoil.
\end{theorem}

This  improves  upon a  result of Kronheimer and Mrowka \cite[Corollary 7.20]{km-excision} which under the same hypotheses concludes only that $K$ is fibered.   

\subsection{The proof of Theorem \ref{thm:khi-fibered}} 
\label{ssec:outline}The rest of this introduction is  devoted to explaining the proof of Theorem \ref{thm:khi-fibered}. This result and its proof were  inspired by work of Baldwin and Vela-Vick \cite{bvv} who proved the  following analogous result in Heegaard knot Floer homology.

\begin{theorem}
\label{thm:hfk-fibered}
Suppose $K$ is a genus $g>0$ fibered knot  in $Y\not\cong \#^{2g}(S^1\times S^2)$ with fiber $\Sigma$. Then $\HFK(Y,K,[\Sigma],g{-}1)\neq 0$.\footnote{The conclusion of this theorem  also holds for $Y\cong \#^{2g}(S^1\times S^2)$.}
\end{theorem}

Theorem \ref{thm:hfk-fibered} can be used to give new proofs that the dimension of $\HFK$ detects the trefoil \cite{hw} and that   $L$-space knots are prime \cite{krcatovich}. It has no bearing, however, on whether Khovanov homology detects the trefoils, as there is no known relationship between Khovanov  homology and Heegaard  knot Floer homology. 

We summarize below the proof of Theorem \ref{thm:hfk-fibered} from \cite{bvv} and then explain how it can be reformulated in a manner that is translatable to the instanton Floer setting. 

Suppose $K$ is a fibered knot as in Theorem \ref{thm:hfk-fibered} and let $(\Sigma,h)$ be an open book with page $\Sigma$ and monodromy $h$ corresponding to the fibration of $K$ with $g(\Sigma)=g,$ supporting a contact structure $\xi$ on $Y$.

When it suits us, we are free to assume in proving Theorem \ref{thm:khi-fibered} that $h$ is \emph{not right-veering}, meaning that $h$ sends some arc in $\Sigma$ \emph{to the left} at one of its endpoints, as shown in Figure \ref{fig:left} and made precise in  \cite{hkm-rv}. 
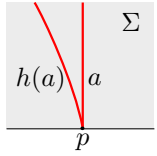
\begin{figure}[ht]
\labellist
\small \hair 2pt

\pinlabel $a$ at 47 28
\pinlabel $h(a)$ at 20 28
\pinlabel $p$ at 41 -3
\pinlabel $\Sigma$ at 65 57

\endlabellist
\centering
\includegraphics[width=2cm]{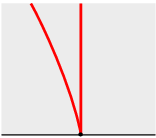}
\caption{$h$ sends $a$ to the left at $p$.
}
\label{fig:left}
\end{figure}
To see that we can make this assumption without loss of generality, note that one of $h$ or $h^{-1}$ is not right-veering since otherwise  $h=\id$ and $Y\cong \#^{2g}(S^1\times S^2)$. (Indeed, if $h=\id$ then we can take $2g$ disjoint, properly embedded arcs which cut $\Sigma$ into a disk, and these arcs trace out $2g$ embedded spheres in $Y$ whose complement is $S^3$ minus $4g$ balls, so gluing the spheres back in gives $\#^{2g}(S^1\times S^2)$.  Alternatively, the open book $(\Sigma,\id)$ is a Murasugi sum of $b_1(\Sigma)=2g$ open books $(S^1\times[0,1],\id)$, each of which supports $S^1\times S^2$ with its tight contact structure, and hence $(\Sigma,\id)$ supports their connected sum.) If $h$ is right-veering then we can use the fact that knot Floer homology is invariant under reversing the orientation of $Y$ and consider instead the knot $K\subset -Y$ with open book $(\Sigma,h^{-1})$.

Recall that the knot Floer homology of $-K\subset -Y$ is the homology of the associated graded object of a filtration 
\[\euscript{F}_{-g}\subset\euscript{F}_{1-g}\subset \dots \subset \euscript{F}_{g}=\CF(-Y)\] of the Heegaard Floer complex of $-Y$ induced by the knot. By careful inspection of a Heegaard diagram for $-K\subset -Y$  adapted to the open book $(\Sigma,h)$, Baldwin and Vela-Vick prove: 
\begin{lemma}
\label{lem:nonrv} If the monodromy $h$ is not right-veering then there exist $c\in\euscript{F}_{-g}$ and $d\in \euscript{F}_{1-g}$ such that $[c]$ generates $H_*(\euscript{F}_{-g})\cong\Z/2\Z$ and $\partial d=c$.
\end{lemma}
To see how Lemma \ref{lem:nonrv} implies Theorem \ref{thm:hfk-fibered}, let us assume that the monodromy $h$ is not right-veering. Given $c$ and $d$ as guaranteed by Lemma \ref{lem:nonrv}, 
it is then an easy exercise to see that $\partial \circ \partial = 0$ and $[c]$ nonzero imply that the class
\[ [d] \in H_*(\euscript{F}_{1-g}/\euscript{F}_{-g}) = \HFK(-Y,-K,-[\Sigma],1-g) \]
is nonzero.
Theorem \ref{thm:hfk-fibered} then follows from the symmetry \[\HFK(Y,K,[\Sigma],g-1)\cong\HFK(-Y,-K,-[\Sigma],1-g).\]

Our strategy is to translate a version of this proof to the instanton Floer setting. Of course, it does not translate readily. For one thing, it makes use of Heegaard diagrams in an essential way. For another, it relies on a description of knot Floer homology as coming from a filtration of the Floer complex of the ambient manifold, for which there is no analogue in $\KHI$.  

Our solution to these difficulties starts with a reformulation of Lemma \ref{lem:nonrv} in terms of the \emph{minus} version of knot Floer homology, which assigns to a knot   a module over the polynomial ring $(\mathbb{Z}/2\mathbb{Z})[U]$.  Specifically, we observe that Lemma \ref{lem:nonrv} can be recast as follows:

\begin{lemma}
\label{lem:nonrv2} If the monodromy $h$ is not right-veering then the generator of \[\HFKm(-Y,K,[\Sigma],g)\cong\Z/2\Z\] is in the kernel of multiplication by $U$.
\end{lemma}

Indeed,  if we reverse the roles of the $z$ and $w$ basepoints in the above Heegaard diagram for $-K\subset -Y$, then we obtain a diagram for $K \subset -Y$ where the generators $c$ and $d$ lie in filtration gradings $g$ and $g-1$, respectively.  In $\CFKm$, one records disks which cross the $w$ basepoint $n$ times with the coefficient $U^n$, so the relation $\partial d = c$ in $\CF$ from the original Heegaard diagram, which is certified by  a holomorphic disk from $d$ to $c$ passing through $z$ once,  becomes $\partial d = Uc$ in $\CFKm$ after the basepoint swap; hence, $U[c] = 0$ in $\HFKm$.

It may seem as though this reformulation of Lemma \ref{lem:nonrv} makes translation even \emph{more} difficult, as there is no analogue of $\HFKm$ whatsoever in the instanton Floer setting. Surprisingly, however, it is Lemma \ref{lem:nonrv2} that proves most amenable to translation.

Our approach is inspired by work of Etnyre, Vela-Vick, and Zarev, who provide in \cite{evvz} a more contact-geometric description of $\HFKm$ with its $(\Z/2\Z)[U]$-module structure. As we show below, their work enables a proof of Lemma \ref{lem:nonrv2} in terms of sutured  Floer homology groups, bypass attachment maps, and contact invariants. The value for us in proving Lemma \ref{lem:nonrv2} from this perspective is that while there is no analogue of $\HFKm$ in the instanton Floer setting, there \emph{are} instanton Floer analogues of these groups, bypass maps, and contact invariants, due to Kronheimer and Mrowka \cite{km-excision} and the authors \cite{bs-shi}. We are thus able to port key elements of this alternative proof of Lemma \ref{lem:nonrv2} to the instanton Floer setting and, with additional work, use these elements to prove Theorem \ref{thm:khi-fibered}. Below, we: 
\begin{itemize}
\item review the work of \cite{evvz}, tailored to the case of our fibered knot $K$,
\item prove Lemma \ref{lem:nonrv2} from this \emph{direct limit} point of view,
\item outline in detail the proof of Theorem \ref{thm:khi-fibered}, based on these ideas.
\end{itemize} 

As the binding of the open book $(\Sigma,h)$, the knot $K$ is naturally a transverse knot in $(Y,\xi)$. Moreover, we will show in Section~\ref{sec:proof} that $K$ has a Legendrian approximation $\mathcal{K}_0^-$ with Thurston-Bennequin invariant \[tb_\Sigma(\mathcal{K}_0^-)=-1.\] For each $i\geq 1$, let $\mathcal{K}^{\pm}_i$ be the result of negatively Legendrian stabilizing the knot $\mathcal{K}^-_0$ a total of $i-1$ times and then positively/negatively stabilizing the result one additional time. Note that each $\mathcal{K}^-_i$ is also a Legendrian approximation of $K$. Let \[(Y(K),\Gamma_i,\xi^\pm_i)\] be the contact manifold with convex boundary and dividing set $\Gamma_i$ obtained by removing a standard neighborhood of $\mathcal{K}_i^\pm$ from $Y$.  These contact manifolds are related to one another via positive and negative bypass attachments. By  work of Honda, Kazez, and Mati{\'c} in \cite{hkm-tqft}, these bypass attachments induce maps on sutured Floer homology, \[\psi_i^\pm:\SFH(-Y(K),-\Gamma_i)\to \SFH(-Y(K),-\Gamma_{i+1})\] for each $i$, which satisfy \[
\psi_i^-(\EH(\xi_i^-)) = \EH(\xi_{i+1}^-)\textrm{ and } \psi_i^+(\EH(\xi_i^-)) = \EH(\xi_{i+1}^+),
\]
where $\EH$ refers to the  Honda-Kazez-Mati{\'c} contact invariant defined in \cite{hkm-sutured}. The main result of \cite{evvz} says that $\HFKm(-Y,K)$ is isomorphic to the direct limit \begin{equation*}\label{eqn:limit}\SFH(-Y(K),-\Gamma_0)\xrightarrow{\psi_0^-}\SFH(-Y(K),-\Gamma_1)\xrightarrow{\psi_1^-}\SFH(-Y(K),-\Gamma_2)\xrightarrow{\psi_2^-}\cdots\end{equation*} of these sutured Floer homology groups and the negative bypass attachment maps. Moreover, under this identification, multiplication by $U$ is  the map on this  limit induced by the positive bypass  attachment maps $\psi_i^+$.


We now observe  that Lemma \ref{lem:nonrv2} has a very natural interpretation and proof in this direct limit formulation. 
The first step  is to identify the element of the direct  limit which corresponds to the generator of  $\HFKm(-Y,K,[\Sigma],g)$. 
For this, recall that Vela-Vick proved in \cite{vv} that the transverse binding $K$ has nonzero   invariant \[\euscript{T}(K)\in\HFKm(-Y,K),\] where $\euscript{T}$ refers to the transverse knot invariant defined by Lisca, Ozsv{\'a}th, Stipsicz, and Szab{\'o} in \cite{loss}. Moreover, this class lies in Alexander grading \begin{equation}\label{eqn:alexgradingcalc}(sl(K)+1)/2=g \end{equation} according to \cite{loss}.
So $\euscript{T}(K)$ is the generator of \[\HFKm(-Y,K,[\Sigma],g)\cong \Z/2\Z.\] But Etnyre, Vela-Vick, and Zarev proved that $\euscript{T}(K)$ corresponds to the element of  the direct limit represented by the contact invariant \[\EH(\xi_i^-)\in\SFH(-Y(K),-\Gamma_i)\] for any $i$. 
It follows that $U \euscript{T}(K)$ corresponds to the element of the  limit represented by \[\psi_0^+(\EH(\xi_0^-))=\EH(\xi_1^+).\] 
 Lemma \ref{lem:nonrv2} therefore follows from the lemma below.
\begin{lemma}
\label{lem:bypassclaim1} If the monodromy $h$ is not right-veering then $\psi_0^+(\EH(\xi_0^-))=\EH(\xi_1^+)=0.$\end{lemma}
 But this lemma follows immediately from the result below since the   $\EH$ invariant vanishes for overtwisted contact manifolds.
\begin{lemma}
\label{lem:otstab}
If the monodromy $h$ is not right-veering then $\xi_1^+$ is overtwisted.
\end{lemma}

This concludes our alternative proof of Lemma \ref{lem:nonrv2}, modulo the proof of Lemma \ref{lem:otstab} which we provide in Section \ref{sec:proof}. 
We now describe in detail our proof of  Theorem \ref{thm:khi-fibered}, inspired by these ideas.

The instanton Floer analogues of  $\EH(\xi_i^\pm)$ and $\psi_i^\pm$ are the contact invariants \[\cinvt(\xi_i^\pm)\in\SHI(-Y(K),-\Gamma_i)\] and bypass attachment maps \[\phi_i^\pm:\SHI(-Y(K),-\Gamma_i)\to\SHI(-Y(K),-\Gamma_{i+1})\] we defined in \cite{bs-shi}. Guided by the discussion above, our  approach to proving Theorem \ref{thm:khi-fibered}  begins with the following analogue of Lemma \ref{lem:bypassclaim1}.
\begin{lemma}\label{lem:bypassclaim}If the monodromy $h$ is not right-veering then $\phi_0^+(\cinvt(\xi_0^-))=\cinvt(\xi_1^+)=0$.\end{lemma} 

We note that Lemma \ref{lem:bypassclaim} follows immediately from  Lemma \ref{lem:otstab} since our contact  invariant $\theta$ vanishes for overtwisted contact manifolds, just as the $\EH$ invariant does.

Unfortunately, Lemma \ref{lem:bypassclaim} does not  automatically   imply Theorem \ref{thm:khi-fibered} in the same way that Lemma \ref{lem:bypassclaim1} implies Theorem \ref{thm:hfk-fibered}, as the latter implication ultimately makes use of structure that is  unavailable in the instanton Floer setting. Indeed, proving Theorem \ref{thm:khi-fibered} from the starting point of Lemma \ref{lem:bypassclaim} requires some  additional  ideas, as explained below.

First, we recall that in \cite{stipsicz-vertesi}, Stipsicz and V{\'e}rtesi proved  that the \emph{hat} version of the Lisca, Ozsv{\'a}th, Stipsicz, Szab{\'o} transverse invariant, \[\widehat{\euscript{T}}(K)\in\HFK(-Y,K),\] can be described as the $\EH$ invariant of the contact manifold obtained by attaching a certain bypass to the  complement of a standard neighborhood of \emph{any} Legendrian approximation of $K$. In particular, the contact manifold resulting from these Stipsicz-V{\'e}rtesi bypass attachments is independent of the Legendrian approximation, as shown in the proof of \cite[Theorem~1.5]{stipsicz-vertesi}. Inspired by this, we  define an   element \[{\kinvt}(K):=\phi^{SV}_i(\cinvt(\xi_i^-))\in \KHI(-Y,K),\] where \begin{equation*}\label{eqn:mapsv}\phi_i^{SV}:\SHI(-Y(K),-\Gamma_i)\to\SHI(-Y(K),-\Gamma_\mu)=\KHI(-Y,K)\end{equation*} is the map   our work  \cite{bs-shi}  assigns to the Stipsicz-V{\'e}rtesi bypass attachment. Since each $\mathcal{K}_i^-$ is a Legendrian approximation of $K$, the contact manifold obtained from these  attachments, and hence $\kinvt(K)$, is independent of $i$.  

We prove that the $\kinvt$ invariant of the transverse binding $K$  lies in the top Alexander grading, just as in  Heegaard Floer homology:

\begin{theorem}
\label{thm:ttopgrading}
$\kinvt(K)\in\KHI(-Y,K,[\Sigma],g)$.
\end{theorem}

Moreover, we prove the following analogue of Vela-Vick's result \cite{vv}  that the transverse binding of an open book has nonzero Heegaard Floer invariant.

\begin{theorem}
\label{thm:nonzero}
$\kinvt(K)$ is nonzero.
\end{theorem}

\begin{remark}
To be  clear, Theorems \ref{thm:ttopgrading} and \ref{thm:nonzero} hold without any assumption on  $h$.\end{remark}

\begin{remark}Our proof of Theorem \ref{thm:nonzero} relies on  formal properties of our contact invariants as well as  the surgery exact triangle and  adjunction inequality in instanton Floer homology. In fact, our argument can be ported directly to the Heegaard Floer setting to give a new proof of Vela-Vick's theorem. 
\end{remark}

The task remains to put all of these pieces together to conclude Theorem \ref{thm:khi-fibered}. This involves proving a bypass exact triangle in sutured instanton homology analogous to Honda's triangle in sutured Heegaard Floer homology. In Section \ref{sec:bypass} we prove the following.

\begin{theorem}
\label{thm:bypass}
Suppose $\Gamma_1,\Gamma_2,\Gamma_3\subset \partial M$ is a 3-periodic sequence of sutures related by the moves in a bypass triangle as in Figure \ref{fig:bypass-triangle2}. Then there is an exact triangle
\[ \xymatrix@C=-35pt@R=30pt{
\SHI(-M,-\Gamma_1) \ar[rr] & & \SHI(-M,-\Gamma_2) \ar[dl] \\
& \SHI(-M,-\Gamma_3), \ar[ul] & \\
} \]
in which the  maps are the corresponding bypass attachment maps.
\end{theorem}

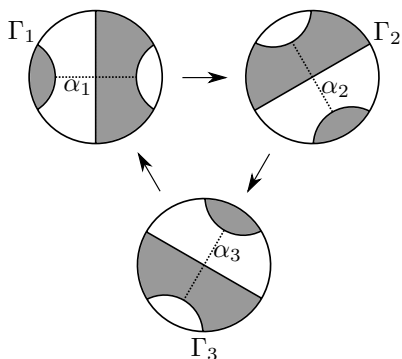
\begin{figure}[ht]
\labellist
\small \hair 2pt
\pinlabel $\alpha_1$  at 24 117
\pinlabel $\alpha_2$  at 149 118
\pinlabel $\alpha_3$  at 97 37
\pinlabel $\Gamma_1$  at -3 145
\pinlabel $\Gamma_2$  at 173 144
\pinlabel $\Gamma_3$  at 85 -8
\endlabellist
\centering
\includegraphics[width=4.7cm]{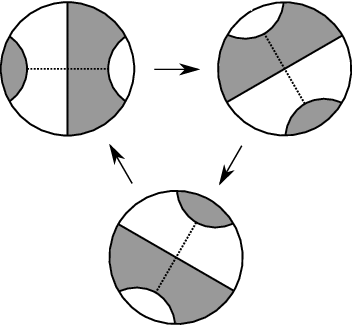}
\caption{The bypass triangle.  Each picture shows the  arc $\alpha_i$ along which a bypass is attached to achieve the next set of sutures in the triangle. The gray and white regions indicate the negative and positive regions, respectively.}
\label{fig:bypass-triangle2}
\end{figure}

For example, we  show that the map $\phi_0^+$ fits into a bypass exact triangle of the form
\begin{equation}\label{eqn:bypasstriintro} \xymatrix@C=-25pt@R=35pt{
\SHI(-Y(K),-\Gamma_0) \ar[rr]^{\phi_0^+} & & \SHI(-Y(K),-\Gamma_1) \ar[dl]^{\phi_1^{SV}} \\
& \KHI(-Y,K) \ar[ul]^C, & \\
} \end{equation}
as provided by Theorem~\ref{thm:bypass}, in which the map $\phi_1^{SV}$ comes from attaching a Stipsicz-V\'ertesi bypass.

To prove Theorem \ref{thm:khi-fibered}, let us now assume that the monodromy $h$ is not right-veering. Then \[\phi_0^+(\cinvt(\xi_0^-))=0\] by Lemma \ref{lem:bypassclaim}. Exactness of the triangle \eqref{eqn:bypasstriintro} then tells us that there is a class $x\in\KHI(-Y,K)$ such that \[C(x)=\cinvt(\xi_0^-).\] 
 The composition \begin{equation}\label{eqn:phicompintro}\phi_0^{SV}\circ C: \KHI(-Y,K)\to \KHI(-Y,K)\end{equation} therefore satisfies \begin{equation*}\phi_0^{SV}(C(x))=\kinvt(K),\end{equation*} which is nonzero by Theorem \ref{thm:nonzero}. It follows that the class $x$ is nonzero as well.
 
 Although the map in \eqref{eqn:phicompintro} is not \emph{a priori} homogeneous with respect to the Alexander grading, we prove that it shifts the  grading by at most 1. On the other hand, this composition  is trivial on  the top summand \[\KHI(-Y,K,[\Sigma],g)\cong\C\] since by Theorems \ref{thm:ttopgrading} and \ref{thm:nonzero} this summand is generated by $\kinvt(K)=\phi_1^{SV}(\xi_1^-)$, and \[C(\kinvt(K)) = C(\phi_1^{SV}(\xi_1^-))=0\] by  exactness of the triangle \eqref{eqn:bypasstriintro}. This immediately implies the result below.

\begin{theorem}
The component of $x$  in $\KHI(-Y,K,[\Sigma],g-1)$ is nonzero.
\end{theorem}

Theorem \ref{thm:khi-fibered} then follows from the symmetry \[\KHI(Y,K,[\Sigma],g-1)\cong \KHI(-Y,K,[\Sigma],g-1).\]

This completes our outline of the proof of Theorem \ref{thm:khi-fibered}. There are several  challenges  involved in making this outline  rigorous. The most substantial and interesting of these has to do with the Alexander grading, as described below.

\subsection{On the Alexander grading}
Kronheimer and Mrowka define the Alexander  grading on $\KHI$ by embedding the knot complement in a  particular closed $3$-manifold. On the other hand, the argument outlined above relies on the contact invariants in $\SHI$ we defined in \cite{bs-shi} and our naturality results from \cite{bs-naturality} (the latter tell us that different choices in the construction of $\SHI$ yield groups that are \emph{canonically} isomorphic, which is needed to talk sensibly about maps between $\SHI$ groups). Both require that we use a much larger class of \emph{closures}. Accordingly, one obstacle we had to overcome was showing that the Alexander grading can be defined in this broader setting in such a way that it agrees with the one Kronheimer and Mrowka defined (so that it still detects genus and fiberedness). We hope this contribution might prove useful for other purposes as well.

\subsection{Organization}  Section \ref{sec:bkgnd} provides the necessary background on instanton Floer homology, sutured instanton homology, and our contact invariants. We also prove several results in this section which do not appear elsewhere but are familiar to experts. In Section \ref{sec:alex}, we  give a more robust definition of the Alexander grading associated with a properly embedded surface in a sutured manifold. In Section \ref{sec:bypass}, we  prove a bypass exact triangle in  sutured instanton homology. In Section \ref{sec:leg}, we define invariants of Legendrian and transverse knots in $\KHI$ and establish some of their basic properties. In Section \ref{sec:proof}, we prove Theorem \ref{thm:khi-fibered} according to the outline above. As discussed, this theorem implies the other theorems stated above, including our main result that Khovanov homology detects the trefoil. 

\subsection{Acknowledgments} We thank Chris Scaduto and Shea Vela-Vick for helpful conversations, and the referees for many useful comments which improved the exposition.  We also thank Etnyre, Vela-Vick, and Zarev for their beautiful article \cite{evvz} which inspired certain parts of our approach. Finally, we would like the acknowledge the debt this paper owes to the foundational work of Kronheimer and Mrowka.

\section{Background}
\label{sec:bkgnd}

\subsection{Instanton Floer homology}
\label{ssec:instanton} This section provides the necessary background on instanton Floer homology. Our discussion is borrowed from \cite{km-excision}, though we include proofs of some propositions and lemmas which are familiar to experts but do not appear explicitly elsewhere. Our description of the surgery exact triangle is taken from \cite{scaduto}.  All surfaces in this paper will be orientable.

Let $(Y,\alpha)$ be an \emph{admissible pair}; that is, a closed, oriented 3-manifold $Y$ and a closed, oriented 1-manifold $\alpha \subset Y$ intersecting some embedded surface transversally in an odd number of points.  We associate the following data to this pair:
\begin{itemize}
\item A Hermitian line bundle $w \to Y$ with $c_1(w)$ Poincar\'e dual to $\alpha$;
\item A $U(2)$ bundle $E \to Y$ equipped with an isomorphism $\cinvt: \wedge^2 E \to w$.
\end{itemize}
The \emph{instanton Floer homology} $I_*(Y)_\alpha$ is the Morse homology of the Chern-Simons functional on the space $\cB=\cC/\cG$ of $SO(3)$ connections on $\operatorname{ad}(E)$ modulo determinant-1 gauge transformations, as in \cite{donaldson-book}. It is a $\ZZ/8\ZZ$-graded $\C$-module. 

\begin{notation}
Given disjoint oriented $1$-manifolds $\alpha,\eta\subset Y$ we will use the shorthand \[I_*(Y)_{\alpha+\eta}:=I_*(Y)_{\alpha\sqcup\eta}\] as it will make the notation cleaner in what follows.
\end{notation}
For each even-dimensional class $\Sigma \in H_d(Y)$, there is an operator \[\mu(\Sigma):I_*(Y)_\alpha\to I_{*+d-4}(Y)_\alpha,\] defined by a cohomology class in $H^{4-d}(\cB)$; these are introduced in \cite[\S7.2]{km-excision} as a straightforward adaptation of a construction for closed 4-manifold invariants in \cite{donaldson-kronheimer}, though they had already appeared in some form in \cite{munoz}. These operators are additive in  that \[\mu(\Sigma_1+\Sigma_2)=\mu(\Sigma_1)+\mu(\Sigma_2).\] Moreover, any two such operators commute.  Using work of Mu\~noz \cite{munoz}, Kronheimer and Mrowka  prove the following in  \cite[Corollary~7.2]{km-excision}.
\begin{theorem} 
\label{thm:simultaneouseigenvalues} Suppose $R$ is a closed surface in $Y$ of positive genus with $\alpha\cdot R$ odd. Then the simultaneous eigenvalues of the operators $\mu(R)$ and $\mu(\pt)$ on $I_*(Y)_\alpha$ belong to a subset of the pairs
\[ (i^r(2k), (-1)^r\cdot 2) \]
for $0 \leq r \leq 3$ and $0 \leq k \leq g(R)-1$, where $i = \sqrt{-1}$.
\end{theorem}  
With this, they make the following definition.
\begin{definition}Given $Y,\alpha,R$  as in Theorem \ref{thm:simultaneouseigenvalues}, let
\[ I_*(Y|R)_\alpha\subset I_*(Y)_\alpha \]
be the simultaneous generalized $(2g(R)-2,2)$-eigenspace of $(\mu(R),\mu(\pt))$ on $I_*(Y)_\alpha$.
\end{definition}


The commutativity of these operators implies that for any closed surface $\Sigma\subset Y$ the operator $\mu(\Sigma)$ acts on $I_*(Y|R)_\alpha$. Moreover, Kronheimer and Mrowka obtain the following bounds on the spectrum of this operator \emph{without} the assumption that $\alpha\cdot \Sigma$ is odd \cite[Proposition~7.5]{km-excision}.  

\begin{proposition} \label{prop:mu-spectrum}
For any closed surface $\Sigma\subset Y$ of positive genus, the eigenvalues of \[\mu(\Sigma):I_*(Y|R)_\alpha\to I_{*-2}(Y|R)_\alpha\] belong to the set of even integers between $2-2g(\Sigma)$ and $2g(\Sigma)-2$.
\end{proposition}

We are interested in studying the eigenspace decompositions of operators of the form $\mu(\Sigma)$, which may not be diagonalizable.  In this section, we will use \emph{eigenspace} in proofs to mean \emph{generalized eigenspace}, but in the statements of results we will always say ``generalized eigenspace'' so that they can be understood unambiguously later.

\begin{lemma}
\label{lem:symmetric}
If $g(R)=1$ then the generalized $m$-eigenspace of $\mu(\Sigma)$ acting on $I_*(Y|R)_\alpha$ is isomorphic to its generalized $-m$-eigenspace for each $m$.
\end{lemma}

\begin{proof} Suppose $g(R)=1$. Then the $m$-eigenspace of $\mu(\Sigma)$ acting on $I_*(Y|R)_\alpha$ is the simultaneous $(m,0,2)$-eigenspace of the operators $(\mu(\Sigma),\mu(R),\mu(\pt))$ on $I_*(Y)_\alpha$. Recall that $I_*(Y)_\alpha$ is a $\Z/8\Z$-graded group. We may thus write an element of this group as  $(c_0,c_1,c_2,c_3,c_4,c_5,c_6,c_7)$, where $c_i$ is in grading $i$ mod 8. It then follows immediately from the fact that $\mu(\Sigma)$ and $\mu(R)$ are degree 2 operators and $\mu(\pt)$ is a degree 4 operator that the map which sends \[(c_0,c_1,c_2,c_3,c_4,c_5,c_6,c_7) \textrm{ to }(c_0,c_1,-c_2,-c_3,c_4,c_5,-c_6,-c_7)\] defines an isomorphism from the 
$(m,0,2)$-eigenspace of $(\mu(\Sigma),\mu(R),\mu(\pt))$ to the $(-m,0,2)$-eigenspace of these operators.
\end{proof}

Suppose $(Y_1,\alpha_1)$ and $(Y_2,\alpha_2)$ are admissible pairs. A cobordism  $(W,\nu)$ from the first pair  to the second induces   a map \begin{equation*}\label{eqn:maptwo}I_*(W)_\nu:I_*(Y_1)_{\alpha_1}\to I_*(Y_2)_{\alpha_2}\end{equation*} which depends up to sign only on the homology class  $[\nu]\in H_2(W,\partial W;\Z/2\Z)$ and the isomorphism class of $(W,\nu)$, where two such pairs  are isomorphic if they are diffeomorphic by a map which intertwines the boundary identifications (the surface $\nu$ specifies a bundle over $W$ restricting to the bundles on the boundary specified by $\alpha_1$ and $\alpha_2$). Moreover, if $\Sigma_1\subset Y_1$ and $\Sigma_2\subset Y_2$ are homologous in $W$ then \begin{equation}\label{eqn:commute}\mu(\Sigma_2)(I_*(W)_\nu(x)) = I_*(W)_\nu(\mu(\Sigma_1)x),\end{equation} which implies the following.

\begin{lemma}
\label{lem:commute}
Suppose $x\in I_*(Y_1)_{\alpha_1}$ is in the generalized $m$-eigenspace of $\mu(\Sigma_1)$. Then $I_*(W)_\nu(x)$ is in the generalized $m$-eigenspace of $\mu(\Sigma_2)$.
\end{lemma}
\begin{proof}
Since $x\in I_*(Y_1)_{\alpha_1}$ is in the $m$-eigenspace of $\mu(\Sigma_1)$, there exists an integer $N$ such that \[(\mu(\Sigma_1)-m)^Nx=0.\] The relation \eqref{eqn:commute} then implies that \[(\mu(\Sigma_2)-m)^NI_*(W)_\nu(x) = I_*(W)_\nu\big((\mu(\Sigma_1)-m)^Nx\big)=0,\] which confirms that $I_*(W)_\nu(x)$ is in the $m$-eigenspace of $\mu(\Sigma_2)$.
\end{proof}
A similar result  holds if $(Y_1,\alpha_1)$ is the disjoint union of two admissible pairs \[(Y_1,\alpha_1) = (Y_1^a,\alpha_1^a)\sqcup(Y_1^b,\alpha_1^b).\] In this case, $(W,\nu)$ induces a map \begin{equation*}\label{eqn:mapthree}I_*(W)_\nu:I_*(Y_1^a)_{\alpha_1^a}\otimes I_*(Y_1^b)_{\alpha_1^b}\to I_*(Y_2)_{\alpha_2}.\end{equation*} Moreover, if \[\Sigma_1^a\sqcup \Sigma_1^b\subset Y_1^a\sqcup Y_1^b\] is homologous in $W$ to $\Sigma_2\subset Y_2$ then \begin{equation}\label{eqn:commutethree}\mu(\Sigma_2)\big(I_*(W)_\nu(x\otimes y)\big) = I_*(W)_\nu\big(\mu(\Sigma_1^a)x\otimes y\big)+I_*(W)_\nu\big(x\otimes \mu(\Sigma_1^b)y\big),\end{equation} which implies the following analogue of Lemma \ref{lem:commute}.
\begin{lemma}
\label{lem:commutethree}
Suppose $x\in I_*(Y_1^a)_{\alpha_1^a}$ is in the generalized $m$-eigenspace of $\mu(\Sigma_1^a)$ and $y\in I_*(Y_1^b)_{\alpha_1^b}$ is in the generalized $n$-eigenspace of $\mu(\Sigma_1^b)$. Then $I_*(W)_\nu(x\otimes y)$ is in the generalized $(m+n)$-eigenspace of $\mu(\Sigma_2)$.
\end{lemma}
\begin{proof}
Under the hypotheses of the lemma, there exists an integer $N>0$ such that \[(\mu(\Sigma_1^a)-m)^Nx=(\mu(\Sigma_1^b)-n)^Ny=0.\] It follows easily from the relation \eqref{eqn:commutethree}  that 
\begin{multline*}
\big(\mu(\Sigma_2) - (m+n)\big)^{2N}I_*(W)_{\nu}(x\otimes y)
=\\ \sum_{j=0}^{2N} {2N\choose j} I_*(W)_{\nu}\big((\mu(\Sigma_1^a)-m)^jx\otimes (\mu(\Sigma_1^b)-n)^{2N-j}y\big).
\end{multline*}
Each term in this sum vanishes since either $j\geq N$ or $2N-j\geq N$,  which confirms that $I_*(W)_{\nu}(x\otimes y)$ lies in the $(m+n)$-eigenspace of $\mu(\Sigma_2)$.
\end{proof}

Lemmas \ref{lem:commute} and \ref{lem:commutethree} will be used repeatedly in Section \ref{sec:alex}. They are also used to prove the next proposition and its corollary, which will in turn be important in the proof of Theorem \ref{thm:khi-fibered} in Section \ref{sec:proof}. In particular, Proposition \ref{prop:grading-shift} will be used to constrain the Alexander grading shift of the map $\phi^{SV}_0\circ C$ described in Section \ref{ssec:outline}.

Suppose for the proposition below that $(W,\nu)$ is a cobordism from $(Y_1,\alpha_1)$ to $(Y_2,\alpha_2)$ and  that $R_1\subset Y_1$ and $R_2\subset Y_2$ are closed surfaces of the same positive genus which are homologous in $W$ with $\alpha_1\cdot R_1$ and $\alpha_2\cdot R_2$ odd. Then Lemma \ref{lem:commute} implies that $I_*(W)_\nu$ restricts to a map \[I_*(W)_\nu:I_*(Y_1|R_1)_{\alpha_1}\to I_*(Y_2|R_2)_{\alpha_2}.\] 

\begin{proposition}
\label{prop:grading-shift}
Suppose $\Sigma_1\subset Y_1$ and $\Sigma_2 \subset Y_2$ are closed surfaces
and  $F \subset W$ is a closed surface of genus $k \geq 1$ and self-intersection $0$ such that
\[ \Sigma_1 + F = \Sigma_2 \]
in $H_2(W)$. If  $x \in I_*(Y_1|R_1)_{\alpha_1}$ belongs to the generalized $2m$-eigenspace of $\mu(\Sigma_1)$, then we can write \[ I_*(W)_\nu(x) = y_{2m-2k+2} + y_{2m-2k+4} + \dots + y_{2m+2k-2}, \]
where each $y_\lambda$ lies in the generalized $\lambda$-eigenspace of the action of $\mu(\Sigma_2)$ on $I_*(Y_2|R_2)_{\alpha_2}$.
\end{proposition}

\begin{proof} First, suppose $\nu\cdot F$ is odd. Consider the cobordism  \[(\overline{W},\overline{\nu}):(Y_1,\alpha_1)\sqcup (F\times S^1,\alpha_F)\to(Y_2,\alpha_2)\] obtained from $W$ by removing a tubular neighborhood $F\times D^2$ of $F$. We may assume that $\alpha_F=\nu\cap (F\times S^1)$ intersects a fiber $F$ transversely in an odd number of points. Then  for $x \in I_*(Y_1|R_1)_{\alpha_1}$ we have  \[I_*(W)_{\nu}(x) = I_*(\overline{W})_{\overline{\nu}}(x\otimes \psi),\] where  $I_*(\overline{W})_{\overline{\nu}}$ is the cobordism map \[I_*(\overline{W})_{\overline{\nu}}:I_*(Y_1|R_1)_{\alpha_1}\otimes I_*(F\times S^1)_{\alpha_F}\to I_*(Y_2|R_2)_{\alpha_2}\] and  $\psi$ is the relative invariant of the $4$-manifold $(F\times D^2,\nu\cap (F\times D^2))$.
From the discussion above, we can write \[\psi=\psi_-+\psi_+,\] where $\psi_{\pm}$ is in the $\pm2$-eigenspace of the operator $\mu(\pt)$ on $I_*(F\times S^1)_{\alpha_F}$.   Recall that an element $x \in I_*(Y_1|R_1)_{\alpha_1}$ lies  in the $2$-eigenspace of $\mu(\pt)$ on  $I_*(Y_1)_{\alpha_1}$ by definition. Since a point in either $Y_1$ or $F\times S^1$ is homologous to a point in $Y_2$, Lemma \ref{lem:commutethree} implies that $I_*(\overline{W})_{\overline{\nu}}(x\otimes \psi_-)$ lies in both the $(+2)$- and $(-2)$-eigenspaces of $\mu(\pt)$ on $I_*(Y_2)_{\alpha_2}$. Thus, \[I_*(\overline{W})_{\overline{\nu}}(x\otimes \psi_-)=0.\]We therefore have that \[I_*(W)_{\nu}(x) = I_*(\overline{W})_{\overline{\nu}}(x\otimes \psi_+).\] From the discussion above, we can  write $\psi_+$ as a sum \[\psi_+ = \psi_{2-2k} + \psi_{4-2k}+\dots +\psi_{2k-4} + \psi_{2k-2},\] where each $\psi_\lambda$ is in the $\lambda$-eigenspace of the operator $\mu(F)$ on $I_*(F\times S^1)_{\alpha_F}.$ 
 
 Suppose $x \in I_*(Y_1|R_1)_{\alpha_1}$ belongs to the $2m$-eigenspace of the operator $\mu(\Sigma_1)$ as in the proposition. It then follows from Lemma \ref{lem:commutethree} that  \[I_*(\overline{W})_{\overline{\nu}}(x\otimes \psi_\lambda)\] lies in the $(2m+\lambda)$-eigenspace of $\mu(\Sigma_2)$ for each $\lambda$. We may therefore write \[I_*(\overline{W})_{\overline{\nu}}(x\otimes \psi)=y_{2m-2k+2} + y_{2m-2k+4} + \dots + y_{2m+2k-2},\] where \[y_\lambda := I_*(\overline{W})_{\overline{\nu}}(x\otimes \psi_{\lambda-2m})\] is in the $\lambda$-eigenspace of $\mu(\Sigma_2)$.

Now suppose  $\nu\cdot F$ is even. We claim that  there is a surface $G\subset W$ homologous to $2F$ of genus $2k-1$. Let $F\times D^2$ be a tubular neighborhood of $F$ in $W$. Let $F'$ be a parallel copy of $F$ in $F\times D^2$. We cut $F$ open along a non-separating curve $c$, cut $F'$ open along a parallel curve $c'$, and glue these cut open surfaces together in a way that is consistent with their orientations and results in a connected surface $G$ of genus $2k-1$. Figure \ref{fig:cut} shows how we modify $F\sqcup F'$ in $A\times D^2$ to obtain $G$, where $A$ is an annular neighborhood of $c$ in $F$.

\begin{figure}[ht]
\labellist

 \tiny\hair 2pt
\pinlabel $p$ at 18 54
\pinlabel $p'$ at 20 84

\endlabellist
\centering
\includegraphics[width=8.2cm]{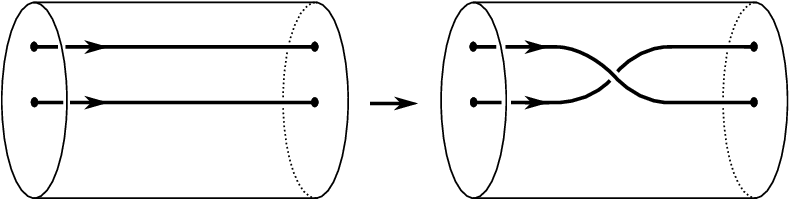}
\caption{A schematic for the modification of $F\sqcup F'$ in $A\times D^2$ to obtain $G$. Left, $A\times D^2$ is represented by $I\times D^2$ while $A\subset F$ and its parallel copy $A'\subset F'$ are represented by the horizontal segments $I\times\{p\}$ and $I\times \{p'\}$. Taking the product of these pictures with $S^1$ is a local model for the actual modification.}
\label{fig:cut}
\end{figure}

By tubing $G$ to a copy of $R_2$, we obtain a closed surface $F'\subset W$  homologous to $2F-R_2$ of genus $2k-1+r$, where $r=g(R_2)$. This surface  has $\nu\cdot F'$ odd and self-intersection $0$, and we have the relation \[2\Sigma_1+F'=2\Sigma_2-R_2\] in $H_2(W)$.  Now  suppose that  $x \in I_*(Y_1|R_1)_{\alpha_1}$ belongs to the $2m$-eigenspace of $\mu(\Sigma_1)$. Then $x$ belongs to the $4m$-eigenspace of $\mu(2\Sigma_1)$. The argument in the previous case tells us that we can write \[ I_*(W)_\nu(x) = z_{4m-2(2k-1+r)+2} + z_{4m-2(2k-1+r)+4} + \dots + z_{4m+2(2k-1+r)-2}, \] where each $z_\lambda$ lies in the $\lambda$-eigenspace of the action of $\mu(2\Sigma_2-R_2)$ on $I_*(Y_2|R_2)_{\alpha_2}$. Then
\[ \big(2\mu(\Sigma_2) - (\lambda+2r-2)\big)^nz_\lambda = \sum_{j=0}^n {n\choose j} \big(\mu(2\Sigma_2-R_2)-\lambda\big)^j\big(\mu(R_2) - (2r-2)\big)^{n-j}z_\lambda, \]
and the right side is again zero for $n$ large enough, meaning that $z_\lambda$ is in the $((\lambda+2r-2)/2)$-eigenspace of $\mu(\Sigma_2)$. Since \[4m-2(2k-1+r)+2\leq \lambda\leq 4m+2(2k-1+r)-2,\] we have that \[2m-2k+1\leq(\lambda+2r-2)/2\leq 2m+2k+2r-3.\] Since the eigenvalues of $\mu(\Sigma_2)$  must also be even integers, we see that the minimum eigenvalue of $\mu(\Sigma_2)$ showing up in the expansion of $I_*(W)_\nu(x)$ into eigenvectors of $\mu(\Sigma_2)$ is $2m-2k+2$. Applying the same argument but for a surface $F'$ homologous to $2F+R_2$ of genus $2k-1+r$ and satisfying\[2\Sigma_1+F'=2\Sigma_2+R_2\] shows that the maximum eigenvalue of $\mu(\Sigma_2)$ showing up in the expansion of $I_*(W)_\nu(x)$ is $2m+2k-2$. This proves the result.
\end{proof}

For the corollary below, suppose  $Y,\alpha,R$ are such that $I_*(Y|R)_\alpha$ is defined.

\begin{corollary} \label{cor:grading-shift}
Suppose $\Sigma_1,\Sigma_2 \subset Y$ are closed surfaces of  the same genus $g\geq 1$ and  $F \subset Y$ is a closed surface of genus $k \geq 1$ such that
\[ \Sigma_1 + F = \Sigma_2 \]
in $H_2(Y)$.  If  $x \in I_*(Y|R)_\alpha$ belongs to the generalized $2m$-eigenspace of $\mu(\Sigma_1)$, then we can write \[ x = x_{2m-2k+2} + x_{2m-2k+4} + \dots + x_{2m+2k-2}, \]
where each $x_\lambda$ lies in the generalized $\lambda$-eigenspace of the action of $\mu(\Sigma_2)$ on $I_*(Y|R)_\alpha$.
\end{corollary}

\begin{proof}
Simply apply Proposition \ref{prop:grading-shift}  to the product cobordism from $(Y,\alpha)$ to itself.
\end{proof} 

We will make repeated use of the surgery exact triangle in instanton Floer homology. This triangle goes back to Floer but appears in the form below in work of Scaduto \cite{scaduto}.

Suppose $K$ is a framed knot in $Y$. Let $\alpha$ be an oriented $1$-manifold in $Y$ which is disjoint from $K$. Let $Y_i$  denote the result of $i$-surgery on $K$, and let $\alpha_i$ be the induced $1$-manifold in $Y_i$, for $i=0,1$.  There is then a sequence of $2$-handle cobordisms
\begin{equation} \label{eq:triangle-cobordisms}
Y\xrightarrow{\,W\,}Y_0\xrightarrow{\,W_0\,}Y_1\xrightarrow{\,W_1\,}Y,
\end{equation}
which induce homomorphisms on instanton Floer homology that fit into an exact triangle as follows.

\begin{theorem}
\label{thm:exacttri}
There is an exact triangle \[ \xymatrix@C=-5pt@R=30pt{
I_*(Y)_{\alpha+ K} \ar[rr]^{I_*(W)_\kappa}  & &I_*(Y_0)_{\alpha_0} \ar[dl]^{I_*(W_0)_{\kappa_0}} \\
&I_*(Y_1)_{\alpha_1}\ar[ul]^{I_*(W_1)_{\kappa_1}} & \\
} \] as long as $(Y,\alpha\sqcup K),$ $(Y_0,\alpha_0)$, and $(Y_1,\alpha_1)$ are all admissible pairs.  Here each of $\kappa$, $\kappa_0$, and $\kappa_1$ is a properly embedded surface in $W$, $W_0$, or $W_1$ bounded by the indicated 1-manifolds at either end: $\kappa_0$ is the union of a product cylinder from $\alpha_0$ to $\alpha_1$ and a closed surface, and $\kappa_1$ and $\kappa$ are each the union of the analogous product cylinder with a disk bounded by $K$ in the corresponding 2-handle.
\end{theorem}


\begin{figure}[ht]
\labellist

 \tiny\hair 2pt
\pinlabel $K$ at 25 100
\pinlabel $\mu_0$ at 62.5 51
\pinlabel $-1$ at 96.5 51
\pinlabel $0$ at 26.5 51
\pinlabel $0$ at 26.5 17.5
\pinlabel $1$ at 61.5 17.5
\pinlabel $\mu_1$ at 96.5 17.5
\pinlabel $0$ at 59 100
\pinlabel $1$ at 94 100
\pinlabel $0$ at 25 66.5
\pinlabel $0$ at 59.5 66.5
\pinlabel $0$ at 93.5 66.5
\pinlabel $1$ at 25 33
\pinlabel $1$ at 59.5 33
\pinlabel $1$ at 93.5 33

\small
\pinlabel $Y$ at 14 1.5
\pinlabel $Y_0$ at 49 1
\pinlabel $Y_1$ at 83 1

\endlabellist
\centering
\includegraphics[width=4cm]{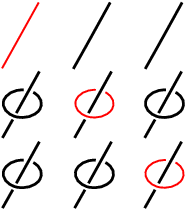}
\caption{Each row describes a slight variation on the surgery exact triangle, in which we realize all three $3$-manifolds by integer surgeries on the red knot.}
\label{fig:surgery}
\end{figure}

Since each manifold in the surgery exact triangle is obtained from the preceding one via integer surgery, we can in fact produce three different surgery exact triangles involving the same 3-manifolds and 4-dimensional cobordisms.  This is illustrated in Figure~\ref{fig:surgery}, where in each row we use a different pair of 3-manifold and framed knot to construct an exact triangle; in the first row, we use $K \subset Y$ to recover Theorem~\ref{thm:exacttri}, while the other two exact triangles correspond to surgeries on $\mu_0 \subset Y_0$ and $\mu_1 \subset Y_1$.  The only differences between these three triangles are the 1-manifolds used to define the instanton Floer groups and the 2-dimensional cobordisms between those 1-manifolds.


\subsection{Sutured instanton homology} \label{ssec:background-shi}
This section provides the necesary background on sutured instanton homology. Our discussion is  adapted from \cite{bs-naturality,bs-shi}, though  the basic construction of $\SHI$ is  of course due to Kronheimer and Mrowka \cite{km-excision}.
\subsubsection{Closures of sutured manifolds} We first recall the following definition.
\begin{definition}A \emph{balanced sutured manifold}  $(M,\Gamma)$ is a compact, oriented 3-manifold $M$ together with an oriented multicurve $\Gamma\subset \partial M$ whose components are called \emph{sutures}. Letting \[R(\Gamma) = \partial M\smallsetminus\Gamma,\] oriented as a subsurface of $\partial M$, it is required that:
\begin{itemize}
\item neither $M$ nor $R(\Gamma)$ has  closed components,
\item $R(\Gamma) = R_+(\Gamma)\sqcup R_-(\Gamma)$ with $\partial R_+(\Gamma) = -\partial R_-(\Gamma) = \Gamma$, and
\item $\chi(R_+(\Gamma)) = \chi(R_-(\Gamma))$.
\end{itemize}
\end{definition}

The following examples will be important for us.
\begin{example}
\label{eg:prodsutured}
Suppose $S$ is a compact, connected, oriented surface with nonempty boundary. The pair \[(H_S,\Gamma_S):=(S\times[-1,1],\partial S\times\{0\})\] is called a \emph{product sutured manifold}.
\end{example}

\begin{example}
\label{eg:knotcomplement}
Given a knot $K$ in a closed, oriented $3$-manifold $Y$, let \[(Y(K),\Gamma_\mu) := (Y\ssm\nu(K),\mu\cup -\mu),\] where $\nu(K)$ is a tubular neighborhood of $K$ and $\mu$ and $-\mu$ are oppositely oriented meridians.
\end{example}


\begin{definition} An \emph{auxiliary surface} for  $(M,\Gamma)$ is a compact, connected, oriented surface $T$ with the same number of boundary components as components of $\Gamma$.\end{definition} Suppose $T$ is an auxiliary surface for $(M,\Gamma)$, that $A(\Gamma)\subset\partial M$ is a  tubular neighborhood of $\Gamma$,  and  that  \[h:\partial T\times[-1,1]\xrightarrow{\cong} A(\Gamma)\] is an orientation-reversing diffeomorphism. 
\begin{definition} \label{def:preclosure}
We form a \emph{preclosure} of $M$ \begin{equation*}\label{eqn:bF}M'=M\cup \big(T\times [-1,1]\big)\end{equation*}  by gluing $T\times[-1,1]$ to $M$ according to $h$.  
\end{definition} 
This preclosure has two diffeomorphic boundary components, \[\partial M' = \partial_+M'\sqcup\partial_-M',\] where \[ \partial_+M':=\big(R_+(\Gamma)\cup T\big)\cong\big(R_-(\Gamma)\cup -T\big)=:\partial_-M'.\] Let $R:=\partial_+M'$ and choose an orientation-reversing diffeomorphism \[\varphi:\partial_+M'\xrightarrow{\cong}\partial_-M'\] which fixes a point $q\in T$.
We   form a closed $3$-manifold \[Y=M'\cup \big(R\times[1,3]\big)\] by gluing $R\times[1,3]$ to $M'$ according to the maps
\begin{align*}
\id&:R\times\{1\}\to \partial_+M',\\
\varphi&:R\times\{3\}\to \partial_-M',
\end{align*}
and we let $r: R\times[1,3] \hookrightarrow Y$ and $m: M \hookrightarrow Y$ be the natural inclusions.

Let $\alpha\subset Y$ be the curve formed as the union of the arcs \[\{q\}\times[-1,1]\textrm{ in } M' \textrm{ and }\{q\}\times [1,3]\textrm{ in } R\times[1,3].\] Choose a nonseparating curve $\eta\subset R\ssm\{q\}$; we note that this requires $g(R) \geq 1$.
For  convenience, we will also use $R$ to denote the \emph{distinguished surface} $R\times\{2\}\subset Y$ and $\eta$ to denote the curve $\eta\times\{2\}\subset R\times\{2\}\subset Y$. 

\begin{definition} We refer to the tuple $\data = (Y,R,r,m,\eta,\alpha)$ as a \emph{closure} of $(M,\Gamma)$.\footnote{In \cite{bs-shi}, we called such a tuple a \emph{marked odd closure}.}
\end{definition}

 The \emph{genus} $g(\data)$ refers to the genus of $R$.


\begin{remark} Suppose $\data = (Y,R,r,m,\eta,\alpha)$ is a closure of $(M,\Gamma)$. Then the tuple \[-\data:=(-Y,-R,r,m,-\eta,-\alpha)\]  is a closure of $-(M,\Gamma):=(-M,-\Gamma).$   \end{remark}

\subsubsection{Sutured instanton homology}

\label{ssec:shi}

Following Kronheimer and Mrowka \cite{km-excision}, we make the definition below.

\begin{definition}
Given a closure $\data = (Y,R,r,m,\eta,\alpha)$ of $(M,\Gamma)$, the \emph{sutured instanton homology}\footnote{In \cite{bs-naturality}, we called this the \emph{twisted sutured instanton homology} and denoted it by $\underline{SHI}(\data)$ instead.} of $\data$ is the $\C$-module $\SHI(\data) = I_*(Y|R)_{\alpha+ \eta}.$
\end{definition}

Kronheimer and Mrowka proved that, up to isomorphism, $\SHI(\data)$ is an invariant of $(M,\Gamma)$.  (In fact, they took $I_*(Y|R)_\alpha$ to be their invariant, but proved it isomorphic to $I_*(Y|R)_{\alpha+\eta}$ via excision.)  In \cite{bs-naturality}, we constructed for any two closures $\data,\data'$ of $(M,\Gamma)$ of genus at least two a \emph{canonical} isomorphism \[\Psi_{\data,\data'}:\SHI(\data)\to\SHI(\data')\] which is well-defined up to multiplication in $\C^\times$.  (See the proof of Theorem~\ref{thm:alexwelldefined} for details of the construction.)  In particular, these isomorphisms satisfy, up to multiplication in $\C^\times$,  \[\Psi_{\data,\data''} =  \Psi_{\data',\data''}\circ\Psi_{\data,\data'}\] for any triple $\data,\data',\data''$ of such closures. The groups $\SHI(\data)$ and  isomorphisms $\Psi_{\data,\data'}$  ranging over closures of $(M,\Gamma)$ of genus at least two thus define what we called a \emph{projectively transitive system of $\C$-modules} in \cite{bs-naturality}.
\begin{definition} The \emph{sutured instanton homology of $(M,\Gamma)$} is the projectively transitive system of $\C$-modules  $\SHI(M,\Gamma)$ defined by the groups and canonical isomorphisms above. \end{definition} 

\begin{remark}
The isomorphisms $\Psi_{\data,\data'}$ are defined using $2$-handle and excision cobordisms. We will provide more details in Section \ref{sec:alex}, where we show that these isomorphisms respect the Alexander gradings associated to certain properly embedded surfaces in $(M,\Gamma)$.
\end{remark}

\begin{remark} \label{rem:canonical-genus-one}
We emphasize that if either $\data$ or $\data'$ has genus one, then there is still an isomorphism $\SHI(\data) \xrightarrow{\sim} \SHI(\data')$, though it may not be canonical.  We will show in Theorem~\ref{thm:alexwelldefined} that even these non-canonical isomorphisms respect the Alexander gradings, which will be necessary to relate our Alexander grading to one defined by Kronheimer and Mrowka (see Remarks~\ref{rem:khi-genus-one-closure} and \ref{rmk:symmetryknots}).
\end{remark}

The following result of Kronheimer and Mrowka \cite[Proposition 7.8]{km-excision} will be important for us. We sketch their proof below so that we can refer to this construction later.

\begin{proposition}
\label{prop:prodsutured}
Suppose $(H_S,\Gamma_S)$ is a product sutured manifold as in Example \ref{eg:prodsutured}. Then $\SHI(H_S,\Gamma_S)\cong\C.$
\end{proposition}

\begin{proof}
Let $T$ be an auxiliary surface for $(H_S,\Gamma_S)$. Form a preclosure  by gluing $T\times[-1,1]$ to $S\times[-1,1]$ according to a map \[h:\partial T\times[-1,1]\to\partial S\times[-1,1]\] of the form $f\times \id$ for some diffeomorphism $f:\partial T\to \partial S$. This preclosure is then a product \[M'=(S\cup T)\times[-1,1].\] To form a closure, we let $R=S\cup T$ and glue $R\times[1,3]$ to $M'$ by   maps 
\begin{align*}
\id&:R\times\{1\}\to \partial_+M',\\
\varphi&:R\times\{3\}\to \partial_-M'
\end{align*}
as usual, where $\varphi$ fixes a point $q$ on $T$. We then define $\eta$ and $\alpha$ as described earlier to obtain a closure $\data_S = (Y,R,r,m,\eta,\alpha)$, where \[Y = R\times_{\varphi}S^1\] is the surface bundle over the circle with fiber $R$ and monodromy $\varphi$. We then have that \[\SHI(\data_S)=I_*(R\times_\varphi S^1|R)_{\alpha+\eta}\cong\C,\] by \cite[Proposition 7.8]{km-excision} in the case $\varphi=\id$, and more generally by excision \cite[Theorem~7.7]{km-excision} exactly as in the proof of \cite[Lemma~4.7]{km-excision}.
\end{proof}

As mentioned in the introduction, Kronheimer and Mrowka define in \cite[Section 7.6]{km-excision} the instanton Floer homology of a knot as follows.

\begin{definition}
\label{def:khi}
Suppose $K$ is a knot in a closed, oriented $3$-manifold $Y$. The \emph{instanton knot Floer homology} of $K$ is given by \[\KHI(Y,K):=\SHI(Y(K),\Gamma_\mu),\] where $(Y(K),\Gamma_\mu)$ is the knot complement with two meridional sutures as in Example \ref{eg:knotcomplement}.
\end{definition}

\subsubsection{Contact handle attachments and surgery}
\label{sssec:handles}

In \cite{bs-shi}, we defined maps on $\SHI$ associated to contact handle attachments and surgeries, which we will review below; see Figure \ref{fig:handles} for illustrations of such handle attachments.  These contact handle attachments can be viewed purely as operations on sutured manifolds, and this is the perspective we take in this subsection.  However, we will first motivate the name by summarizing the contact geometric origin of this construction, which we will return to in Section \ref{ssec:contact}.

A collection of sutures $\Gamma \subset \partial M$ specifies an $I$-invariant contact structure in a collar neighborhood of $\partial M$, for which $\partial M$ is convex with dividing set $\Gamma$; see \cite{giroux-convexite} for details on convex surfaces.  We can then naturally glue other contact manifolds with convex boundaries to $M$ by orientation-reversing diffeomorphisms which identify the dividing sets.  For example, if we view $D^2$ as the unit disk in the plane and let $I=[-\epsilon,\epsilon]$, then $D^2_{(x,y)} \times I_z$ has a tight contact form $\alpha = dz-ydx$ which makes the boundary convex.  Its dividing set is disconnected, consisting of the four arcs
\[ \big(\{y = 0\} \times \partial I \big) \sqcup \big(\{(0,\pm1)\} \times I\big), \]
lying on $D^2 \times \partial I$ and $\partial D^2 \times I$ respectively.  Attaching a contact 1-handle means gluing this $D^2\times I$ to $M$ along $D^2 \times \partial I$ by a map taking dividing curves to sutures, whereas a contact 2-handle is instead glued along $\partial D^2 \times I$.  In either case we get a manifold with corners, and when we smooth them out, Honda's ``edge-rounding'' procedure \cite[Lemma~3.11]{honda-lens} tells us how to connect the various sutures/dividing arcs on the boundary, exactly as shown in Figure~\ref{fig:handles}.




From here on we will suppress the inclusion maps $r$ and $m$ from the notation specifying a closure $\data$, since any time we want to use $\data$ to construct some new closure $\data'$ there will be a natural, essentially unique way to define the corresponding $r'$ and $m'$.

\begin{figure}[ht]
\labellist

 \hair 2pt
\pinlabel $c$ at 300 277

\endlabellist
\centering
\includegraphics[width=7cm]{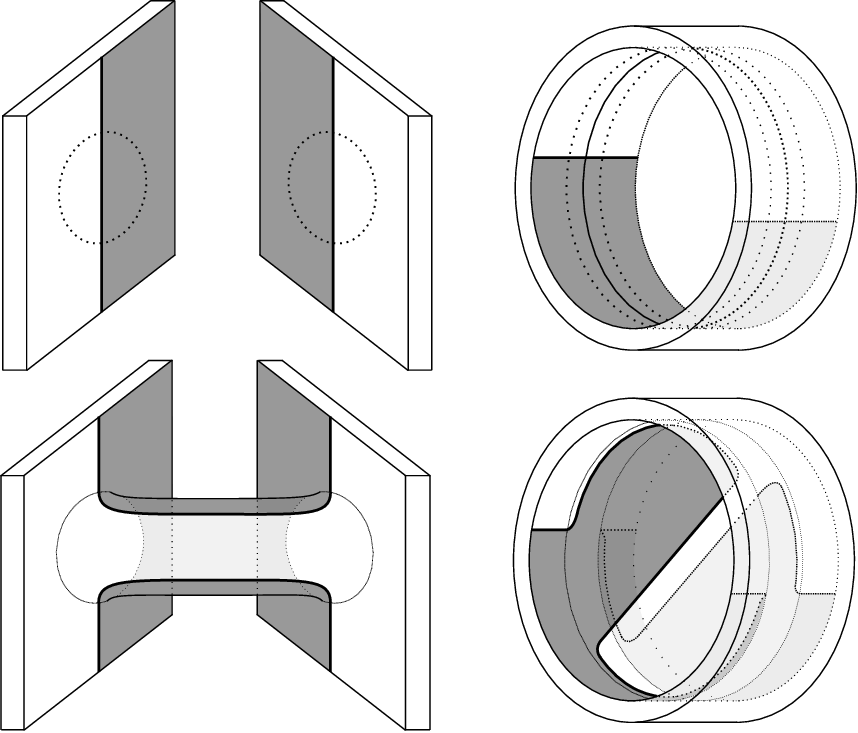}
\caption{Left, a $1$-handle attachment. Right, a $2$-handle attachment.  Recall that a $2$-handle is attached along an annular neighborhood of a curve $c$ which intersects $\Gamma$ in exactly two points, as shown. The gray regions represent $R_-(\Gamma)$ and the white regions $R_+(\Gamma)$, and appear as shown \emph{after} edge-rounding. }
\label{fig:handles}
\end{figure}

First, we consider contact $1$-handle attachments.

Suppose $(M',\Gamma')$ is obtained from $(M,\Gamma)$ by attaching a contact $1$-handle. Then one can show that any closure $\data' = (Y,R,\eta,\alpha)$ of the first also naturally specifies a closure $\data = (Y,R,\eta,\alpha)$ of the second (the only difference between these closures are the embeddings $M,M'\hookrightarrow Y$). We  therefore define the $1$-handle attachment map  \[\id:\SHI(-\data)\to\SHI(-\data')\] to be  the identity map on instanton Floer homology. For closures   of genus at least two, these maps  commute with the canonical isomorphisms described above and thus define  a  map  \[H_1:\SHI(-M,-\Gamma)\to\SHI(-M',-\Gamma'),\] as   in \cite[Section 3.2]{bs-shi}.

Next, we consider contact $2$-handle attachments.  

Suppose $(M',\Gamma')$ is obtained from $(M,\Gamma)$ by attaching a contact $2$-handle along a curve $c\subset \partial M$. Let $\data = (Y,R,\eta,\alpha)$ be a closure of $(M,\Gamma)$. We proved  in \cite[Section 3.3]{bs-shi} that $\partial M$-framed surgery on $c\subset Y$ naturally yields a closure $\data'=(Y',R,\eta,\alpha)$ of $(M',\Gamma')$, where $Y'$ is the surgered manifold.  Let \[W:-Y\to-Y'\] be the cobordism obtained from $-Y\times[0,1]$ by attaching a $2$-handle along $c\times\{1\}\subset -Y\times\{1\}$, and let $\nu$ be the cylinder \[\nu=(-\alpha\sqcup-\eta)\times [0,1].\] The fact that  $c$ is disjoint from $R$ means that  $-R\subset-Y$ and $-R\subset-Y'$ are isotopic in $W$.
Since these surfaces have the same genus, Lemma \ref{lem:commute} implies that the induced map \[I_*(W)_\nu:I_*(-Y)_{-\alpha-\eta}\to I_*(-Y')_{-\alpha-\eta}\]
restricts to a map \[I_*(W)_\nu:\SHI(-\data)\to\SHI(-\data').\]
For closures of genus at least two, these maps commute with the canonical isomorphisms and therefore define a map \[H_2:\SHI(-M,-\Gamma)\to\SHI(-M',-\Gamma'),\] as shown in  \cite[Section 3.3]{bs-shi}.

Finally, we consider surgeries.

Suppose $(M',\Gamma')$  is obtained from $(M,\Gamma)$ via  $(+1)$-surgery on a framed knot $K\subset M$. A closure $\data = (Y,R,\eta,\alpha)$ of $(M,\Gamma)$ naturally gives rise to a closure $\data'=(Y',R,\eta,\alpha)$ of $(M',\Gamma')$, where $Y'$ is the surgered manifold as in the $2$-handle attachment case. The $2$-handle cobordism $(W,\nu)$ corresponding to this surgery induces a map
\begin{equation} \label{eq:2-handle-surgery-nu-map}
I_*(W)_\nu:\SHI(-\data)\to\SHI(-\data')
\end{equation}
as in the previous case. We showed in \cite[Section 3.3]{bs-shi} that for closures of genus at least two, these maps commute with the canonical isomorphisms and thus give rise to a map \begin{equation*}\label{eqn:leg-surg}F_K:\SHI(-M,-\Gamma)\to\SHI(-M',-\Gamma').\end{equation*}

\begin{remark}
We have asserted that all of the handle-attaching and surgery maps commute with the canonical isomorphisms $\Psi_{\data,\data'}$ when the closures $\data$ and $\data'$ have genus at least two, but this restriction can be weakened: if either closure has genus one, then these maps still commute with the non-canonical isomorphisms mentioned in Remark~\ref{rem:canonical-genus-one}.
\end{remark}

\subsection{The contact invariant} \label{ssec:contact}

We assume some familiarity with contact structures and open books; see Etnyre \cite{etnyre-intro} or Geiges \cite{geiges} for a general introduction, Etnyre \cite{etnyre-ob} for material on open books, and Massot \cite{massot} for convex surfaces.  One of the key things we will take from convex surface theory is Giroux flexibility \cite{giroux-convexite}, which essentially says that given a surface $\Sigma$ and multicurve $\Gamma$ bounding a subsurface $R_+ \subset \Sigma$, there is a unique (up to isotopy) vertically invariant contact structure on $\Sigma \times [0,1]$ such that $\Sigma \times \{t\}$ is convex with dividing set $\Gamma$.  In particular, this allows us to glue contact structures along convex surfaces with the same dividing sets.

The background material on partial open books is introduced in large part to establish common notation. In what follows, we write $(M,\Gamma,\xi)$ to refer to a contact manifold $(M,\xi)$ with nonempty convex boundary and dividing set $\Gamma$; we call such a triple a \emph{sutured contact manifold}.

 \begin{definition}
 \label{def:pob} A \emph{partial open book} is a triple $(S,P,h)$, where: 
\begin{itemize}
\item $S$ is a connected, oriented surface with nonempty boundary, 
\item $P$ is a subsurface of $S$ formed as the union of a neighborhood of $\partial S$ with $1$-handles in $S$, and
\item $h:P\to S$ is an embedding which restricts to the identity on $\partial P\cap \partial S$.
\end{itemize}
\end{definition}

\begin{definition}
A \emph{basis} for a partial open book $(S,P,h)$ is a collection \[\mathbf{c}=\{c_1,\dots,c_n\}\] of disjoint, properly embedded arcs in $P$ such that $S\ssm \mathbf{c}$ deformation retracts onto $S\ssm P$; essentially, the \emph{basis arcs} in $\mathbf{c}$ specify the cores of $1$-handles used to form $P$.
\end{definition}
A partial open book specifies a sutured contact manifold as follows.

Suppose $(S,P,h)$ is a partial open book with basis $\mathbf{c}=\{c_1,\dots,c_n\}$. Let $\xi_S$ be the unique tight contact structure on the handlebody \[H_S=S\times[-1,1] \textrm{ with dividing set }\Gamma_S = \partial S\times\{0\}.\] 
For $i=1,\dots,n$, let $\gamma_i$ be the curve on $\partial H_S$ given by   \begin{equation}\label{eqn:basishandle}\gamma_i=(c_i\times\{1\})\cup (\partial c_i\times [-1,1])\cup (h(c_i)\times\{-1\}).\end{equation} 
\begin{definition} We define $M(S,P,h,\mathbf{c})$ to be the sutured contact manifold obtained from $(H_S,\Gamma_S,\xi_S)$ by attaching contact $2$-handles along the curves $\gamma_1,\dots,\gamma_n$ above.
\end{definition}

\begin{remark}
Up to a canonical isotopy class of contactomorphisms, $M(S,P,h,\mathbf{c})$ does not depend on the choice of basis $\mathbf{c}$.  It does depend on $h$, however, via the $2$-handles attached along the $\gamma_i$, which can be understood as gluing each arc $c_i \times \{1\}$ to its image $h(c_i) \times \{-1\}$.
\end{remark}

\begin{definition}
\label{def:pobd}
A \emph{partial open book decomposition} of $(M,\Gamma,\xi)$ consists of a partial open book $(S,P,h)$ together with a contactomorphism  \[M(S,P,h,\mathbf{c})\xrightarrow{\cong}(M,\Gamma,\xi)\] for some basis $\mathbf{c}$ of $(S,P,h)$.
\end{definition}
\begin{remark} We will generally conflate partial open book decompositions with partial open books. Given a partial open book decomposition  as above, we will simply think of $(M,\Gamma,\xi)$ as being equal to $M(S,P,h,\mathbf{c})$. In particular, we will view $(M,\Gamma,\xi)$ as obtained from $(H_S,\Gamma_S,\xi_S)$ by attaching contact $2$-handles along the curves $\gamma_1,\dots,\gamma_n$ in \eqref{eqn:basishandle}.
\end{remark}

Half of the \emph{relative Giroux correspondence}, proven in \cite{hkm-sutured},   states the following.
\begin{theorem}
Every $(M,\Gamma,\xi)$ admits a partial open book decomposition.
\end{theorem}

\begin{example}
Given an open book decomposition $(\Sigma, h)$ of a closed contact manifold $(Y,\xi)$, the triple \[(S=\Sigma, P=\Sigma\ssm D^2, h|_P)\] is a partial open book decomposition for the complement of a Darboux ball in $(Y,\xi)$.
\end{example}

\begin{example}
\label{eq:legendrian-complement}
Given an open book decomposition $(\Sigma, h)$ of a closed contact manifold $(Y,\xi)$ with a Legendrian knot $K$ realized on the page $\Sigma$, the triple \[(S=\Sigma, P=\Sigma\ssm\nu(K), h|_P)\] is a partial open book decomposition for the complement of a standard neighborhood of $K$.
\end{example}

To define the contact invariant of  $(M,\Gamma,\xi)$, we choose a partial open book decomposition $(S,P,h)$ of $(M,\Gamma,\xi)$. Let $\mathbf{c}=\{c_1,\dots,c_n\}$ be a basis for $P$ and let \[\gamma_1,\dots,\gamma_n\subset \partial H_S\] be the corresponding curves as in \eqref{eqn:basishandle}. Let \[H:\SHI(-H_S,-\Gamma_S)\to\SHI(-M,-\Gamma)\] be the composition of the maps associated to contact $2$-handle attachments along the curves $\gamma_1,\dots,\gamma_n$, as  described in Section \ref{sssec:handles}. Recall from Proposition \ref{prop:prodsutured} that \[\SHI(-H_S,-\Gamma_S)\cong\C,\] and let $\mathbf{1}$ be a generator of this group. The following is from  \cite[Definition 4.2]{bs-shi}.
\begin{definition}
\label{def:contact} We define the contact class to be
 \[\cinvt(M,\Gamma,\xi):= H(\mathbf{1})\in\SHI(-M,-\Gamma).\] 
\end{definition} As the notation suggests, this class does not depend on the partial open book decomposition of $(M,\Gamma,\xi)$.  Indeed, we proved the following in \cite[Theorem 4.3]{bs-shi}.

\begin{theorem}
$\cinvt(M,\Gamma,\xi)$ is an invariant of the sutured contact manifold $(M,\Gamma,\xi)$. 
\end{theorem}

We will  often think about the contact class in terms of closures. For that point of view, let $\data_S = (Y,R,\eta,\alpha)$ be a closure of $(H_S,\Gamma_S)$ with \[Y=R\times_\varphi S^1\] as  in the proof of Proposition \ref{prop:prodsutured}.  As mentioned in the definition of the contact $2$-handle attachment maps, performing $\partial H_S$-framed surgery on the curves \[\gamma_1,\dots,\gamma_n\subset \partial H_S\subset Y\] naturally yields a closure $\data=(Y',R,\eta,\alpha)$ of $(M,\Gamma)$. We then define \[\cinvt(M,\Gamma,\xi,\data):=I_*(V)_\nu(\mathbf{1}),\] where \[I_*(V)_\nu:I_*(-Y|{-}R)_{-\alpha-\eta}\to I_*(-Y'|{-}R)_{-\alpha-\eta},\] is the map associated to the corresponding $2$-handle cobordism $V$, $\nu=(-\alpha\sqcup-\eta)\times[0,1]$ is the usual cylindrical cobordism, and $\mathbf{1}$ is a generator of \[I_*(-Y|{-}R)_{-\alpha-\eta}\cong \C.\] In particular, this class is well-defined up to multiplication in $\C^\times$. Our invariance result is then the statement that the classes $\cinvt(M,\Gamma,\xi,\data)$ defined in this manner, for any  partial open book decompositions and closures of genus at least two, are related by the canonical isomorphisms between the  groups assigned to  different closures. 

\begin{remark}
For convenience of notation, we will often use the shorthand $\cinvt(\xi)$ and $\cinvt(\xi,\data)$ to denote the classes $\cinvt(M,\Gamma,\xi)$ and  $\cinvt(M,\Gamma,\xi,\data)$, respectively.
\end{remark}

Below are some important properties of the contact class, all proven in \cite[Section 4.2]{bs-shi}. In brief, the contact class vanishes for overtwisted contact manifolds and behaves naturally with respect to   contact handle attachment and contact $(+1)$-surgery on Legendrian knots.  We view both contact 1-handles and contact 2-handles as having the unique $I$-invariant contact structure on $D^2 \times I$ whose dividing set on $D^2$ is a properly embedded arc, as described at the beginning of Section~\ref{sssec:handles}.

\begin{theorem}
\label{thm:zero-overtwisted}
If $(M,\Gamma,\xi)$ is overtwisted then $\cinvt(\xi)=0$.
\end{theorem}

\begin{theorem}
\label{thm:handletheta2} Suppose $(M',\Gamma',\xi')$  is the result of attaching a contact $i$-handle to $(M,\Gamma,\xi)$ for $i=1$ or $2$. Then the associated map  \[H_i:\SHI(-M,-\Gamma)\to\SHI(-M',-\Gamma')\] defined in Section \ref{sssec:handles} sends $\cinvt(\xi)$ to $\cinvt(\xi')$.
\end{theorem}

\begin{theorem}
\label{thm:legendrian-surgery}
Suppose $K$ is a Legendrian knot in $(M,\Gamma,\xi)$ and  that $(M',\Gamma',\xi')$ is the result of contact $(+1)$-surgery on $K$. Then the associated map \[F_K:\SHI(-M,-\Gamma)\to\SHI(-M',-\Gamma')\] defined in Section \ref{sssec:handles} sends $\cinvt(\xi)$ to $\cinvt(\xi')$.
\end{theorem}

For the sake of later applications, we now observe that the map $F_K$ fits into a surgery exact triangle, where the third term involves the Floer homology of the sutured manifold obtained by $0$-surgery on $K$. 
%
%
Suppose $(M',\Gamma')$ is obtained from $(M,\Gamma)$ via $(+1)$-surgery on a framed knot $K\subset M$.  Let $\data = (Z,R,\eta,\alpha)$ be a closure of $(M,\Gamma)$, and let \[\data_1=(Z_1,R,\eta,\alpha)\textrm{ and }\data_0=(Z_0,R,\eta,\alpha)\]
be the tuples obtained from $\data$ by performing $1$- and $0$-surgery on $K\subset Z$.
These  are naturally closures of the sutured manifolds 
\[(M,\Gamma)\textrm{ and }
(M_1(K),\Gamma)=(M',\Gamma')\textrm{ and }
(M_0(K),\Gamma).
\] Note that $-Z_1$ and $-Z_0$ are obtained from surgeries on  $ K\subset -Z$, 
\begin{align*}
-Z_1&\cong (-Z)_{-1}(K),\\
-Z_0&\cong (-Z)_{0}(K).
\end{align*}
By Theorem \ref{thm:exacttri} (and Lemma \ref{lem:commute}), there is an exact triangle 
\begin{equation}\label{eqn:surgeryexacttriangleZ} \xymatrix@C=-35pt@R=35pt{
I_*(-Z|{-}R)_{-\alpha -\eta} \ar[rr]^{I_*(W)_{{\kappa}}}  & &I_*(-Z_1|{-}R)_{-\alpha -\eta} \ar[dl]^{I_*(W_{0})_{{\kappa}_0}} \\
&I_*(-Z_0|{-}R)_{-\alpha -\eta+ \mu}\ar[ul]^{I_*(W_{1})_{{\kappa}_1}}. & \\
} \end{equation}
Here, $\mu$ is the curve in $-Z_0$ corresponding to the meridian of $K\subset -Z$, as in Figure \ref{fig:meridian}; and $\kappa$ is the product $(-\alpha-\eta) \times I$ in $W$, which was called $\nu$ in \eqref{eq:2-handle-surgery-nu-map}, so we have $F_K = I_*(W)_\kappa$.

Strictly speaking, this triangle is one of the variations described after Theorem~\ref{thm:exacttri}, corresponding to the second row of Figure~\ref{fig:surgery}; again, the 3-manifolds and cobordisms between them are the same, but the 1-manifolds and cobordisms between them are different.  This particular exact triangle corresponds to surgery on $\mu \subset -Z_0$, with its usual framing inherited from $Z$, since $0$- and $1$-surgeries on $\mu$ produce $-Z$ and $-Z_1$.


\begin{figure}[ht]
\labellist

 \tiny\hair 2pt
\pinlabel $\mu$ at 26.5 16.5
\pinlabel $0$ at 61 16.5
\pinlabel $1$ at 95 16.5

\pinlabel $0$ at 60.5 34
\pinlabel $0$ at 26 34
\pinlabel $0$ at 94.5 34

\small
\pinlabel $-Z_0$ at 12 1
\pinlabel $-Z$ at 46 1
\pinlabel $-Z_1$ at 81 1

\endlabellist
\centering
\includegraphics[width=4cm]{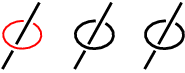}
\caption{The curve $\mu$ in $-Z_0$ on which $0$- and $1$-surgeries produce $-Z$ and $-Z_1$.}
\label{fig:meridian}
\end{figure}
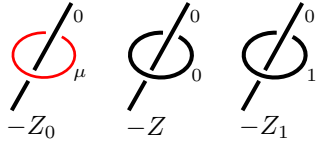

\section{The Alexander grading}
\label{sec:alex}
In \cite[Section 7.6]{km-excision}, Kronheimer and Mrowka explain how a Seifert surface for a knot $K\subset Y$ gives rise to an ``Alexander grading'' on the instanton knot Floer homology $\KHI(Y,K)$ defined in Definition \ref{def:khi}; the name indicates that $\KHI$ together with this grading categorifies the Alexander polynomial of $K$.  In their construction, they use a genus one closure of $(Y(K),\Gamma_\mu)$ formed with an annular auxiliary surface. In this section, we describe how more generally a properly embedded surface in any sutured manifold which intersects the sutures twice gives rise to a grading on $\SHI$, which we also call an Alexander grading, using closures of any genus. In particular, we prove that this grading is preserved by the canonical isomorphisms relating the groups assigned to different closures.

Before beginning the discussion in earnest, we recall why this is necessary.  In Section~\ref{sec:leg} we will define an invariant which associates to a transverse knot $K \subset (Y,\xi)$ an element 
\[ \kinvt(K) \in \KHI(-Y,K). \]
In Section~\ref{sec:proof} we will prove that if $K$ is the connected binding of an open book supporting $\xi$, with fiber surface $\Sigma$, then $\kinvt(K)$ generates the top Alexander grading $\KHI(-Y,K,[\Sigma],g) \cong \C$.  Its grading is determined in Theorem~\ref{thm:ttopgrading}, which we prove by computing for a certain contact structure $\xi_0^-$ the Alexander grading of our contact invariant
\[\cinvt(\xi_0^-) \in \SHI(-Y(K), -\Gamma_0,[\Sigma],g), \]
where the sutures $\Gamma_0$ are \emph{not} meridional, and then showing that $\kinvt(K)$ is the image of this invariant under a map which preserves this grading.  This argument therefore requires a notion of grading for more general choices of sutures.  In fact, even in the case of $\KHI$, we need a more general construction because $\kinvt(K)$ cannot be defined directly using Kronheimer and Mrowka's preferred genus one closures.

Suppose $(M,\Gamma)$ is a balanced sutured manifold and $\Sigma$ is a properly embedded, oriented surface in $M$ with one boundary component such that the closed curve $\sigma=\Sigma\cap \partial M$  intersects $\Gamma$ transversely in two points, $p_+$ and $p_-$, and let \[\sigma_{\pm} = \sigma\cap R_{\pm}(\Gamma),\] as shown in Figure \ref{fig:surfaceclosure}. To  define the Alexander grading on $\SHI(M,\Gamma)$ associated with $\Sigma$, we first cap off $\Sigma$ in a closure of $(M,\Gamma)$.

\begin{figure}[ht]
\labellist

\tiny \hair 2pt
\pinlabel $p_+$ at 94 87
\pinlabel $p_-$ at 204 87
\pinlabel $\sigma_+$ at 68 99
\pinlabel $\sigma_+$ at 230 99
\pinlabel $\sigma_-$ at 148 98
\pinlabel $\Sigma$ at 55 59
\pinlabel $\Sigma'$ at 332 59

\pinlabel $\sigma_+$ at 342 99
\pinlabel $\sigma_-$ at 427 98
\pinlabel $\tau_+$ at 358 143
\pinlabel $\tau_-$ at 387 143

\pinlabel $\tau_-$ at 467 143
\pinlabel $\tau_+$ at 496 143

\endlabellist
\centering
\includegraphics[width=12.3cm]{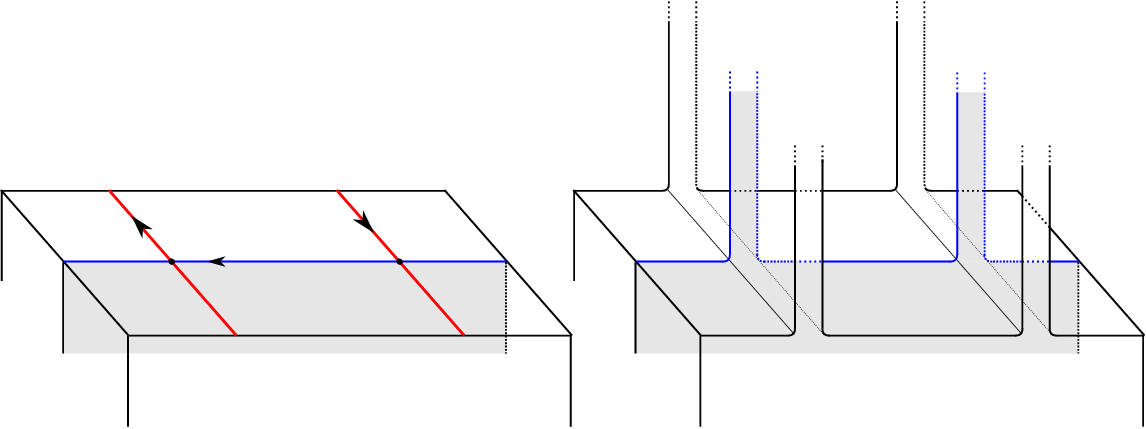}
\caption{Left, a portion of the manifold $M$ with $\Gamma$ shown in red. The surface $\Sigma$ is shown with oriented boundary  in blue. Right, the preclosure $M'$ formed by attaching $T\times[-1,1]$. The $1$-handle $\tau\times[-1,1]$ glues to $\Sigma$ to form the properly embedded surface $\Sigma'$ with boundary shown in blue. }
\label{fig:surfaceclosure}
\end{figure}
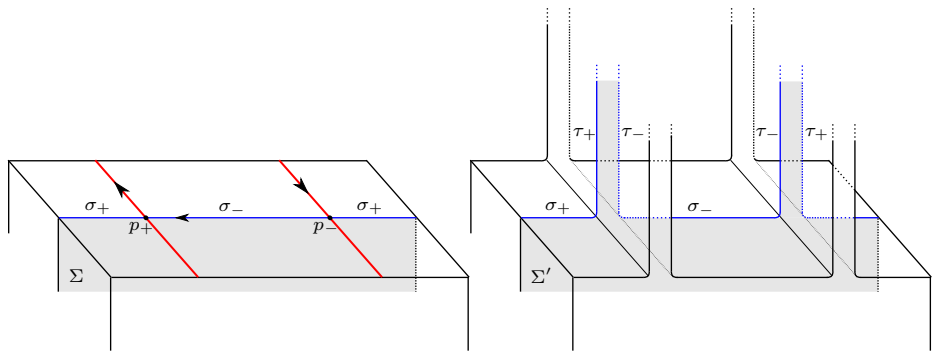

Let $T$ be an auxiliary surface with an identification \[f:\partial T\xrightarrow{\cong}\Gamma;\] we note that $\partial T$ and $\Gamma$ need not be connected.  Let $A(\Gamma)=\Gamma\times[-1,1]$ be an annular neighborhood of $\Gamma$ such that $\sigma$ intersects $A(\Gamma)$ in the arcs  \[\{p_+\}\times[-1,1]\textrm{ and }\{p_-\}\times[-1,1].\] and let \[h:\partial T\times[-1,1]\xrightarrow{\cong} A(\Gamma)\] be the diffeomorphism $h=f\times\id$.  Let $M'$  be the preclosure obtained by gluing $T\times[-1,1]$ to $M$ according to $h$. Let $\tau$ be a nonseparating, properly embedded arc in $T$ with endpoints at $p_+$ and $p_-$. Let $\Sigma'$ be the properly embedded surface in $M'$ obtained as the union of $\Sigma$ with the 1-handle $\tau\times[-1,1]$, as shown in Figure \ref{fig:surfaceclosure}. The boundary of $\Sigma'$ consists of the two circles \[c_\pm:=\sigma_\pm\cup\tau_\pm\subset \partial_{\pm}M',\] where $\tau_\pm = \tau\times\{\pm 1\}$. The fact that $\tau$ is nonseparating in $T$ implies that the circles $c_\pm$ are nonseparating in $\partial_\pm M'$. In forming a closure, we can therefore choose a diffeomorphism \[\varphi:\partial_+M'\to\partial_-M'\]  which identifies $c_+$ with $c_-$. Let $R=\partial_+M'$ and let \[Y= M'\cup \big(R\times[1,3]\big)\] be the closed manifold formed according to $\varphi$ in the manner described in Section~\ref{ssec:background-shi}, immediately after Definition~\ref{def:preclosure}. We define \[\alpha=\big(\{q\}\times[-1,1]\big)\cup\big(\{q\}\times[1,3]\big)\] in the usual way, for some $q\in T$ fixed by $\varphi$, and choose for $\eta\subset R$ a curve which intersects $c_+$ in one point.  (This last condition will be useful at the beginning of the proof of Theorem~\ref{thm:alexwelldefined}.)

\begin{definition}
We say that a closure $\data = (Y,R,\eta,\alpha)$ defined as above is \emph{adapted to $\Sigma$}. 
\end{definition}

By way of comparison, Kronheimer and Mrowka insist in \cite{km-excision} that $(M,\Gamma)$ be the complement of a null-homologous knot $K$ in a closed manifold $Z$, with $\Gamma$ a pair of oppositely-oriented meridians.  They then work with a specific genus-one closure (denoted $\tilde{Y}(Z,K)$ in \cite[Definition~5.4]{km-excision}) which is adapted to any Seifert surface $\Sigma$ for $K$.
 
Let \[\overline\Sigma=\Sigma'\cup \big( c_+\times[1,3]\big)\] be the closed surface in $Y$ obtained as the union of $\Sigma'$ with the annulus $c_+\times[1,3]\subset R\times[1,3]$. Note that $\overline\Sigma$ is obtained from $\Sigma$ by capping off its boundary with a punctured torus, so that \[g(\overline \Sigma)= g(\Sigma)+1.\] We use $\overline\Sigma$ to define an Alexander grading on $\SHI(\data)$, generalizing the construction of \cite[Section~7.6]{km-excision}, as follows.

\begin{definition}
\label{def:alexgrading} Given a closure $\data$ of $(M,\Gamma)$  adapted to $\Sigma$,  the sutured instanton homology of $\data$ \emph{in Alexander grading $i$ relative to $\Sigma$} is the generalized $2i$-eigenspace of the operator \[\mu(\overline\Sigma):\SHI(\data)\to \SHI(\data).\] We denote this generalized eigenspace by $\SHI(\data,[\Sigma],i)$.\end{definition}

\begin{remark} As the notation above suggests, the Alexander grading on $(M,\Gamma)$ relative to $\Sigma$ depends only on the class \[[\Sigma]\in H_2(M, \sigma)\] as this  class determines the homology class of $\overline\Sigma$, and the operator $\mu(\overline\Sigma)$ depends only on $[\overline\Sigma]$.\end{remark}

\begin{lemma}
\label{lem:gradingbound}
The group $\SHI(\data,[\Sigma],i)$ is trivial for $|i|> g(\Sigma)$. Thus, \[\SHI(\data) = \bigoplus_{i=-g(\Sigma)}^{g(\Sigma)}\SHI(\data,[\Sigma],i).\]
\end{lemma}
\begin{proof}
This follows   immediately from Proposition \ref{prop:mu-spectrum}, which implies  that  the eigenvalues of $\mu(\overline\Sigma)$ on $\SHI(\data)$ belong to the set of even integers from $2-2g(\overline\Sigma)$ to $2g(\overline\Sigma)-2$.
 \end{proof}

We next prove that this construction gives a well-defined Alexander grading on $\SHI(M,\Gamma)$.

\begin{theorem}
\label{thm:alexwelldefined}
Suppose $\data$ and $\data'$ are closures of $(M,\Gamma)$ adapted to $\Sigma$. For each $i$, we have \[\SHI(\data,[\Sigma],i)\cong\SHI(\data',[\Sigma],i).\]  Moreover, when $\data$ and $\data'$ have genus at least two, the canonical  isomorphism $\Psi_{\data,\data'}$ restricts to an isomorphism \[\Psi_{\data,\data'}:\SHI(\data,[\Sigma],i)\xrightarrow{\cong}\SHI(\data',[\Sigma],i)\] for each $i$.  \end{theorem}

\begin{proof}
We will proceed in two steps: first we consider closures of the same genus, and then we see what happens when we increase the genus by one.  In the first case, the isomorphism $\Psi_{\data,\data'}$ is built out of cobordism maps induced by Dehn surgeries, and these surgeries avoid the closed surfaces $\overline\Sigma$ and $\overline\Sigma'$ built from $\Sigma$ whose eigenspaces define the Alexander grading, so the grading passes through $\Psi_{\data,\data'}$ unchanged.  In the second case, we choose a convenient closure $\data$ whose auxiliary portion can be cut open to ``insert'' extra genus, leading to a closure $\data'$ of genus $g(\data)+1$.  Then $\Psi_{\data,\data'}$ is realized by a more complicated ``excision'' cobordism, and we will have to work a bit harder to relate the homology classes (and hence eigenspaces) of $\overline\Sigma$ and $\overline\Sigma'$ on either end of this cobordism.


To begin, we let \[\data=(Y,R,\eta,\alpha) \,\textrm{ and }\, \data' = (Y',R',\eta',\alpha')\] be closures of $(M,\Gamma)$ adapted to $\Sigma$. Let us suppose first that $g(\data)=g(\data')$. Since the curves $c_+$ and $\eta$ are essential in $R$ and intersect in one point, we can find a diffeomorphism \begin{equation}\label{eqn:isoR}R\xrightarrow{\cong} R'\end{equation} which identifies $c_+,\eta\subset R$ with the corresponding  $c_+',\eta'\subset R'$. We can also ensure that this map sends the point $q$ defining $\alpha$ to the corresponding point $q'$ defining $\alpha'$. 

Note that for genus $1$ closures, there is a unique isotopy class of such diffeomorphisms since the complements \[R\ssm(c_+\cup\eta) \textrm{ and } R'\ssm(c_+'\cup\eta')\] are disks in this case. Thus, we automatically have \[\SHI(\data,[\Sigma],i)\cong \SHI(\data',[\Sigma],i)\] when $g(\data) = g(\data')=1$.
More generally,   the diffeomorphism in \eqref{eqn:isoR} allows us to view $\data$ and $\data'$ as formed from the \emph{same} preclosure $M'$, according to different diffeomorphisms \[\varphi,\varphi':\partial_+M'\to\partial_-M',\] where $\varphi^{-1}\varphi'$ is a diffeomorphism of $R$  fixing $(c_+,\eta,q)$. We may  factor $\varphi^{-1}\varphi'$ as a composition of positive and negative Dehn twists around curves \[a_1,\dots,a_n\subset R\] disjoint from $(c_+,\eta,q)$. This then allows us to view $\data'$ as obtained from $\data$ via $(\pm 1)$-surgeries on copies \[a_i\times\{t_i\}\subset R\times\{t_i\}\subset R\times[1,3]\subset Y\] of the $a_i$.

Suppose first  that only positive Dehn twists  appear in the factorization of $\varphi^{-1}\varphi'$. Let $(W,\nu)$ be the associated cobordism from $Y$ to $Y'$, obtained from $Y\times[0,1]$ by attaching $(-1)$-framed $2$-handles along the $a_i\times\{t_i\}$ in $Y\times \{1\}$, with $\nu = (\alpha\sqcup\eta)\times[0,1]$. Since $R$ and $R'$ are isotopic in $W$ and \[2g(R)-2=2g(R')-2\]
 Lemma \ref{lem:commute} implies that the induced map \[I_*(W)_\nu:I_*(Y)_{\alpha+\eta}\to I_*(Y')_{\alpha'+\eta'}\] restricts to a map \begin{equation}\label{eqn:mapiso}I_*(W)_\nu:\SHI(\data)\to\SHI(\data').\end{equation}  The map in \eqref{eqn:mapiso} defines the canonical isomorphism $\Psi_{\data,\data'}$ when both  closures have genus at least 2 \cite[Definition~9.10]{bs-naturality}. Since the  $a_i$ are disjoint from $c_+$, the surgery curves $a_i\times\{t_i\}$ are disjoint from the annulus $c_+\times[1,3]$ which caps off the surface $\Sigma'$ to form $\overline \Sigma\subset Y$. The capped off  surfaces  $\overline\Sigma\subset Y$ and $\overline\Sigma'\subset Y'$ are therefore isotopic in $W$.  Lemma \ref{lem:commute} then implies that $I_*(W)_\nu$ restricts to an isomorphism \[I_*(W)_\nu:\SHI(\data,[\Sigma],i)\xrightarrow{\cong}\SHI(\data',[\Sigma],i)\] for each $i$, as claimed. 

If both positive and negative Dehn twists  appear in the factorization of $\varphi^{-1}\varphi'$ then the isomorphism relating $\SHI(\data)$ and $\SHI(\data')$ is defined as the composition of  a map as in \eqref{eqn:mapiso} with the inverse of such a map, so the same argument applies.

Next, suppose $\data=(Y,R,\eta,\alpha)$ is a closure of $(M,\Gamma)$ adapted to $\Sigma$ as in the beginning of this section, and let us borrow the notation used there. We will construct a closure $\data'$ adapted to $\Sigma$ with \[g(\data')=g(\data)+1\]  and show that the isomorphism relating $\SHI(\data)$ and $\SHI(\data')$ preserves Alexander gradings.

To start, let $d\subset T$ be a closed curve dual to the arcs $\tau\subset T$ and $\eta\cap T$ (we can choose $\eta$ so that these arcs are parallel). Let us assume that the map \[\varphi:\partial_+M'\to\partial_-M'\] used to form $Y$ identifies the curves \[d\times\{\pm 1\}\subset \partial T\times[-1,1]\subset \partial_{\pm}M'.\] Let $F$ be a closed genus 2 surface containing parallel  curves $\boldsymbol{\tau}$ and $\boldsymbol{\eta}$ and a curve $\mathbf{d}$  dual to both. Let $T'$ be the  surface obtained by cutting $T$ and $F$ open along $d$ and $ \mathbf{d}$ and gluing these cut-open surfaces together according to a homeomorphism of their boundaries which identifies $d\cap \tau$ and $d\cap \eta$ with $ \mathbf{d}\cap  \boldsymbol{\tau}$ and   $ \mathbf{d}\cap  \boldsymbol{\eta}$, respectively, as shown in Figure \ref{fig:splice}. 

\begin{figure}[ht]
\labellist
\tiny \hair 2pt
\pinlabel $R_+(\Gamma)$ at 25 23
\pinlabel $R_+(\Gamma)$ at 345 22

\pinlabel $T$ at 60 24
\pinlabel $F$ at 177 23
\pinlabel $T'$ at 382 23
\pinlabel $\tau$ at 60 97
\pinlabel $\sigma_{+}$ at 28 97
\pinlabel $d$ at 92 65
\pinlabel $\eta$ at 25 53
\pinlabel $\boldsymbol{\tau}$ at 179 97
\pinlabel $\mathbf{d}$ at 142 65
\pinlabel $\boldsymbol{\eta}$ at 178 52
\pinlabel $\sigma_{+}$ at 347 96
\pinlabel $\tau'$ at 382 98
\pinlabel $\eta'$ at 347 54

\endlabellist
\centering
\includegraphics[width=14cm]{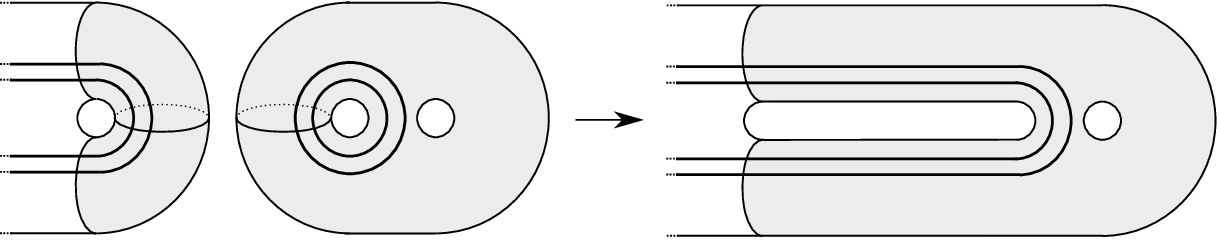}
\caption{Left, the genus $2$ surface $F$ and portion of $R = R_+(\Gamma)\cup T$  in an example where $T$ is has genus $0$ and $2$ boundary components. Right, the result of cutting and regluing to form $R' = R_+(\Gamma)\cup T'$. In performing this operation, we increase the genus of the closure by one.}
\label{fig:splice}
\end{figure}
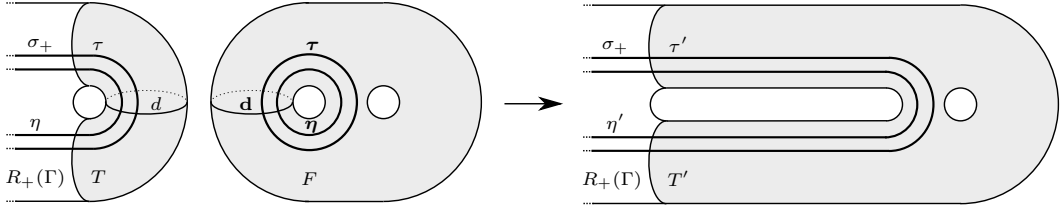

The union $\tau'=\tau\cup \boldsymbol{\tau}$ is then a nonseparating, properly embedded arc in $T'$. Let \[\data'=(Y',R',\eta'=\eta\cup \boldsymbol{\eta},\alpha'=\alpha)\] be the closure of $(M,\Gamma)$  adapted to $\Sigma$ formed using the auxiliary surface $T'$ and a diffeomorphism $\varphi'$ which restricts to \[\varphi:\partial_+M'\ssm (d\times\{1\})\to\partial_-M'\ssm (d\times\{-1\})\] and to \[\id:(F\ssm \mathbf{d})\times\{1\}\to(F\ssm \mathbf{d})\times\{-1\}.\] In particular, $\varphi'$ identifies the two circles \[c'_\pm:=\sigma_\pm\cup\tau'_\pm\] so that $\Sigma$ caps off to a closed  surface $\overline\Sigma'$ in $Y'$ in the usual manner. 

Note that $Y'$ is obtained by cutting $Y$  and $F\times S^1$ open along the  tori \[d\times S^1=(d\times[-1,1])\cup (d\times[1,3])\] and $\mathbf{d}\times S^1$, respectively, and regluing. From this point of view, the capped off  surface $\overline\Sigma'$ is obtained by cutting $\overline\Sigma$ and the torus $\boldsymbol{\tau}\times S^1\subset F\times S^1$ open along essential curves and regluing. There is a standard excision-type cobordism \[(W,\nu):(Y,\alpha\sqcup \eta)\sqcup(F\times S^1,\boldsymbol{\eta})\to(Y',\alpha'\sqcup\eta')\] associated to this cutting and regluing, as described in \cite[Section 3]{km-excision} and \cite[Section 9.3.2]{bs-naturality}, and a corresponding map \[I_*(W)_\nu:  I_*(Y)_{\alpha+ \eta}\otimes I_*(F\times S^1)_{\boldsymbol{\eta}}\to  I_*(Y')_{\alpha'+ \eta'}.\] 

Kronheimer and Mrowka prove in \cite[Proposition 7.9]{km-excision} that the generalized $(2,2)$-eigenspace of the operator $(\mu(F),\mu(\pt))$ acting on $I_*(F\times S^1)_{\boldsymbol{\eta}}$ is one-dimensional. Let $\Theta$ be a generator of this eigenspace. Then, since the disjoint union of surfaces \[R\sqcup F\subset Y\sqcup(F\times S^1)\] is homologous in $W$ to $R'\subset Y'$, and \[2g(R')-2 = (2g(R)-2)+2,\] Lemma \ref{lem:commutethree} implies that  $I_*(W)_\nu(\cdot,\Theta)$ defines a map \begin{equation}\label{eqn:thetamap}I_*(W)_\nu(\cdot,\Theta): \SHI(\data)\to\SHI(\data').\end{equation}
The map in \eqref{eqn:thetamap} is an isomorphism \cite[Section 3]{km-excision}.  (Roughly, we can glue $W$ to an excision-type cobordism in the other direction, and the result can be turned into a product cobordism by replacing an embedded $R\times I\times S^1$ with two copies of $R\times D^2$; this only changes the induced map on the top eigenspace by a scalar factor, so it must have been an isomorphism.)  Moreover, it agrees with the canonical isomorphism $\Psi_{\data,\data'}$ defined in \cite{bs-naturality} when $g(\data)\geq 2$. 

Since $\boldsymbol\eta$ intersects the torus $\mathbf{d}\times S^1\subset F\times S^1$ in a single point, Theorem \ref{thm:simultaneouseigenvalues} says that the only eigenvalue of  $\mu(\mathbf{d}\times S^1)$ acting on $I_*(F\times S^1)_{\boldsymbol{\eta}}$ is $0$. It follows that the generalized $2$-eigenspace of $\mu(\pt)$ on $I_*(F\times S^1)_{\boldsymbol{\eta}}$ is simply the group we denote by $I_*(F\times S^1|\,\mathbf{d}\times S^1)_{\boldsymbol{\eta}}.$ We therefore have that \[\Theta\in I_*(F\times S^1|\,\mathbf{d}\times S^1)_{\boldsymbol{\eta}}.\] Note that  the only eigenvalue of $\mu(\boldsymbol{\tau}\times S^1)$ acting on $I_*(F\times S^1|\,\mathbf{d}\times S^1)_{\boldsymbol{\eta}}$ is $0$ by Proposition \ref{prop:mu-spectrum} since $\boldsymbol{\tau}\times S^1$ is a torus. In particular,  $\Theta$ is in the generalized $0$-eigenspace of this operator $\mu(\boldsymbol{\tau}\times S^1)$.
Then, since the disjoint union of surfaces \[ \overline\Sigma\sqcup (\boldsymbol{\tau}\times S^1)\subset  Y\sqcup(F\times S^1)\] is homologous in $W$ to $\overline\Sigma'\subset Y'$, Lemma \ref{lem:commutethree} implies that the map $I_*(W)_\nu(\cdot,\Theta)$ restricts to an isomorphism \[I_*(W)_\nu(\cdot,\Theta): \SHI(\data,[\Sigma],i)\xrightarrow{\cong}\SHI(\data',[\Sigma],i)\]  for each $i$.

Finally, for \emph{any} two closures $\data$ and $\data'$ of $(M,\Gamma)$ adapted to $\Sigma$, we define the isomorphism (the canonical isomorphism if both closures have genus at least 2) \[\SHI(\data)\to \SHI(\data')\] to be a composition of isomorphisms defined as above. This completes the proof.
\end{proof}

Given Theorem \ref{thm:alexwelldefined}, we make the following definition for a surface $\Sigma$ in $(M,\Gamma)$ as above.

\begin{definition}
 The sutured instanton homology of $(M,\Gamma)$ \emph{in Alexander grading $i$ relative to $\Sigma$} is the projectively transitive system of $\C$-modules \[\SHI(M,\Gamma,[\Sigma],i)\] consisting of the  groups $\SHI(\data,[\Sigma],i)$ for $g(\data)\geq 2$, together with the canonical isomorphisms $\Psi_{\data,\data'}$ between them from Theorem~\ref{thm:alexwelldefined}.
\end{definition}

\begin{remark}
Note that  \[\SHI(M,\Gamma) = \bigoplus_{i=-g(\Sigma)}^{g(\Sigma)}\SHI(M,\Gamma,[\Sigma],i),\] by Lemma \ref{lem:gradingbound}.
\end{remark}


\begin{remark}
\label{rmk:symmetry} If $(M,\Gamma)$ admits a genus one closure then  the Alexander grading on $\SHI(M,\Gamma)$ with respect to $\Sigma$ is symmetric by Lemma \ref{lem:symmetric}. That is, 
\[\SHI(M,\Gamma,[\Sigma],i)\cong\SHI(M,\Gamma,[\Sigma],-i)\] for each $i$.
\end{remark}

Following \cite[Section 7.6]{km-excision}, we  make the  definition below.

\begin{definition} 
\label{def:khialex}Given a knot $K$ in a closed, oriented $3$-manifold $Y$ with Seifert surface $\Sigma$, the instanton knot Floer homology of $K$ \emph{in Alexander grading $i$ relative to $\Sigma$} is \[\KHI(Y,K,[\Sigma],i):=\SHI(Y(K),\Gamma_\mu,[\Sigma],i).\] When the relative homology class of $\Sigma$ is unambiguous we will omit it from the notation.
\end{definition}

\begin{remark} \label{rem:khi-genus-one-closure}
In \cite{km-excision}, Kronheimer and Mrowka defined the Alexander grading on $\KHI$ in exactly the same way as we do, but using  genus one closures only. It follows from Theorem \ref{thm:alexwelldefined} that the Alexander grading we give in Definition \ref{def:khialex} agrees with theirs up to isomorphism.
\end{remark}


\begin{remark}
\label{rmk:symmetryknots} Since the knot complement $(Y(K),\Gamma_\mu)$ admits a genus one closure,  the Alexander grading on $\KHI(Y,K)$ relative to $\Sigma$ is symmetric by Remark \ref{rmk:symmetry}. That is, 
\[\KHI(Y,K,[\Sigma],i)\cong\KHI(Y,K,[\Sigma],-i)\] for each $i$.
\end{remark}

Kronheimer and Mrowka proved the following three results. The first follows from \cite[Theorem 7.18]{km-excision} and the discussion at the end of \cite[Section 7.6]{km-excision}; the second  is \cite[Proposition 7.16]{km-excision}; and the third is \cite[Proposition 4.1]{km-alexander}. As explained in the introduction, all three are important in our proof of  Theorem \ref{thm:khi-detects-trefoil}.

\begin{theorem}
\label{thm:fiberedY}
If $K\subset Y$ is fibered with fiber $\Sigma$ then $\KHI(Y,K,[\Sigma],g(K))\cong\C$.
\end{theorem}

\begin{theorem}
\label{thm:genus}
For $K\subset S^3$, $\KHI(S^3,K,g(K))\neq 0$.
\end{theorem}

\begin{theorem}
\label{thm:fiberedS}
For $K\subset S^3$, $\KHI(S^3,K,g(K))\cong\C$ if and only if $K$ is fibered.
\end{theorem}

\section{The bypass exact triangle}
\label{sec:bypass}

In this section, we prove the bypass exact triangle, stated in the introduction as Theorem \ref{thm:bypass}. Our proof is very similar to the proof we gave in \cite[Section 5]{bs-shm} for sutured monopole homology, and we will rely on the topological ideas developed there. We must be a bit careful in the instanton Floer setting, however, with the bundles involved in the exact triangle. 

Suppose $(M,\Gamma)$ is a balanced sutured manifold and  $\alpha\subset \partial M$ is an arc which intersects $\Gamma$ transversally in three points, including both endpoints of $\alpha$. A \emph{bypass move} along $\alpha$ replaces $\Gamma$ with a new set of sutures $\Gamma'$ which differ from $\Gamma$ in a neighborhood of $\alpha$, as shown in Figure \ref{fig:bypass-move}. 

\begin{figure}[ht]
\labellist
\small \hair 2pt
\pinlabel $\alpha$ at 22 28

\endlabellist
\centering
\includegraphics[width=4.5cm]{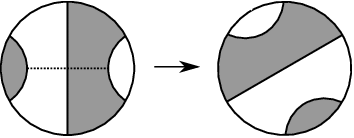}
\caption{A bypass move along the arc $\alpha$, with $\Gamma$ on the left and $\Gamma'$ on the right. The gray and white regions indicate the negative and positive regions, respectively.}
\label{fig:bypass-move}
\end{figure}

As background, bypasses were introduced by Honda \cite[\S3.4]{honda-lens}.  The rough idea, also present in \cite{giroux-lens}, is that a contact structure on a thickened surface $\Sigma \times [0,1]$ can be described in terms of some discrete data, namely the dividing curves on $\Sigma \times \{0\}$ and finitely many times in $[0,1]$ at which the dividing curves on $\Sigma \times \{t\}$ change.  These changes are modeled by taking $\Sigma \times \{t \pm\epsilon\}$ to be the boundary of a neighborhood of $\Sigma \cup D$, where the bypass $D$ is a certain half-disk with Legendrian boundary, attached to $\Sigma$ along an arc $\alpha$ as shown in Figure~\ref{fig:bypass-move}, and then the change in dividing curves from $\Sigma\times\{t-\epsilon\}$ to $\Sigma \times \{t+\epsilon\}$ is precisely a bypass move.  They have been widely applied to classify tight contact structures on many 3-manifolds, as well as Legendrian representatives of various smooth knot types in $S^3$; the original use in \cite{honda-lens} was to classify tight contact structures on lens spaces.

A bypass move can be achieved by attaching a contact $1$-handle  along disks in $\partial M$ centered at the endpoints of $ \alpha$ and then attaching a contact $2$-handle along the union $\beta$ of $\alpha$ with an arc  on the boundary of this $1$-handle, as shown in Figure \ref{fig:bypass-handles}. We  refer to this  sequence of handle attachments  as a \emph{bypass attachment along $\alpha$}, following \cite{honda-lens, ozbagci}. A bypass attachment along $\alpha$ therefore gives rise to a morphism \[\phi_\alpha:\SHI(-M,-\Gamma)\to \SHI(-M,-\Gamma')\] which is the composition of the corresponding contact $1$- and $2$-handle attachment maps defined in Section \ref{ssec:shi}.  Theorem \ref{thm:handletheta2} implies the following.

\begin{figure}[ht]
\labellist
\small \hair 2pt
\pinlabel $\alpha$ at 51 56
\pinlabel $\beta$ at 242 36
\endlabellist

\centering
\includegraphics[width=8cm]{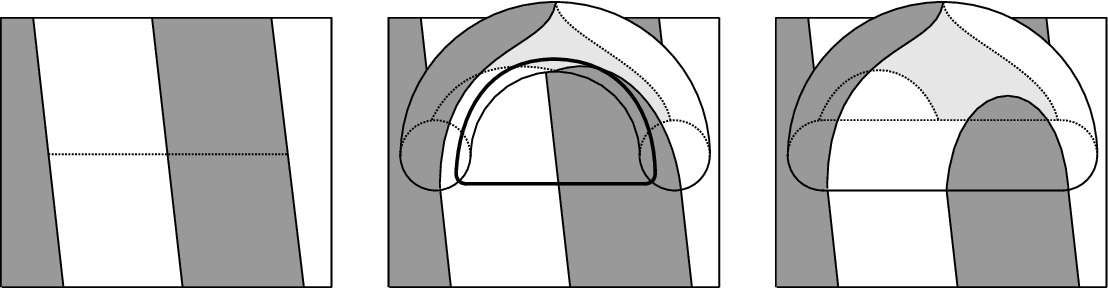}
\caption{Performing a bypass move by attaching a contact $1$-handle at the endpoints of $\alpha$ and a contact $2$-handle along $\beta$.}
\label{fig:bypass-handles}
\end{figure}

\begin{proposition}
\label{prop:bypass}
Suppose $(M,\Gamma',\xi')$ is obtained from $(M,\Gamma,\xi)$ by attaching a bypass along $\alpha$. Then the induced  map $\phi_\alpha$ sends $\cinvt(\xi)$ to $\cinvt(\xi')$.
\end{proposition}

Figure \ref{fig:bypass-triangle} shows a   sequence of bypass moves, performed in some fixed neighborhood in $\partial M$, resulting in a 3-periodic sequence of sutures on $M$. Such a sequence  is called a \emph{bypass triangle}. Unpublished work of Honda shows that a bypass triangle gives rise to a \emph{bypass exact triangle} in sutured Heegaard Floer homology. We proved a similar result in the  monopole Floer setting  in \cite[Theorem 5.2]{bs-shm}. Here, we prove the analogue  for sutured instanton   homology, stated below.

\begin{figure}[ht]
\labellist
\small \hair 2pt
\pinlabel $\alpha_1$  at 24 117
\pinlabel $\alpha_2$  at 149 118
\pinlabel $\alpha_3$  at 97 37
\pinlabel $\Gamma_1$  at -3 145
\pinlabel $\Gamma_2$  at 173 144
\pinlabel $\Gamma_3$  at 85 -8
\endlabellist
\centering
\includegraphics[width=4.5cm]{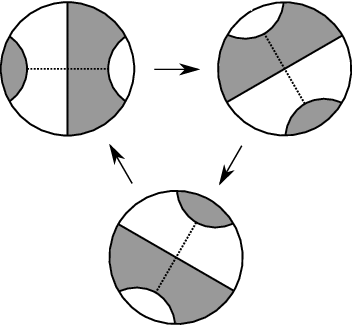}
\caption{The bypass triangle.  Each picture shows the attaching arc used to achieve the next set of sutures in the triangle.}
\label{fig:bypass-triangle}
\end{figure}

{
\renewcommand{\thetheorem}{\ref{thm:bypass}}
\begin{theorem}
Suppose $\Gamma_1,\Gamma_2,\Gamma_3\subset \partial M$ is a 3-periodic sequence of sutures resulting from successive bypass moves along arcs $\alpha_1,\alpha_2,\alpha_3$ as in Figure \ref{fig:bypass-triangle}. Then there is an exact triangle
\[ \xymatrix@C=-25pt@R=35pt{
\SHI(-M,-\Gamma_1) \ar[rr]^{\phi_{\alpha_1}} & & \SHI(-M,-\Gamma_2) \ar[dl]^{\phi_{\alpha_2}} \\
& \SHI(-M,-\Gamma_3), \ar[ul]^{\phi_{\alpha_3}} & \\
} \]
in which $\phi_{\alpha_1}, \phi_{\alpha_2},\phi_{\alpha_2}$ are the corresponding bypass attachment maps.
\end{theorem}
\addtocounter{theorem}{-1}
}

\begin{proof}
The main idea behind the proof  is that there are closures of these three sutured manifolds which are related as in the surgery exact triangle of Theorem \ref{thm:exacttri}. This was first shown in the proof of \cite[Theorem 5.2]{bs-shm} by an argument   we review  below. Unlike in the monopole Floer case, though, it is not obvious that the cobordism maps in the surgery exact triangle relating the instanton Floer  groups of these  closures are the same as  those which induce the bypass attachment maps on $\SHI$, as the bundles on these cobordisms are different in the two cases. However, we show that one can choose closures so that any \emph{two} of the three maps agree with those that induce the bypass attachment maps. This allows us to prove exactness of the triangle in Theorem \ref{thm:bypass} at each group, one group at a time.

Note that by enlarging our local picture slightly, we can think of the arcs $\alpha_1,\alpha_2,\alpha_3$ as being arranged as in Figure \ref{fig:bypass-setup} with respect to $\Gamma_1$. We may therefore view \[(M,\Gamma_2)\,\,{\rm and} \,\,(M,\Gamma_3)\,\, {\rm and}\,\,(M,\Gamma_1)\] as being obtained from $(M,\Gamma_1)$ by attaching bypasses along the arcs \[\alpha_1\,\,{\rm and}\,\,\alpha_1,\alpha_2\,\, {\rm and}\,\, \alpha_1,\alpha_2, \alpha_3,\] respectively. As described above, attaching a bypass along $\alpha_i$ amounts to attaching a contact $1$-handle $h_i$ along disks centered at the endpoints of $\alpha_i$ and then attaching a contact $2$-handle along a curve $\beta_i$ which extends $\alpha_i$ over the handle, as in Figure \ref{fig:bypass-setup}. Let $(Z_1,\gamma_1)$ be the sutured manifold obtained by attaching all three $h_1,h_2,h_3$ to $(M,\Gamma_1)$, as  in Figure \ref{fig:bypass-setup}. 
For $i=1,2,3$, let $(Z_{i+1},\gamma_{i+1})$ be the result of attaching a contact $2$-handle to $(Z_i,\gamma_i)$ along $\beta_i$. Then
\begin{align*} 
(Z_1,\gamma_1)&= (M,\Gamma_1)\cup h_1\cup h_2\cup h_3,\\
(Z_2,\gamma_2)&= (M,\Gamma_2)\cup h_2\cup h_3,\\
(Z_3,\gamma_3)&= (M,\Gamma_3)\cup h_3,\\
(Z_4,\gamma_4)&= (M,\Gamma_1).
\end{align*}
\begin{figure}[ht]
\labellist
\small \hair 2pt
\pinlabel $\alpha_1$ at 61 79
\pinlabel $\alpha_2$ at 106 181
\pinlabel $\alpha_3$ at 151 282
\pinlabel $\beta_1$ at 769 71
\pinlabel $\beta_2$ at 815 173
\pinlabel $\beta_3$ at 861 273
\endlabellist
\centering
\includegraphics[width=11cm]{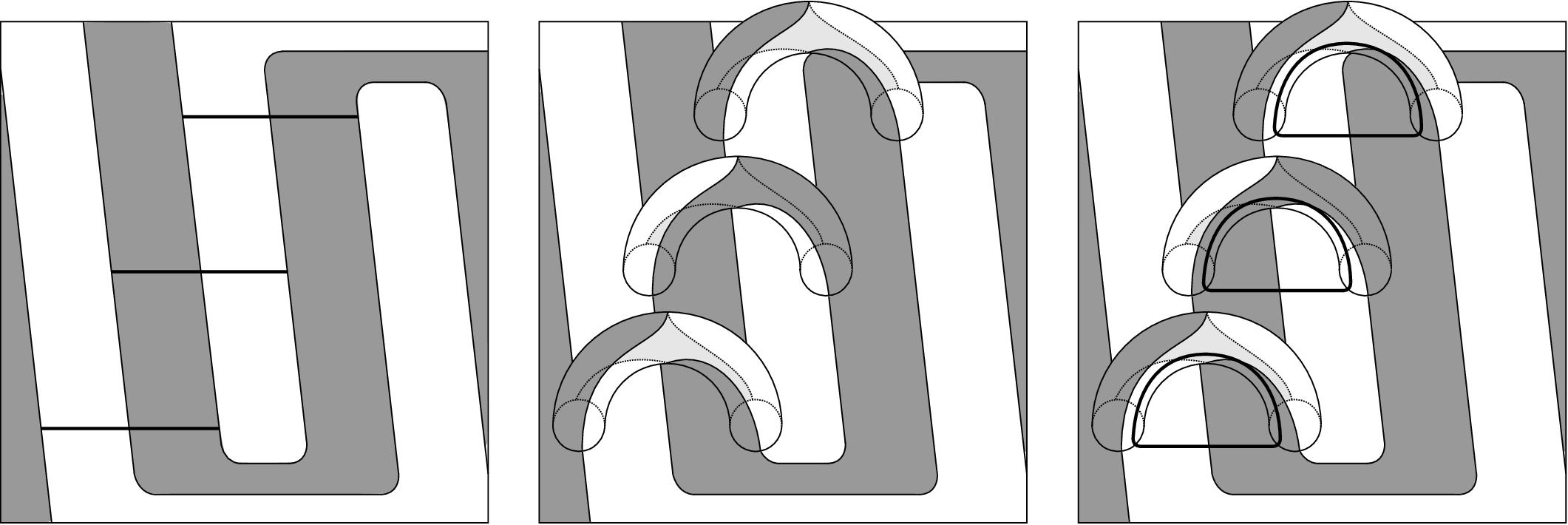}
\caption{Left, another view of the arcs $\alpha_1,\alpha_2,\alpha_3$ of attachment for the bypasses in the triangle, where the suture shown here is $\Gamma_1$. Middle, a view of $(Z_1,\gamma_1)$, obtained by attaching the contact 1-handles $h_1,h_2,h_3$ to $(M,\Gamma_1).$ Right, the attaching curves $\beta_1,\beta_2,\beta_3$ for the contact $2$-handles.}
\label{fig:bypass-setup}
\end{figure}
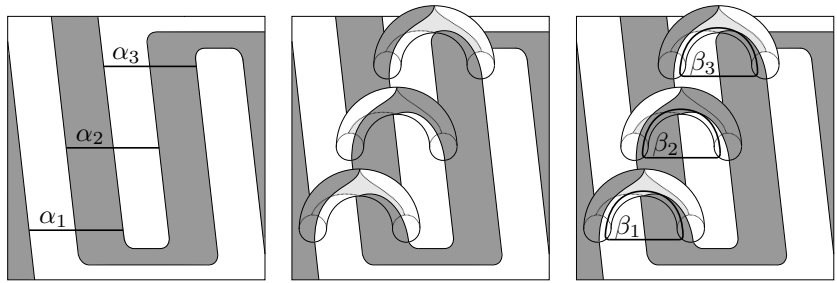

Recall from Section \ref{ssec:shi} that contact $1$-handle attachment has little effect on the level of closures. Specifically, a closure   of a sutured manifold after a $1$-handle attachment can also be viewed naturally as a closure of the   sutured manifold before the $1$-handle attachment, and the corresponding $1$-handle attachment morphism is simply  the identity map. We therefore have canonical identifications
\[\SHI(-Z_i,-\gamma_i)\cong \SHI(-M,-\Gamma_i),\] for $i=1,2,3,4$, where the subscript of $\Gamma_i$ is taken mod $3$. In particular, $\SHI(-Z_4,-\gamma_4)$ is canonically identified with $\SHI(-Z_1,-\gamma_1)$.
Therefore, to prove Theorem \ref{thm:bypass}, it  suffices to prove that there is an exact triangle 
\begin{equation}\label{eqn:exacttriz} \xymatrix@C=-25pt@R=35pt{
\SHI(-Z_1,-\gamma_1) \ar[rr]^{H_{\beta_1}} & & \SHI(-Z_2,-\gamma_2) \ar[dl]^{H_{\beta_2}} \\
& \SHI(-Z_3,-\gamma_3), \ar[ul]^{H_{\beta_3}}& \\
} \end{equation} where $H_{\beta_i}$ is the map associated to contact $2$-handle attachment along $\beta_i$.

Recall that on the level of closures, contact $2$-handle attachment corresponds to surgery. Specifically,  if $\data_i = (Y_i,R,\eta,\alpha)$ is a closure of $(Z_i,\gamma_i)$, then there is a closure of $(Z_{i+1},\gamma_{i+1})$ of the form $\data_{i+1} = (Y_{i+1},R,\eta,\alpha)$, where $Y_{i+1}$ is the result of $(\partial Z_i)$-framed surgery on $\beta_i\subset Y_i$.  To see this, we write $W_i = \overline{Y_i \setminus Z_i}$ and consider the effect of the surgery operation on each of $Z_i$ and $W_i$ separately.  Since $\beta_i$ sits on the boundaries of both $Z_i$ and $W_i$, removing a small neighborhood $N(\beta_i)$ carves a piece out of each without changing their topology.  After doing so, the intersection $\partial N(\beta_i) \cap Z_i$ is now an annulus fibered by $(\partial Z_i)$-framed push-offs of $\beta_i$, and the surgery fills each of them with a disk, which amounts to gluing a 2-handle to $Z_i$ along $\beta_i$.  The same happens on the auxiliary part $W_i$ of the closure, where gluing in the family of disks now effectively adds a handle to the auxiliary surface $T_i$ that was used to construct $Y_i$, filling in the arcs $\beta_i \cap R_\pm(\gamma_i)$ that the 2-handle operation removes from $R_\pm(\gamma_i)$, so that this surgery leads to a closure of $(Z_{i+1},\gamma_{i+1})$ without having to change the surface $R$.

Having said this, the map $H_{\beta_i}$ is now induced by the $2$-handle cobordism map
\[I_*(W_i)_{\nu}:I_*(-Y_i|{-}R)_{-\alpha-\eta}\to I_*(-Y_{i+1}|{-}R)_{-\alpha-\eta}\] corresponding to this surgery, where $\nu$ is the usual cylindrical cobordism \[\nu=(-\alpha\sqcup -\eta)\times [0,1].\] In order to prove the exactness of the triangle  \eqref{eqn:exacttriz} at $\SHI(-Z_2,-\gamma_2)$, say, it therefore suffices to prove exactness of the sequence
\begin{equation}\label{eqn:exacttriy} \xymatrix@C=18pt@R=35pt{
I_*(-Y_1|{-}R)_{-\alpha -\eta} \ar[rr]^{I_*(W_1)_{\nu}}  & &I_*(-Y_2|{-}R)_{-\alpha -\eta} \ar[rr]^{I_*(W_2)_{\nu}}  & &I_*(-Y_3|{-}R)_{-\alpha -\eta}. 
} \end{equation}
Our  approach  is to find a closure $\data_1$ of $(Z_1,\gamma_1)$ such that the surgeries relating the $-Y_i$ above are exactly the sort  one encounters in the  surgery exact triangle of Theorem \ref{thm:exacttri}, as depicted in Figure \ref{fig:surgery}. Fortunately, the proof of  \cite[Theorem 5.2]{bs-shm} shows that  for \emph{any} closure $\data_1$,

\begin{itemize}
 \item $W_1$ is the cobordism associated to $0$-surgery on some  $K=\beta_1\subset-Y_1$,
 \item $W_2$ is the cobordism associated to $(-1)$-surgery on a meridian $\mu_1\subset -Y_2$ of $K$,
 \item $W_3$ is the cobordism associated to $(-1)$-surgery on a meridian $\mu_2\subset -Y_3$ of $\mu_1$,
  \end{itemize}
  as desired.
 Theorem \ref{thm:exacttri} then says that there is an exact sequence of the form 
 \begin{equation}\label{eqn:exacttriy2} \xymatrix@C=18pt@R=35pt{
I_*(-Y_1|{-}R)_{-\alpha -\eta+ K} \ar[rr]^{I_*(W_1)_{\kappa_1}}  & &I_*(-Y_2|{-}R)_{-\alpha -\eta} \ar[rr]^{I_*(W_2)_{\kappa_2}}  & &I_*(-Y_3|{-}R)_{-\alpha -\eta}.
} \end{equation}

Note that this sequence \eqref{eqn:exacttriy2} is subtly different from that in \eqref{eqn:exacttriy}. For one thing, the Floer group on the  left is defined using the $1$-manifold $-\alpha\sqcup -\eta\sqcup K$ rather than $\alpha\sqcup -\eta.$ On the other hand, this  group depends, up to isomorphism, only on the homology class of this $1$-manifold, and we can find a closure $\data_1$ such that $K$ is null-homologous in $-Y_1$. Indeed, the construction in the beginning of Section \ref{sec:alex}, and illustrated in Figure \ref{fig:surfaceclosure}, shows that for any curve (like $K=\beta_1$) in the boundary of a sutured manifold which intersects the sutures twice, we can find a closure in which this curve bounds a once-punctured torus such that the  framing on the curve induced by this surface agrees with the framing induced by the boundary. Let $\data_1$ be such a closure. Then \eqref{eqn:exacttriy2} becomes the exact sequence
\begin{equation}\label{eqn:exacttriy3} \xymatrix@C=18pt@R=35pt{
I_*(-Y_1|{-}R)_{-\alpha -\eta} \ar[rr]^{I_*(W_1)_{\overline{\kappa}_1}}  & &I_*(-Y_2|{-}R)_{-\alpha -\eta} \ar[rr]^{I_*(W_2)_{\kappa_2}}  & &I_*(-Y_3|{-}R)_{-\alpha -\eta}
} \end{equation}
where $\overline{\kappa}_1$ is the cobordism given as the composition of the once-punctured torus viewed as a cobordism from $\emptyset$ to $K$ with  $\kappa_1$. To deduce \eqref{eqn:exacttriy} from \eqref{eqn:exacttriy3}, it suffices to show that \begin{equation}\label{eqn:mapsequal}I_*(W_1)_\nu = I_*(W_1)_{\overline{\kappa}_1} \textrm{ and }I_*(W_2)_\nu = I_*(W_2)_{\kappa_2}.\end{equation} Recall that these maps depend only on the relative homology classes of the various $2$-dimensional cobordisms involved. Since we have chosen a closure $\data_1$ in which $K$ is nullhomologous in $Y_1$ and $Y_2$ is obtained via $0$-surgery on $K$ with respect to the framing induced by the once-punctured torus providing the nullhomology, we have that \[b_1(Y_1)=b_1(Y_2)-1=b_1(Y_3).\] The long exact sequence of the pair $(W_1,\partial W_1)$ shows in this case that a relative homology class in $H_2(W_1,\partial W_1)$ is determined by its boundary. Since $\nu$ and $\overline{\kappa}_1$ have the same boundary, they represent the same class, which then implies the first equality in \eqref{eqn:mapsequal}; likewise, for the second equality. This proves that the triangle in \eqref{eqn:exacttriz} is exact at $\SHI(-Z_2,\gamma_2)$. Identical arguments show that it is exact at the other groups, completing the proof of Theorem \ref{thm:bypass}.
\end{proof}

\section{Invariants of Legendrian and transverse knots}
\label{sec:leg}

In this section, we  define invariants of Legendrian and transverse knots in instanton knot Floer homology. As described in the introduction, our construction is motivated by Stipsicz and V{\'e}rtesi's interpretation \cite{stipsicz-vertesi} of  the Legendrian and transverse knot invariants in Heegaard Floer homology defined by Lisca, Ozsv{\'a}th, Stipsicz, and Szab{\'o}  \cite{loss}, and is nearly identical to a previous construction of the authors in monopole knot Floer homology   \cite{bs-legendrian}. These  invariants and their properties will be  important in our proof of Theorem \ref{thm:khi-fibered} in the next section, as outlined in the introduction.

Suppose $K$ is an oriented Legendrian knot in a closed contact $3$-manifold $(Y,\xi)$. Let \begin{equation*}\label{eqn:complement1}(Y(K),\Gamma_K,\xi_{K})\end{equation*} be the sutured contact manifold obtained by removing a standard neighborhood of $K$ from $(Y,\xi)$. Attaching a bypass to this knot complement along the arc $c\subset \partial Y(K)$  shown in Figure \ref{fig:bypasses2} yields  the knot complement with its  two meridional sutures. We   refer to this operation as a \emph{Stipsicz-V{\'e}rtesi bypass attachment} as it was studied extensively by those authors in \cite{stipsicz-vertesi}.  Let us denote by \begin{equation}\label{eqn:complement}(Y(K),\Gamma_\mu,\xi_{\mu,K})\end{equation}  the sutured contact manifold  obtained via this  attachment. Inspired by  \cite[Theorem 1.1]{stipsicz-vertesi}, we define the Legendrian invariant $\linvt(K)$ to be the contact invariant of this  manifold.

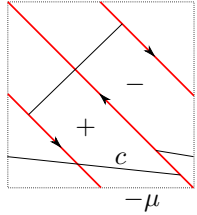
\begin{figure}[ht]
\labellist
\small \hair 2pt
\pinlabel $-\mu$ at 170 0

\pinlabel $+$ at 105 85
\pinlabel $-$ at 162 135

\pinlabel $c$ at 145 49
\endlabellist
\centering
\includegraphics[width=2.5cm]{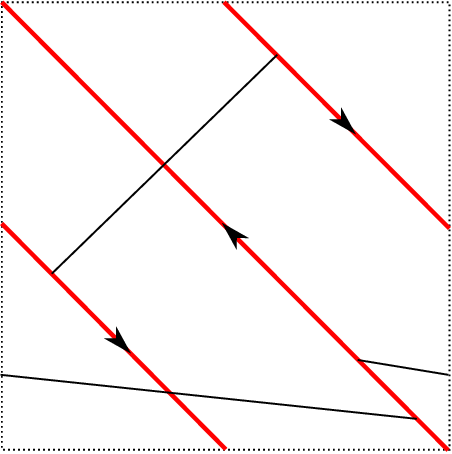}
\caption{The boundary of the complement of a standard neighborhood of $K$, with dividing set $\Gamma_K$ in red. Here the horizontal axis is a meridian, and the vertical axis is some longitude of $K$.  The Stipsicz-V{\'e}rtesi bypass is attached along the arc $c$. The $\pm$ indicate the regions $R_\pm(\Gamma_{K})$.}
\label{fig:bypasses2}
\end{figure}

\begin{definition} Given a Legendrian knot $K\subset (Y,\xi)$, let \[\linvt(K):=\cinvt(\xi_{\mu,K})\in\SHI(-Y(K),-\Gamma_\mu)=\KHI(-Y,K).\] This class is by construction an invariant of the Legendrian knot type of $K$.
\end{definition} 

Stipsicz and V{\'e}rtesi observed in the proof of \cite[Theorem 1.5]{stipsicz-vertesi} that the sutured contact manifold $(Y(K),\Gamma_\mu,\xi_{\mu,K})$ and therefore the class $\linvt(K)$ is invariant under negative Legendrian stabilization of $K$.
This  then enables us to define an invariant of transverse knots in $(Y,\xi)$ via Legendrian approximation as below since any two Legendrian approximations of a transverse knot are related by negative Legendrian stabilization.

\begin{definition}
Given a transverse knot $K\subset (Y,\xi)$ with  Legendrian approximation $\mathcal{K}$, let 
\[\kinvt(K):=\linvt(\mathcal{K})\in\KHI(-Y,K).\] This class is an invariant of the transverse knot type of $K$.
\end{definition}

\begin{remark} 
\label{rmk:sv}Given a transverse knot  $K\subset (Y,\xi)$ with Legendrian approximation $\mathcal{K}$, we have  \[\kinvt(K)=\linvt(\mathcal{K})=\cinvt(\xi_{\mu,\mathcal{K}})=\phi^{SV}(\cinvt(\xi_\mathcal{K}))\] where \[\phi^{SV}:\SHI(-Y(K),-\Gamma_\mathcal{K})\to\KHI(-Y,K)\] is the map our theory associates to the Stipsicz-V{\'e}rtesi bypass attachment.
\end{remark}


 
 
 Below, we prove some results about  the Legendrian invariant $\linvt$ which will be important in Section \ref{sec:proof}. Our proofs  are similar to those of analogous  results in the monopole Floer setting  \cite{bs-legendrian, sivek-legendrian}. First, we establish the following notation.

\begin{notation}Given a closed contact 3-manifold $(Y,\xi)$, we will denote by $Y(1)$ the sutured contact manifold \[Y(1)=(Y\ssm B^3, \Gamma_{S^1},\xi|_{Y\ssm B^3})\] obtained by removing a Darboux ball from $(Y,\xi)$, with dividing set $\Gamma_{S^1}$ consisting of a single curve on the boundary. In particular, we will write $\cinvt(Y(1))$ for the contact invariant of this sutured contact manifold.
\end{notation}

The result below is an analogue of \cite[Proposition 3.13]{bs-legendrian}.

\begin{lemma}
\label{lem:unknot}
Suppose $U\subset (Y,\xi)$ is a Legendrian unknot  with $tb(U)=-1$ contained inside a Darboux ball in $(Y,\xi)$. Then there is an isomorphism \[\SHI(-Y(1))\to\KHI(-Y,U)\] which sends $\cinvt(Y(1))$ to $\linvt(U)$. 
\end{lemma}

\begin{proof}
Attaching a contact $1$-handle to $Y(1)$ results in a  sutured contact manifold of the form \[(Y(U),\Gamma_\mu,\xi').\]  The associated contact $1$-handle attachment isomorphism \begin{equation}\label{eqn:isomorphism}\SHI(-Y(1))\to\SHI(-Y(U),\Gamma_\mu)=\KHI(-Y,U)\end{equation} identifies $\cinvt(Y(1))$ with $\cinvt(\xi')$. (We recall from \S\ref{sssec:handles} that this isomorphism is merely the identity map on $\SHI(-\data)$ for some $\data$ which is simultaneously a closure of both $Y(1)$ and $Y(U)$.)  It thus suffices to check that $\xi'$ is isotopic to the contact structure $\xi_{\mu,U}$  used to define $\linvt(U)$, obtained by removing a standard neighborhood of $U$ and attaching a Stipsicz-V{\'e}rtesi bypass. Since  we can arrange that  this operation and the contact $1$-handle attachment both take place in a Darboux ball, it suffices to check this in the case  \[(Y,\xi)=(S^3,\xi_{std}).\] But in this case, $(Y(U),\Gamma_\mu)$ is a solid torus with two longitudinal sutures. As there is a unique isotopy class of tight contact structures on the solid torus with this dividing set, we need only check that  $\xi'$ and $\xi_{\mu,U}$ are both tight. 

The class $\cinvt(Y(1))$ is nonzero since $\xi_{std}$ is Stein fillable \cite[Theorem 1.4]{bs-shi}. It follows that $\cinvt(\xi')$ is nonzero as well since the isomorphism \eqref{eqn:isomorphism} identifies this class  with $\cinvt(Y(1))$. This implies that $\xi'$ is tight by Theorem \ref{thm:zero-overtwisted}.

The fact that $\xi_{\mu,U}$ is tight follows from the fact that the Heegaard Floer  Legendrian invariant of the $tb=-1$ unknot in $(S^3,\xi_{std})$ is nonzero, since   this   invariant agrees with  the Heegaard Floer contact invariant of $\xi_{\mu,U}$ according to \cite[Theorem 1.1]{stipsicz-vertesi}. (One can also prove tightness directly---i.e., without using Heegaard Floer homology---though it takes more room.)
\end{proof}

The result below follows immediately from Theorem \ref{thm:legendrian-surgery}.

\begin{lemma}
\label{lem:leg}
Let $K$ and $S$ be disjoint Legendrian knots in $(Y,\xi)$, and let $(Y',\xi')$ be the contact manifold obtained by contact $(+1)$-surgery on $S$. If $K$ has image $K'$ in $Y'$ then there is a homomorphism \[\KHI(-Y,K)\to\KHI(-Y',K')\] which sends $\linvt(K)$ to $\linvt(K')$.\qed
\end{lemma}

At the level of closures of sutured manifolds, the homomorphism of Lemma~\ref{lem:leg} is simply the cobordism map on instanton homology associated to attaching a $4$-dimensional 2-handle along $S$, as described in \S\ref{sssec:handles}.

The following is an analogue of \cite[Theorem 5.2]{sivek-legendrian} and \cite[Corollary 3.15]{bs-legendrian}.

\begin{lemma}
\label{lem:plusone}Suppose $K$ is a Legendrian knot in $(Y,\xi)$ and that $(Y',\xi')$ is the result of contact $(+1)$-surgery on $K$. Then there is a homomorphism \[\KHI(-Y,K)\to\SHI(-Y'(1))\] which sends $\linvt(K)$ to $\cinvt(Y'(1))$. 
\end{lemma}

\begin{proof}
Let $S$ be a Legendrian isotopic pushoff of $K$ with an extra positive twist around $K$, as in Figure \ref{fig:legpushoff}. 
\begin{figure}[ht]
\labellist
\tiny \hair 2pt
\pinlabel $K$ at 15 10
\pinlabel $S$ at 15 28
\endlabellist
\centering
\includegraphics[width=3.3cm]{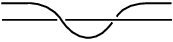}
\caption{$S$ is obtained by adding a positive twist to a Legendrian pushoff of $K$.}
\label{fig:legpushoff}
\end{figure}

 As  in the proof of \cite[Proposition 1]{ding-geiges-handles},  the image of $K$ in contact $(+1)$-surgery on $S$ is a $tb=-1$ Legendrian unknot $U\subset(Y',\xi')$.  By Lemma \ref{lem:leg},  there is a map \[\KHI(-Y,K)\to\KHI(-Y',U)\] which sends $\linvt(K)$ to $\linvt(U)$. But by Lemma \ref{lem:unknot}, we also have an isomorphism \[\KHI(-Y',U)\to\SHI(-Y'(1))\] which sends $\linvt(U)$ to $\cinvt(Y'(1)).$ Composing these two maps gives the desired map.
\end{proof}

The lemmas above culminate in  Theorem \ref{thm:tb} below, which is an analogue of \cite[Corollary 5.4]{sivek-legendrian} and \cite[Corollary 3.16]{bs-legendrian}. The exact triangle argument used in its proof was inspired by  Lisca and Stipsicz's  proof of \cite[Theorem 1.1]{lisca-stipsicz}.  We will use Theorem \ref{thm:tb}  in our proof of Theorem \ref{thm:nonzero} in the next section, which states that $\kinvt$ is nonzero for transverse bindings of open books.

\begin{theorem}
\label{thm:tb}
Suppose $K$ is a Legendrian knot in $(S^3,\xi_{std})$ with 4-ball genus $g_4(K)>0$ such that $tb(K)=2g_4-1$. Then $\linvt(K)\neq 0.$
\end{theorem}

\begin{proof}
Choose a Darboux ball disjoint from $K$. Let $S^3(1)$ denote the sutured contact manifold obtained by removing this ball from $(S^3,\xi_{std})$. Note that contact $(+1)$-surgery on $K\subset S^3(1)$ results in a sutured contact manifold \[S^3_{tb+1=2g_4}(K)(1),\] which is topologically the result of $2g_4$-surgery on $K$ with respect to its Seifert framing, minus a ball.  We have that $\cinvt(S^3(1))$ is nonzero since $\xi_{std}$ is Stein fillable  \cite[Theorem 1.4]{bs-shi}. We will use the surgery exact triangle and the adjunction inequality to show that the map \begin{equation}\label{eqn:injective}F_K:\SHI(-S^3(1))\to\SHI(-S^3_{2g_4}(K)(1))\end{equation} defined in Section \ref{sssec:handles} is injective. But we know that \[F_K(\cinvt(S^3(1)))=\cinvt(S^3_{2g_4}(K)(1))\] by Theorem \ref{thm:legendrian-surgery}. This will show that $\cinvt(S^3_{2g_4}(K)(1))$ is nonzero as well, which will then imply that $\linvt(K)$ is nonzero by Lemma \ref{lem:plusone}, proving the theorem.

Let $\data = (Z,R,\eta,\alpha)$ be a closure of $S^3(1)$ and let \begin{equation*}
\data_1=(Z_1,R,\eta,\alpha)\textrm{ and }
\data_0=(Z_0,R,\eta,\alpha)
\end{equation*} be the tuples obtained from $\data$ by performing $1$- and $0$-surgeries on $K\subset Z$ with respect to its contact framing. These tuples are naturally  closures of \[S^3_{2g_4}(K)(1)\textrm{ and } S^3_{2g_4-1}(K)(1),\] respectively. 
As observed in \eqref{eqn:surgeryexacttriangleZ}, Theorem~\ref{thm:exacttri} gives us an exact triangle
\[ \xymatrix@C=-35pt@R=30pt{
I_*(-Z|{-}R)_{-\alpha -\eta} \ar[rr]^{I_*(W)_{{\kappa}}}  & &I_*(-Z_1|{-}R)_{-\alpha -\eta} \ar[dl]^{I_*(W_{0})_{{\kappa}_0}} \\
&I_*(-Z_0|{-}R)_{-\alpha -\eta+ \mu}\ar[ul]^{I_*(W_{1})_{{\kappa}_1}}, & \\
} \]
where $\mu$ is the curve in $-Z_0$ corresponding to the meridian of $K\subset -Z$, as in Figure \ref{fig:meridian}, and moreover the map $F_K$ on $\SHI$ is induced by $I_*(W)_\kappa$.

Note that $W_{1}$ is topologically the same as the  cobordism from $Z$ to $Z_0$ obtained from $Z\times[0,1]$ by attaching a $(2g_4-1)$-framed $2$-handle along $K\subset Z\times\{1\}$ with respect to its Seifert framing. This cobordism  contains a closed surface $\overline\Sigma$ of genus $g_4$ and self-intersection \[\overline\Sigma\cdot\overline\Sigma=2g_4-1\] gotten by capping off the surface $\Sigma\subset Z\times[0,1]$ bounded by $K$ which attains \[g(\Sigma)=g_4(K)\] with the core of the $2$-handle. The surface $\overline\Sigma$ violates the adjunction inequality \cite[Theorem 1.1]{km-gauge2}, which implies that the map \[I_*(W_{1})_{{\kappa}_1}\equiv 0.\]
Therefore, \[I_*(W)_{{\kappa}}:I_*(-Z|{-}R)_{-\alpha -\eta} \to I_*(-Z_1|{-}R)_{-\alpha -\eta} \] is injective, and this means that the map $F_K$ in \eqref{eqn:injective} is injective as well. This completes the proof of Theorem \ref{thm:tb} as discussed above.
\end{proof}

\section{Proof of Theorem \ref{thm:khi-fibered}}
\label{sec:proof}

In this section we prove Theorem \ref{thm:khi-fibered}, restated below, as outlined  in the introduction.
{
\renewcommand{\thetheorem}{\ref{thm:khi-fibered}}
\begin{theorem}
Suppose $K$ is a genus $g>0$ fibered knot in $Y\not\cong \#^{2g}(S^1\times S^2)$ with fiber $\Sigma$. Then $\KHI(Y,K,[\Sigma],g-1)\neq 0$.
\end{theorem}
\addtocounter{theorem}{-1}
}

 Suppose for the rest of this section that $K$ is a fibered knot as in the hypothesis of Theorem \ref{thm:hfk-fibered}. Let $(\Sigma,h)$ be an open book corresponding to the fibration of $K$ with $g(\Sigma)=g,$ supporting a contact structure $\xi$ on $Y$. 
 We begin with some preliminaries.

\begin{definition}
 Given a properly embedded arc $a\subset \Sigma$, we say that $h$ sends $a$ \emph{to the left at an endpoint $p$} if $h(a)$ is not isotopic to $a$ and if, after isotoping $h(a)$ so that it intersects $a$ minimally, $h(a)$ is to the left of $a$ near $p$, as shown in Figure \ref{fig:left2}. 
 \end{definition}
 
This definition and the one below are due to Honda, Kazez, and Mati{\'c} \cite{hkm-rv}. 
 
 \begin{definition} $h$ is  \emph{not right-veering} if it sends some arc to the left at one of its endpoints.
\end{definition}

\begin{figure}[ht]
\labellist
\small \hair 2pt

\pinlabel $a$ at 47 28
\pinlabel $h(a)$ at 19 28
\pinlabel $p$ at 41 -3
\pinlabel $\Sigma$ at 65 57

\endlabellist
\centering
\includegraphics[width=2cm]{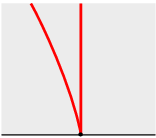}
\caption{$h$ sends $a$ to the left at $p$.
}
\label{fig:left2}
\end{figure}

\begin{remark}For the proof of Theorem \ref{thm:khi-fibered}, we may assume without loss of generality that the monodromy $h$ is not right-veering. 
Indeed, note that one of $h$ or $h^{-1}$  is not right-veering since otherwise  \[h=\id \textrm{ and } Y\cong \#^{2g}(S^1\times S^2).\] If $h$ is right-veering then we  use the fact that $\KHI$ is invariant under reversing the orientation of $Y$ and consider instead the knot $K\subset -Y$ with open book $(\Sigma,h^{-1})$. 
We will make clear below where we are relying on the assumption that $h$ is not right-veering; most of the results in this section do not require it.
\end{remark}
 
\begin{remark}Since $\KHI$ is invariant under reversing the orientation of $Y$, 
it suffices to prove that \[\KHI(-Y,K,[\Sigma],g-1)\neq 0.\] That is what we  do in this section.
\end{remark}

With these preliminaries out of the way, we may proceed to the proof of Theorem \ref{thm:khi-fibered}.

As the binding of the open book $(\Sigma,h)$, the knot $K$ is naturally a transverse knot in $(Y,\xi)$.  We may not be able to approximate $K$ by a Legendrian on a page of $(\Sigma, h)$, but we can positively stabilize $(\Sigma,h)$ to an open book $(S,f)$ as in Figure~\ref{fig:surface3}, whose monodromy
\[ f = h \circ D_\gamma \]
is the composition of $h$ with a positive Dehn twist around the curve $\gamma$ shown in the figure, and then let $\mathcal{K}_0^-$ denote the Legendrian approximation of $K$ shown there on a page of $(S,f)$.  The fact that this can be made Legendrian is a consequence of the Legendrian realization principle \cite[Theorem~3.7]{honda-lens}, because the union of two pages of $(S,f)$ is a convex surface $\tilde{S}$ whose dividing set $\tilde{\Gamma}$ is the binding, and because $\mathcal{K}_0^-$ is nonseparating in $R_+(\tilde{S}) = S$ and hence a ``non-isolating'' curve in $\tilde{S}$.

Since $\mathcal{K}_0^-$ lies in a convex surface $\tilde{S}$ and is disjoint from the dividing curves $\tilde{\Gamma}$, we have
\[ tb_S(\mathcal{K}_0^-) = tb_{\tilde{S}}(\mathcal{K}_0^-) = -\frac{1}{2}\#(\mathcal{K}_0^- \cap \tilde{\Gamma}) = 0, \]
where $tb_S(\mathcal{K}_0^-)$ denotes the twisting of the contact planes along $\mathcal{K}_0^-$ with respect to the framing given by $S$.
We observe that $S$ is obtained from $\Sigma$ by plumbing on a positive Hopf band, through which $\mathcal{K}_0^-$ passes once, and so it follows that \begin{equation}\label{eqn:tbneg1}tb_\Sigma(\mathcal{K}_0^-)=-1.\end{equation} We will see later (Remark \ref{rmk:tbneg1}) why this is important. 

Let 
 $\mathcal{K}_1^{\pm}$ denote the positive/negative Legendrian stabilization of $\mathcal{K}_0^-$. In particular, $\mathcal{K}_1^-$ is also a Legendrian approximation of $K$. Let \[(Y(K),\Gamma_i,\xi^\pm_i)\] denote the contact manifold with convex boundary and dividing set $\Gamma_i$ obtained by removing a standard neighborhood of $\mathcal{K}_i^\pm$ from $Y$.  
 
 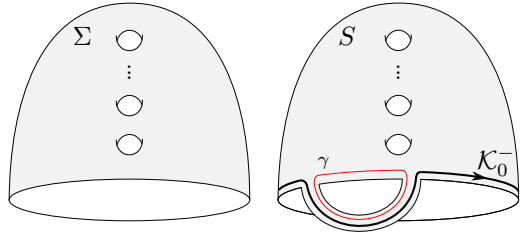
\begin{figure}[ht]
\labellist
\small \hair 2pt

\pinlabel $\Sigma$ at 97 249
\pinlabel $S$ at 425 249

\pinlabel $\mathcal{K}_0^-$ at 610 93

\tiny

\pinlabel $\gamma$ at 396 91

\endlabellist
\centering
\includegraphics[width=6.8cm]{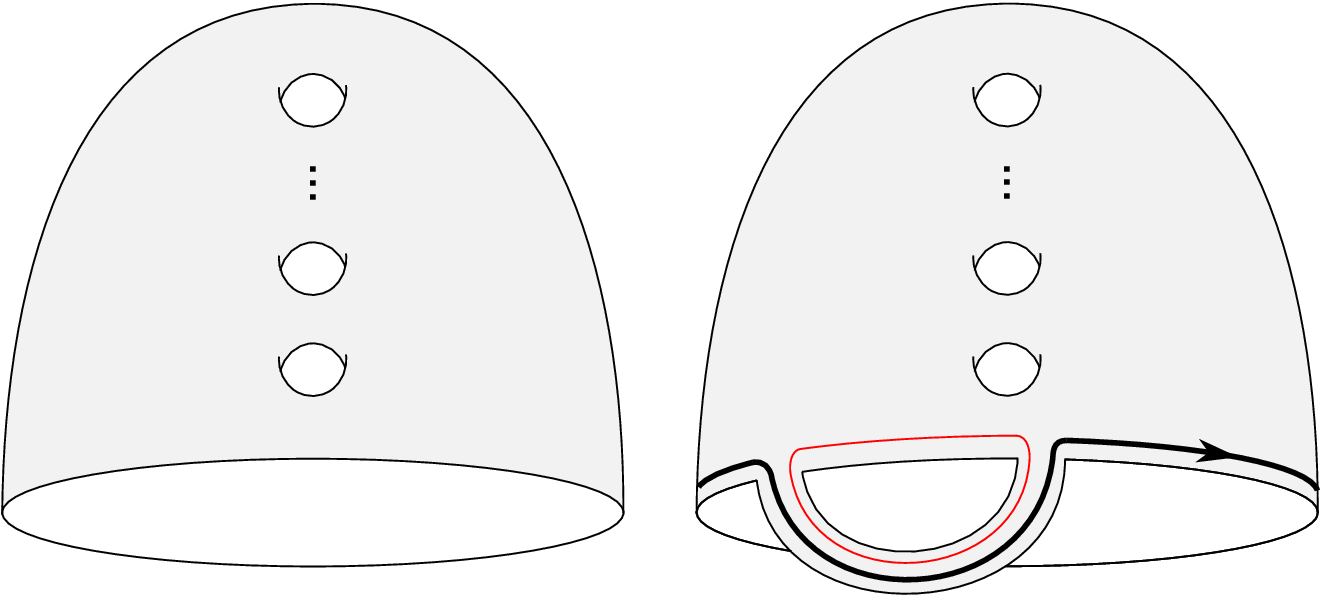}
\caption{Left, the fiber surface  for $K$. Right, a  stabilization of the open book $(\Sigma,h)$ obtained by attaching a $1$-handle  and composing $h$ with a positive Dehn twist around the  curve $\gamma$ in red, with  $\mathcal{K}=\mathcal{K}_0^-$ realized on the page $S$. 
}
\label{fig:surface3}
\end{figure}
 
 Recall from the previous section (Remark \ref{rmk:sv}) that the transverse invariant $\kinvt(K)$ is given by \[\kinvt(K) = \phi_0^{SV}(\cinvt(\xi_0^-)) = \phi_1^{SV}(\cinvt(\xi_1^-))\in\KHI(-Y,K)\] where  \[\phi_i^{SV}:\SHI(-Y(K),-\Gamma_i)\to\SHI(-Y(K),-\Gamma_\mu)=\KHI(-Y,K)\] is the map induced by a Stipsicz-V{\'e}rtesi bypass attachment. In the case $i=0$, this bypass is attached along the arc $c$ shown in Figure \ref{fig:bypasses}. Stipsicz and V{\'e}rtesi also showed in the proof of \cite[Theorem 1.5]{stipsicz-vertesi} that \[(Y(K),\Gamma_1,\xi^+_1)\textrm{ is obtained from }(Y(K),\Gamma_0,\xi^-_0)\] by attaching a bypass  along the arc $p$ shown in Figure \ref{fig:bypasses}. Letting
  \[\phi_0^+:\SHI(-Y(K),-\Gamma_0)\to\SHI(-Y(K),-\Gamma_1)\] denote the associated bypass attachment map, as in the introduction, we  have that \[\phi_0^+(\cinvt(\xi_0^-))=\cinvt(\xi_1^+).\]
Recall from the introduction that our proof of Theorem \ref{thm:khi-fibered}  begins with the following lemma.

  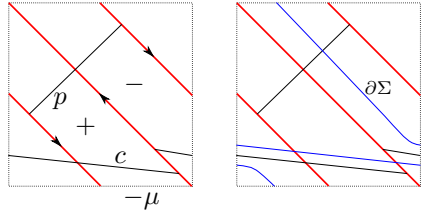
\begin{figure}[ht]
\labellist
\small \hair 2pt
\pinlabel $-\mu$ at 170 0

\pinlabel $+$ at 105 85
\pinlabel $-$ at 162 135

\pinlabel $p$ at 75 115
\pinlabel $c$ at 145 49
\tiny
\pinlabel $\partial\Sigma$ at 448 132
\endlabellist
\centering
\includegraphics[width=5.5cm]{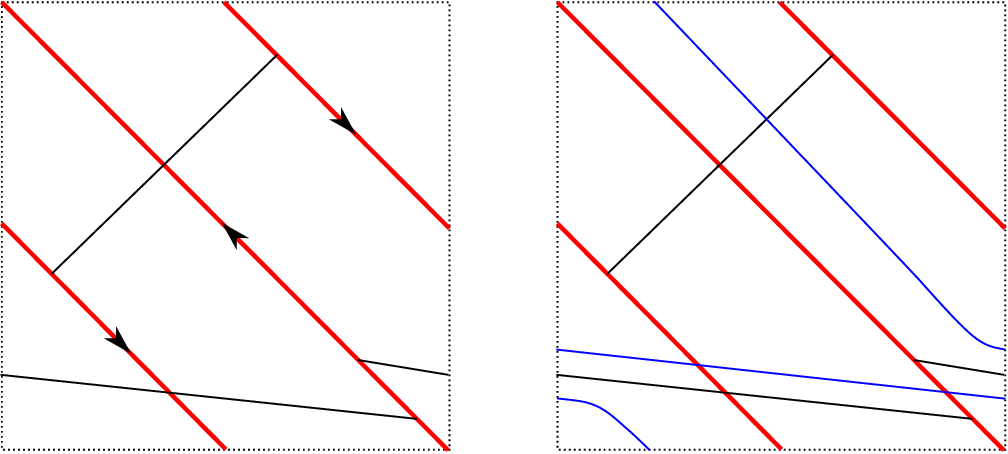}
\caption{The boundary of the complement of a standard neighborhood of $\mathcal{K}_0^-$ with dividing set $\Gamma_0$  in red. The Stipsicz-V{\'e}rtesi bypass  is attached along the arc $c$. Attaching a bypass along $p$ yields the complement of a standard neighborhood of  $\mathcal{K}_1^+$. Right, the  blue curve is the intersection of the boundary of the fiber surface $\Sigma$ of $K$ with this  torus.}
\label{fig:bypasses}
\end{figure}

 {
\renewcommand{\thetheorem}{\ref{lem:bypassclaim}}
\begin{lemma}
If the monodromy $h$ is not right-veering then $\phi_0^+(\cinvt(\xi_0^-))=\cinvt(\xi_1^+)=0.$
\end{lemma}
\addtocounter{theorem}{-1}
}
This in turn follows immediately from the lemma below, by Theorem \ref{thm:zero-overtwisted}.

{
\renewcommand{\thetheorem}{\ref{lem:otstab}}
\begin{lemma}
If the monodromy $h$ is not right-veering then $\xi_1^+$ is overtwisted.
\end{lemma}
\addtocounter{theorem}{-1}
}

\begin{remark}
Baker and Onaran proved in \cite[Theorem 5.2.3]{baker-onaran} a similar result under the strictly stronger assumption that $(\Sigma,h)$ is the negative stabilization of another open book.
\end{remark}

\begin{proof}[Proof of Lemma \ref{lem:otstab}]
Suppose $h$ is not right-veering. Let $a$ be a properly embedded arc in $\Sigma$ with endpoints $p$ and $q$ such that $h$ sends $a$ to the left at $p$. We can arrange that the $1$-handle added to $\Sigma$ in forming the stabilization $(S,f)$ and the curve $\gamma$ are as shown in the middle of Figure \ref{fig:surface}. Note in particular that $a$ intersects the Legendrian realization $\mathcal{K}_0^-\subset S$ in one point and with negative sign.

  \begin{figure}[ht]
\labellist
\small \hair 2pt

\pinlabel $\Sigma$ at 106 258
\pinlabel $S$ at 434 257

\pinlabel $S'$ at 769 254
\pinlabel $a$ at 233 118

\pinlabel $\mathcal{K}_0^-$ at 620 100
\pinlabel $\mathcal{K}_1^+$ at 955 98

\tiny
\pinlabel $p$ at 208 68
\pinlabel $q$ at 118 68
\pinlabel $\gamma$ at 396 96

\endlabellist
\centering
\includegraphics[width=11cm]{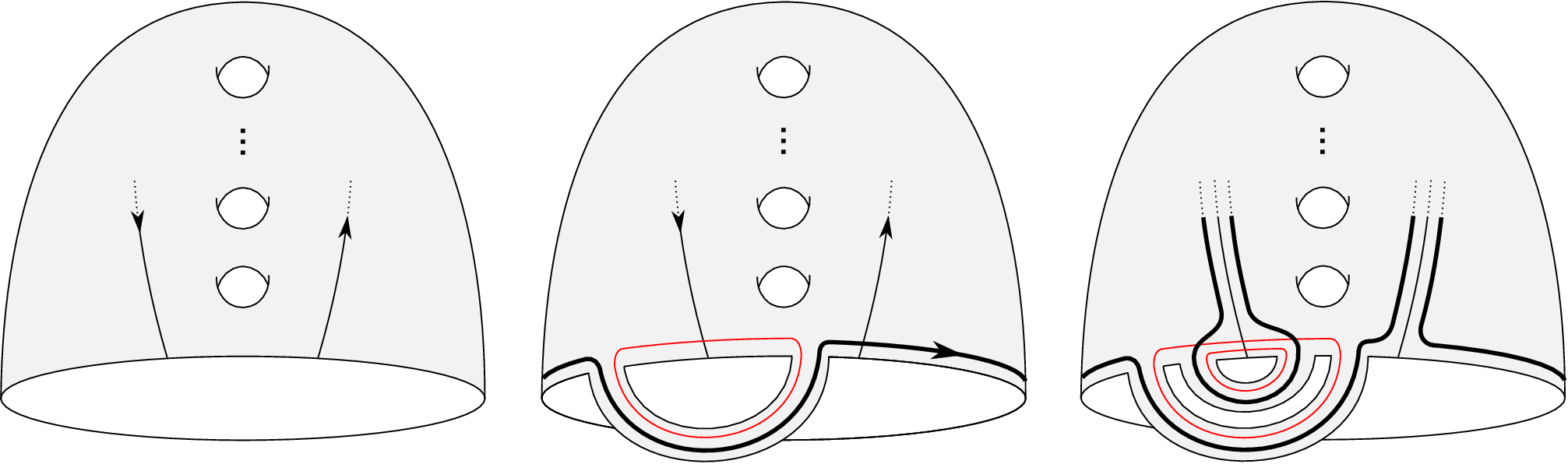}
\caption{Left, the fiber surface  for $K$ and the arc $a$ which is sent to the left at $p$ by the monodromy $h$. Middle, the   stabilization $(S,f)$ with  $\mathcal{K}_0^-$ realized on the page. Right, a further  stabilization  with $\mathcal{K}_1^+$ realized on the page $S'$.
}
\label{fig:surface}
\end{figure}
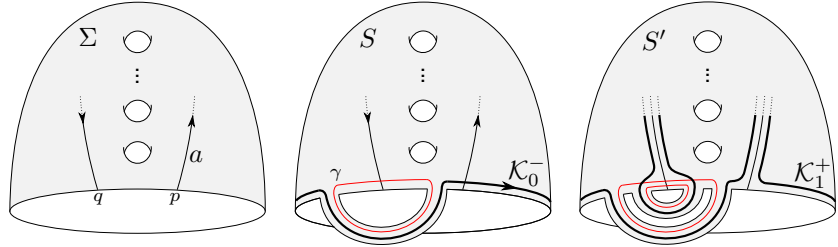

Given this setup, there is a standard way to realize $\mathcal{K}_1^+$ on a page of the open book $(S',f')$ obtained by positively stabilizing $(S,f)$, as shown in Figure \ref{fig:posstab}. Namely, $\mathcal{K}_1^+$ is  given by  the curve obtained by pushing $\mathcal{K}_0^-$ along the arc $a$ and then over the   $1$-handle that was added  in the stabilization. The right of Figure \ref{fig:surface} provides another view of this realization $\mathcal{K}_1^+\subset S'$.

\begin{figure}[ht]
\labellist
\small \hair 2pt
\pinlabel $S$ at 15 34

\pinlabel $S'$ at 178 34
\pinlabel $a$ at 78 29

\pinlabel $\mathcal{K}_0^-$ at 122 27
\pinlabel $\mathcal{K}_1^+$ at 285 27
\tiny
\pinlabel $p$ at 72 -1
\pinlabel $q$ at 72 50
\endlabellist
\centering
\includegraphics[width=8.7cm]{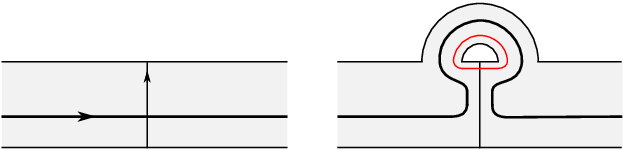}
\caption{Left, the Legendrian knot $\mathcal{K}_0^-$ on a page of the open book $(S,f)$. Right, the positive Legendrian stabilization $\mathcal{K}_1^+$ on a page of the stabilized open book $(S',f')$, where $f'$ is the composition of $f$ with the positive Dehn twist around the  curve in red. 
}
\label{fig:posstab}
\end{figure}
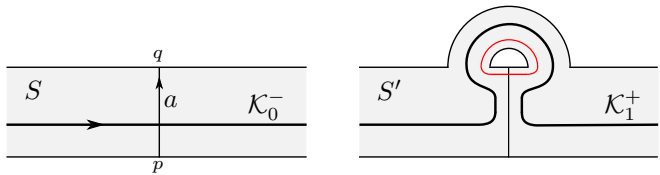

To show that $\xi_1^+$ is overtwisted, recall from Example \ref{eq:legendrian-complement} that  \[(S',P=S'\ssm \nu(\mathcal{K}_1^+), f'|_{P})\] is a partial open book for the complement $(Y(K),\Gamma_1,\xi_1^+).$ Note  that $a$ is an  arc in $P$. Moreover,  $f'|_{P}$ sends $a$ to the left at $p$ since $h$ does. This means that $f'|_{P}$ is not right-veering. Work of Honda, Kazez, and Mati{\'c}  \cite[Proposition 4.1]{hkm-sutured} then says that $\xi_1^+$ is overtwisted.
\end{proof}


We now  prove the remaining theorems which combine with Lemma \ref{lem:bypassclaim} to prove Theorem \ref{thm:khi-fibered}, as outlined in the introduction. First, we have the following.

{
\renewcommand{\thetheorem}{\ref{thm:ttopgrading}}
\begin{theorem}
$\kinvt(K)\in\KHI(-Y,K,[\Sigma],g)$.  In other words, the class $\kinvt(K)$ is supported in Alexander grading $g$.
\end{theorem}
\addtocounter{theorem}{-1}
}

\begin{proof}
We will first prove  that \begin{equation}\label{eqn:otheralex}\cinvt(\xi_0^-)\in \SHI(-Y(K),-\Gamma_0,[\Sigma],g).\end{equation} 
As   in Example \ref{eq:legendrian-complement}, a partial open book for the complement $(Y(K),\Gamma_0,\xi_0^-)$ is given by \[(S,P=S\ssm \nu(\mathcal{K}_0^-), f|_P).\] Let $\mathbf{c} = \{c_1,\dots,c_n\}$ be a basis for $P$ such  that the endpoints of the basis arcs intersect $\gamma$ as shown in the middle of Figure \ref{fig:handleseifert}. Let \[\gamma_1,\dots,\gamma_n\] be the corresponding curves on the boundary of $H_{S}=S\times[-1,1]$ as defined  in \eqref{eqn:basishandle}, so that  \[(Y(K),\Gamma_0,\xi_0^-)\] is obtained from \[(H_{S},\Gamma_{S},\xi_{S})\] by attaching contact $2$-handles along the $\gamma_i$. By Definition \ref{def:contact}, the element $\cinvt(\xi_0^-)$ is then the image of the generator \[\mathbf{1}=\cinvt(\xi_{S})\in\SHI(-H_{S},-\Gamma_{S})\]  under the map associated to these $2$-handle attachments.  
\begin{figure}[ht]
\labellist
\small \hair 2pt

\pinlabel $\Sigma$ at 218 33
\pinlabel $S$ at 538 33
\pinlabel $K$ at 33 52
\pinlabel $\mathcal{K}_0^-$ at 360 52
\pinlabel $\gamma$ at 520 75
\tiny
\pinlabel $c_i$ at 470 5

\pinlabel $\gamma_i$ at 790 5
\endlabellist
\centering
\includegraphics[width=12cm]{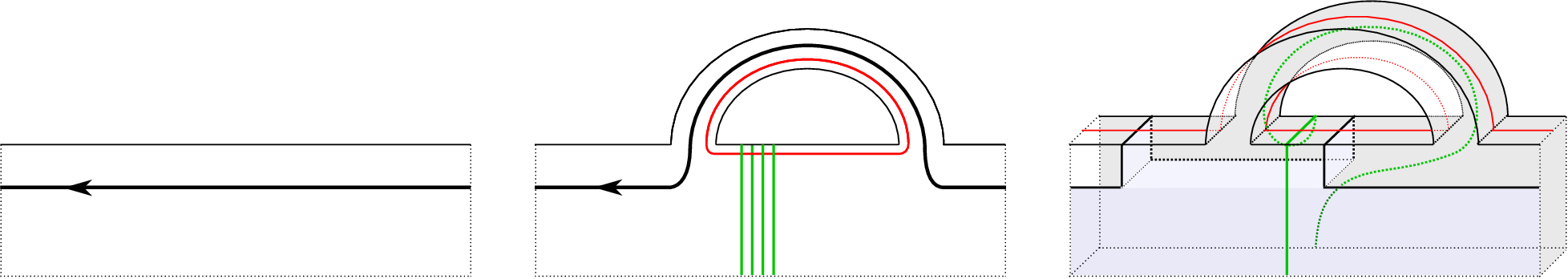}
\caption{Left, a portion of $\Sigma$ near its boundary with the knot $K$. Middle,  the stabilized surface $S$ with $\mathcal{K}_0^-$. The basis arcs $c_i$ are shown in green. Right, $(H_{S},\Gamma_{S})$ with a copy of the Seifert surface $\Sigma$ properly embedded in $H_{S}$ as shown in blue. Note that this copy $\Sigma$ is disjoint from the curves $\gamma_i$. Here, the red curves represent $\Gamma_S$. 
}
\label{fig:handleseifert}
\end{figure}
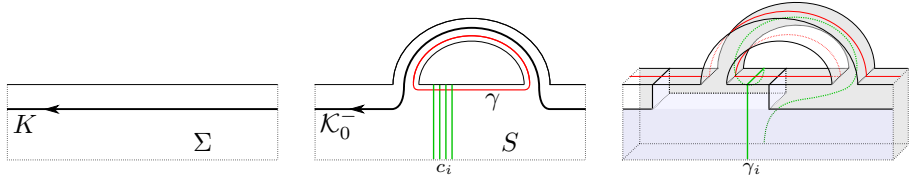

We claim that \begin{equation}\label{eqn:calex}\cinvt(\xi_{S})\in\SHI(-H_{S},-\Gamma_{S},[\Sigma],g).\end{equation} 
To prove  this, we let $\data_{S} = (Z,R,\eta,\alpha)$ be a closure of $(H_{S},\Gamma_{S})$ adapted to a properly embedded copy of $\Sigma$ in  $H_S$ as shown on the right of Figure \ref{fig:handleseifert}. Let us assume that this closure is formed  using an annular auxiliary surface. That is, to form $Z$ we first form a preclosure $M'$ by gluing  a thickened annulus to $H_S$ in such a way  that  $\Sigma'$ extends to a surface in $M'$ obtained from $\Sigma$ by adding a $1$-handle, as shown in the lower left of Figure \ref{fig:torusdifference}. 
\begin{figure}[ht]
\labellist
\small \hair 2pt
\pinlabel $B$ at 510 145
\endlabellist
\centering
\includegraphics[width=7.2cm]{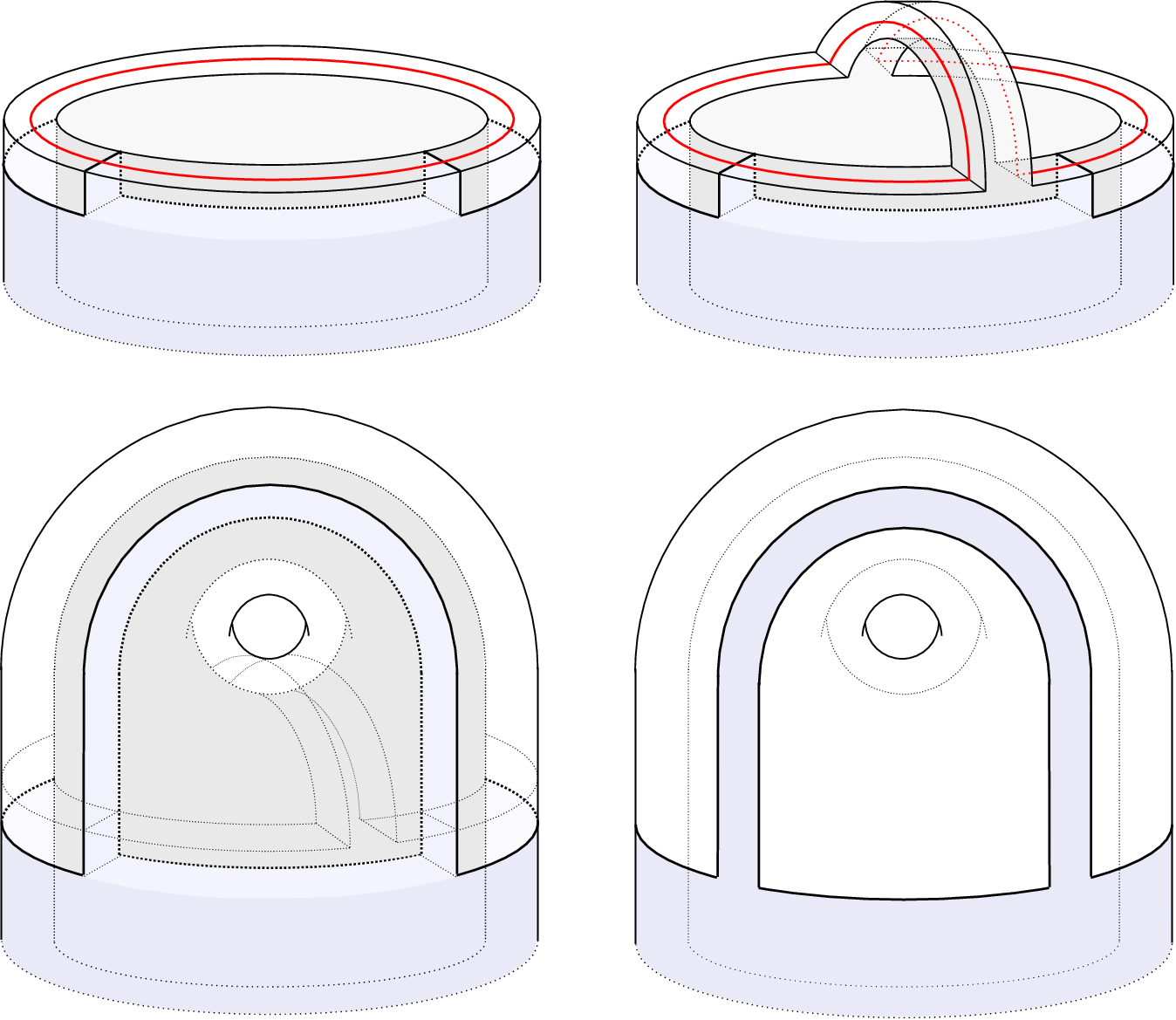}
\caption{Top left, the  product $(H_{\Sigma},\Gamma_\Sigma)$ with a copy of the surface $\Sigma$ shown in blue and properly embedded. Top right, the product $(H_{S},\Gamma_{S})$. Bottom left, the preclosure $M'$ of $(H_{S},\Gamma_{S})$ formed by gluing on a thickened annulus. The surface $\Sigma$ extends by a $1$-handle to the surface $\Sigma'$ shown in blue, with boundary curves $c_\pm$. Bottom right, $\Sigma'$ is isotopic to the subsurface of $R=\partial_+M'$ shown in blue. This subsurface differs from $R$ by the white annulus $B$.
}
\label{fig:torusdifference}
\end{figure}
Then $R$ is defined to be $\partial_+M$ and $Z$ is  formed  by gluing $R\times[1,3]$ to $M'$ as specified by  a diffeomorphism \[\varphi:\partial_+M'\to\partial_-M'\] which identifies the two boundary components \[c_\pm\subset \partial M'\] of $\Sigma'$. The surface $\Sigma'$  caps off in $Z$ to a closed surface $\overline\Sigma$, which is the union of $\Sigma'$ with the annulus \[A=c_+\times [1,3]\subset R\times[1,3].\] 
We claim that $\overline\Sigma$ and $R$ differ in $H_2(Z)$ by a torus. Indeed, if we remove the annulus $A$ from $\overline\Sigma$ we recover $\Sigma'$, which is isotopic to the subsurface of $R$ shown in blue in the lower right of Figure \ref{fig:torusdifference}. The union of this subsurface with the white annulus $B$ is equal to $R$. So, the difference $\overline\Sigma-R$ is homologous to the torus  $T=A\cup-B$ obtained as the union of these two annuli. Note that \[2g(R)-2=2g(\overline\Sigma)-2=2g.\] By definition, then, the class $\cinvt(\xi_{S},\data_{S})$ is  an element of the generalized $2g$-eigenspace of the operator \[\mu(R):I_*(-Z)_{-\alpha-\eta}\to I_*(-Z)_{-\alpha-\eta}.\] Corollary \ref{cor:grading-shift} then implies that it is also an element of the generalized $2g$-eigenspace of the operator \[\mu(\overline\Sigma):I_*(-Z|{-}R)_{-\alpha-\eta}\to I_*(-Z|{-}R)_{-\alpha-\eta}\] since $\overline\Sigma$ and $R$ differ by a torus. But the latter eigenspace is, by definition, $\SHI(-\data_{S},[\Sigma],g)$. This implies that $\cinvt(\xi_{S})$ lies in Alexander grading $g$, as claimed in \eqref{eqn:calex}.

To show that $\cinvt(\xi_0^-)$  lies in Alexander grading $g$, as  in \eqref{eqn:otheralex}, recall that on the level of closures, this class is the image of $\cinvt(\xi_S,\data_S)$ under the map induced by  the  cobordism associated to $\partial H_S$-framed surgeries on $\gamma_1,\dots,\gamma_n\subset Z$. The right of Figure \ref{fig:handleseifert} shows that these curves do not intersect the surface $\Sigma$ in $H_S$, which means that they are disjoint from the capped off surface $\overline\Sigma\subset Z$. This means that the capped off surfaces defining the Alexander gradings on the two ends of this cobordism are isotopic in the cobordism. Lemma \ref{lem:commute} then implies the cobordism map respects the Alexander grading, which proves the claim.

 To complete the proof of Theorem \ref{thm:ttopgrading}, we need only show that the bypass attachment map $\phi_0^{SV}$ preserves Alexander grading, keeping in mind that \[\kinvt(K)=\phi_0^{SV}(\cinvt(\xi_0^-)).\] The argument for this is the same as above. The map $\phi_0^{SV}$ is the composition of the maps associated to  attaching a $1$-handle with feet at the endpoints of the arc $c$ in Figure \ref{fig:bypasses} and then attaching a $2$-handle along the union of $c$ with an arc passing  once over this 1-handle.  Call this union $s$. Suppose $\data=(Z,R,\alpha,\eta)$ is a closure of the manifold obtained after the $1$-handle attachment, adapted to $\Sigma$. Then it is also  a closure of $(Y(K),\Gamma_0)$, and, on the level of closures, $\phi_0^{SV}$  is   induced by  the  cobordism associated to surgery on $s\subset Z$. The right of Figure \ref{fig:bypasses} shows that the arc $c$ does not intersect the surface $\Sigma\subset Y(K)$, which implies that the surgery curve $s$ is disjoint from the capped off surface $\overline\Sigma\subset Z$. This cobordism map therefore respects the Alexander grading, as argued above, which proves the claim.
\end{proof}

\begin{remark}
\label{rmk:tbneg1} It was important at the end of the proof of Theorem \ref{thm:ttopgrading} that the arc $c$ was disjoint from the surface $\Sigma\subset Y(K)$. This claim relied on the depiction  of $\partial \Sigma$  in Figure \ref{fig:bypasses}, so it behooves us to justify this depiction. This is where the observation \[tb_\Sigma(\mathcal{K}_0^-)=-1\]  in \eqref{eqn:tbneg1}  comes into play. First, this condition implies that \[tb_\Sigma(\mathcal{K}_1^+)=-2.\] These two Thurston-Bennequin calculations then  imply that  each component of  $\Gamma_0$ and $\Gamma_1$  intersects $\partial \Sigma$ in $1$ and $2$ points, respectively. This  forces  $\partial\Sigma$ to be as indicated in Figure \ref{fig:bypasses}.
\end{remark}

Next, we prove the following analogue of Vela-Vick's theorem \cite{vv} that the connected binding of an open book has nonzero transverse invariant in Heegaard Floer homology.

{
\renewcommand{\thetheorem}{\ref{thm:nonzero}}
\begin{theorem}
$\kinvt(K)$ is nonzero.
\end{theorem}
\addtocounter{theorem}{-1}
}
\begin{proof}
We first show that there is \emph{some} genus $g$ fibered knot $K'$  with open book $(\Sigma,\phi)$  such that the corresponding transverse invariant $\kinvt(K')$ is nonzero. We then use the surgery exact triangle to show that there  is an isomorphism \begin{equation}
\label{eqn:isoknots}\KHI(-Y',K',[\Sigma],g)\xrightarrow{\cong} \KHI(-Y'',K'',[\Sigma],g)\end{equation}  for \emph{any} two such fibered knots which sends $\kinvt(K')$ to $\kinvt(K'').$   Theorem \ref{thm:nonzero} will follow.

For the first, let $K'$ be the braid with two strands and $2g+1$ positive crossings. Note that $K'$ has genus $g$. As a positive braid, $K'$ is fibered with $g_4(K')=g$, with open book supporting the tight contact structure on $S^3$. As the transverse binding of this open book, we have that \[sl(K') = 2g-1.\] Let $\mathcal{K}'$ be the Legendrian representative of $K'$ shown in Figure \ref{fig:legbraid}. We have that \[tb(\mathcal{K}')=2g-1=2g_4(K')-1 \text{ and } rot(\mathcal{K}')=0.\] Theorem \ref{thm:tb} therefore implies that \[\linvt(\mathcal{K}')\neq 0.\]  We  claim  that $\mathcal{K}'$ is a Legendrian approximation of $K'$. For this, simply note that the transverse pushoff $K''$ of $\mathcal{K}'$  has self-linking number \[sl(K'')=tb(\mathcal{K}')+rot(\mathcal{K}')=2g-1.\] The transverse simplicity of torus knots,  proven by Etnyre and Honda  \cite{etnyre-honda}, then implies that $K''$ is transversely isotopic to the transverse binding $K'$. In other words, $\mathcal{K}'$ is a Legendrian approximation of $K'$. It  then follows that \[\kinvt(K'):=\linvt(\mathcal{K'})\] is nonzero, as desired.

\begin{figure}[ht]
\labellist
\tiny
\pinlabel $2g{+}1$ at 60 45
\pinlabel $\cdots$ at 64.8 31.5

\endlabellist
\centering
\includegraphics[width=4.5cm]{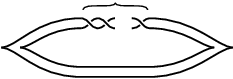}
\caption{A Legendrian approximation $\mathcal{K}'$ of the torus braid $K'=T(2g+1,2)$.
}
\label{fig:legbraid}
\end{figure}

Suppose next that $K'\subset Y'$ is a genus $g$ fibered knot  with open book $(\Sigma,\phi)$. Let $K''\subset Y''$ be the fibered knot corresponding to  the open book $(\Sigma,\phi\circ D_{\beta}^{-1})$, where $D_{\beta}$ is a positive Dehn twist around a nonseparating curve $\beta\subset \Sigma$. We  first prove the isomorphism \eqref{eqn:isoknots} for this pair.

Note that $\beta$ can be viewed as a Legendrian knot in the contact structure on $Y'$ corresponding to $(\Sigma,\phi)$, with contact framing equal to the framing induced by $\Sigma$. In this case, $Y''$ is the result of contact $(+1)$-surgery on $\beta$.  Let \[(Y'(K'),\Gamma_\mu,\xi_{\mu,\mathcal{K}'})\textrm{ and } (Y''(K''),\Gamma_\mu,\xi_{\mu,\mathcal{K}''})\] be the sutured contact manifolds obtained by removing standard neighborhoods of Legendrian approximations $\mathcal{K}'$ and $\mathcal{K}''$  of $K'$ and $K''$ and attaching Stipsicz-V{\'e}rtesi bypasses, as  described in the previous section. Then  $\beta$ is naturally a Legendrian knot in the first of these sutured contact manifolds, as $\Sigma$ is a subsurface of a page of a partial open book for this contact manifold. The second sutured contact manifold above is the result of  contact $(+1)$-surgery on $\beta$. Moreover, we have that \begin{equation}\label{eqn:Tcontact}\kinvt(K'):=\theta(\xi_{\mu,\mathcal{K}'})\textrm{ and }\kinvt(K''):=\theta(\xi_{\mu,\mathcal{K}''})\end{equation} as per Remark \ref{rmk:sv}.

Let $\data = (Z,R,\alpha,\eta)$ be a closure of $(Y'(K'),\Gamma_\mu)$ adapted to $\Sigma$. Let 
\[\data_1=(Z_1,R,\eta,\alpha)\textrm{ and }
\data_0=(Z_0,R,\eta,\alpha)
\] be the tuples obtained from $\data$ by performing $1$- and $0$-surgery on $\beta$ with respect to the framing induced by $\Sigma$. These are naturally closures of the sutured manifolds 
\[(Y''(K''),\Gamma_\mu) =(Y'(K')_1(\beta),\Gamma_\mu)\textrm{ and } (Y'(K')_0(\beta),\Gamma_\mu),\] adapted to $\Sigma$ in each case.
By Theorem \ref{thm:exacttri}, there is a surgery exact triangle 
\[ \xymatrix@C=-35pt@R=30pt{
I_*(-Z|{-}R)_{-\alpha -\eta} \ar[rr]^{I_*(W)_{{\kappa}}}  & &I_*(-Z_1|{-}R)_{-\alpha -\eta} \ar[dl] \\
&I_*(-Z_0|{-}R)_{-\alpha -\eta+ \mu}\ar[ul] & \\
} \] as in \eqref{eqn:surgeryexacttriangleZ}, where $\mu$ is the curve in $-Z_0$ corresponding to the meridian of $ \beta\subset -Z$.

We can push $\beta$ slightly off of  $\Sigma$ to ensure that the capped off surfaces $\overline{\Sigma}$ in these closures are isotopic in the cobordisms between them. Lemma \ref{lem:commute} then implies that the maps in this exact triangle respect the eigenspace decompositions associated with the operators $\mu(\overline\Sigma)$. In particular, we have an exact triangle \[ \xymatrix@C=-35pt@R=30pt{
\SHI(-\data,[\Sigma],g) \ar[rr]^{I_*(W)_{{\kappa}}}  & &\SHI(-\data_1,[\Sigma],g)  \ar[dl] \\
&I_*(-Z_0|{-}R)_{-\alpha -\eta+ \mu}^{(2g)},\ar[ul] & \\
} \] where the third group denotes the generalized $2g$-eigenspace of the operator $\mu(\overline\Sigma)$ acting on $I_*(-Z_0|{-}R)_{-\alpha -\eta+ \mu}$. But the surface $\overline\Sigma$ is homologous  in $-Z_0$ to a surface of genus \[g(\overline\Sigma)-1=g\] obtained by surgering $\overline\Sigma$ along $\beta$, so Proposition \ref{prop:mu-spectrum} tells us that this $2g$-eigenspace is trivial. The map $I_*(W)_{{\kappa}}$ is therefore an isomorphism. Recall that this map gives rise to the map we denote by $F_\beta$ in Section~\ref{sssec:handles}. We have shown above that $F_\beta$ restricts to an isomorphism \[F_\beta:\KHI(-Y',K',[\Sigma],g)\to\KHI(-Y'',K'',[\Sigma],g).\] Moreover, this map sends $\kinvt(K')$ to $\kinvt(K'')$ by Theorem~\ref{thm:legendrian-surgery} and \eqref{eqn:Tcontact}.

Finally, since any two genus $g$ fibered knots $K'\subset Y'$ and $K''\subset Y''$ with fiber $\Sigma$ have  open books with  monodromies related by positive and negative Dehn twists around nonseparating curves in $\Sigma$, we conclude that there is for any two such knots an isomorphism  \[\KHI(-Y',K',[\Sigma],g)\to \KHI(-Y'',K'',[\Sigma],g)\] sending $\kinvt(K')$ to $\kinvt(K'').$  As discussed above, this completes the proof of Theorem \ref{thm:nonzero}.
\end{proof}

The bypass attachment along the arc $p$ in Figure \ref{fig:bypasses} fits into a bypass triangle as shown in Figure \ref{fig:bypasstri2}. Note that the arc of attachment in the upper right is precisely the arc defining the Stipsicz-V{\'e}rtesi  bypass attachment which induces the map $\phi_1^{SV}$.

\begin{figure}[ht]
\labellist
\small \hair 2pt

\pinlabel $+$ at 402 370
\pinlabel $-$ at 459 420

\pinlabel $+$ at 782 370
\pinlabel $=$ at 905 421

\pinlabel $-$ at 837 420
\tiny
\pinlabel $p$ at 345 428
\pinlabel $\phi_0^+$ at 578 452
\pinlabel $\phi_1^{SV}$ at 758 261
\pinlabel $C$ at 410 260
\endlabellist
\centering
\includegraphics[width=11.5cm]{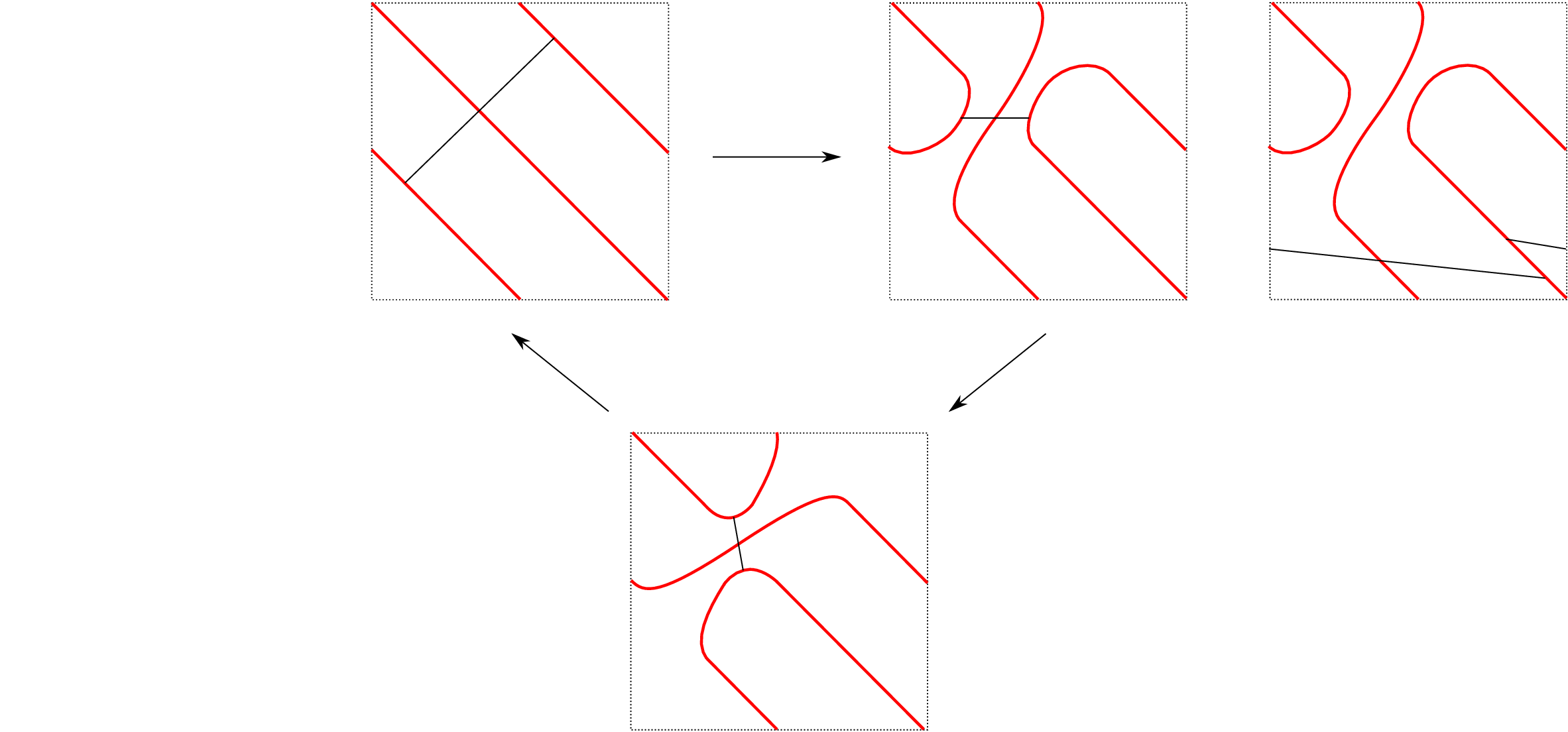}
\caption{The arcs of attachment for the bypass exact triangle \eqref{eqn:bypasstri}. The arc in the upper right  specifies  a Stipsicz-V{\'e}rtesi  bypass attachment.
}
\label{fig:bypasstri2}
\end{figure}
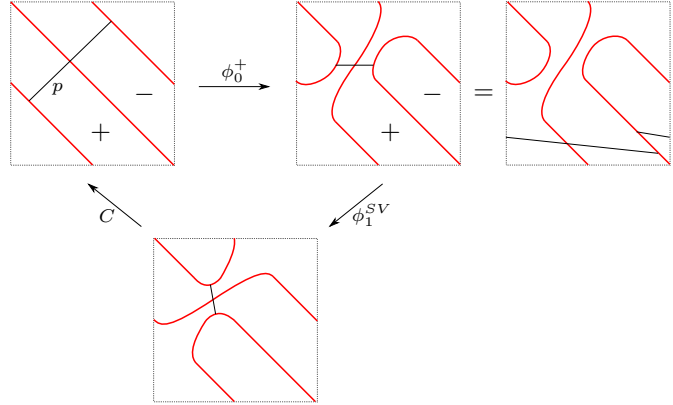

By Theorem \ref{thm:bypass}, there is a corresponding bypass exact triangle of the form
\begin{equation}\label{eqn:bypasstri} \xymatrix@C=-25pt@R=35pt{
\SHI(-Y(K),-\Gamma_0) \ar[rr]^{\phi_0^+} & & \SHI(-Y(K),-\Gamma_1) \ar[dl]^{\phi_1^{SV}} \\
& \KHI(-Y,K) \ar[ul]^C. & \\
} \end{equation}

To prove Theorem \ref{thm:khi-fibered}, let us now assume that the monodromy $h$ is not right-veering. Then \[\phi_0^+(\cinvt(\xi_0^-))=0\] by Lemma \ref{lem:bypassclaim}. Exactness of the triangle \eqref{eqn:bypasstri} then tells us that there is a class \[x\in\KHI(-Y,K)\] such that \[C(x)=\cinvt(\xi_0^-).\] 
 The composition \begin{equation}\label{eqn:phicomp}\phi_0^{SV}\circ C: \KHI(-Y,K)\to \KHI(-Y,K)\end{equation} therefore satisfies \begin{equation}\label{eqn:xmapsto}\phi_0^{SV}(C(x))=\kinvt(K),\end{equation}  which is nonzero by Theorem \ref{thm:nonzero}. Thus, $x$ is nonzero.
 
The class $x$ is not \emph{a priori}  homogeneous with respect to the Alexander grading on $\KHI(-Y,K)$. However, we prove the following, which completes the proof of Theorem \ref{thm:khi-fibered}.

\begin{theorem}
The component of $x$  in $\KHI(-Y,K,[\Sigma],g-1)$ is nonzero.
\end{theorem}

\begin{proof}
The composite map $\phi_0^{SV}\circ C$ is ultimately induced by a  cobordism \[(W,\nu):(-Y_1,-\alpha\sqcup-\eta)\to (-Y_2,-\alpha\sqcup-\eta)\] from a closure \[-\data_1=(-Y_1,-R,-\eta,-\alpha)\textrm{ of }(-Y(K),-\Gamma_\mu)\] to another closure \[-\data_2=(-Y_2,-R,-\eta,-\alpha)\textrm{ of }(-Y(K),-\Gamma_\mu),\] corresponding to surgery on curves away from the embedded copies of $-Y(K)$ in these $-Y_i$. We can arrange that these two closures are adapted to $\Sigma$. The capped off surfaces $\overline\Sigma_a\subset -Y_1$ and  $\overline\Sigma_b\subset -Y_2$ in these closures are unions of $\Sigma$ with once-punctured tori, \[\overline\Sigma_a=\Sigma\cup T_a\textrm{ and } \overline\Sigma_b=\Sigma\cup T_b.\] The surfaces $\Sigma\subset Y_i$ are isotopic within $W$ as they are contained on the boundary of a product region \[-Y(K)\times I\subset W.\] We therefore have that \begin{equation}\label{eqn:difference}\overline\Sigma_a + F = \overline\Sigma_b\end{equation} in $H_2(W)$, where $F$ is the genus two surface in $W$ with self-intersection $0$ given as the union \[F=T_b\cup (\partial \Sigma\times I) \cup -T_a\] of these once-punctured tori. Let $y \in \SHI(-\data_1)$ be the element representing the class $x \in \KHI(-Y,K).$  Write \[y=y_{-g}+y_{1-g}+\dots +y_{g-1}+y_g\] where $y_i\in \SHI(-\data_1,[\Sigma],i)$. 
We know from \eqref{eqn:xmapsto} that the map $I_*(W)_\nu$ sends $y$  to a representative of $\kinvt(K)$, which must be   a nonzero element of $\SHI(-\data_2,[\Sigma],g)$ by Theorems \ref{thm:ttopgrading} and \ref{thm:nonzero}. 
It follows from Proposition \ref{prop:grading-shift} and the relation \eqref{eqn:difference}, however, that the  only components of $y$ whose images can have nonzero components in $\SHI(-\data_2,[\Sigma],g)$ are $y_{g-1}$ and $y_g$, since $g(F)=2$. On the other hand, we claim that \[I_*(W)_\nu(y_g)=0.\] This will prove that $y_{g-1}$ must be nonzero, proving the theorem. For this claim, note that \[\SHI(-Y,K,[\Sigma],g)\cong \C\] since $K$ is fibered. This group is therefore generated by $\kinvt(K)$ by Theorems \ref{thm:ttopgrading} and \ref{thm:nonzero}. If $y_g$ is zero then we are done. If not, then $y_g$ represents a nonzero multiple of $\kinvt(K)$, so \[I_*(W)_\nu(y_g) = 0 \textrm{ if and only if }\phi_0^{SV}(C(\kinvt(K)))= 0.\] Thus, it remains to show that the latter is zero. For this, recall that \[\kinvt(K) = \phi_1^{SV}(\cinvt(\xi_1^-)).\] It then follows that \[C(\kinvt(K)) = C(\phi_1^{SV}(\cinvt(\xi_1^-)))=0\] by the exactness of the triangle \eqref{eqn:bypasstri}.
\end{proof}

\bibliographystyle{alpha}
\bibliography{References}

\begin{thebibliography}{EVVZ17}

\bibitem[BO15]{baker-onaran}
Kenneth~L. Baker and Sinem Onaran.
\newblock Nonlooseness of nonloose knots.
\newblock {\em Algebr. Geom. Topol.}, 15(2):1031--1066, 2015.

\bibitem[BS15]{bs-naturality}
John~A. Baldwin and Steven Sivek.
\newblock Naturality in sutured monopole and instanton homology.
\newblock {\em J. Differential Geom.}, 100(3):395--480, 2015.

\bibitem[BS16a]{bs-shm}
John~A. Baldwin and Steven Sivek.
\newblock A contact invariant in sutured monopole homology.
\newblock {\em Forum Math. Sigma}, 4:e12, 82, 2016.

\bibitem[BS16b]{bs-shi}
John~A. Baldwin and Steven Sivek.
\newblock Instanton {F}loer homology and contact structures.
\newblock {\em Selecta Math. (N.S.)}, 22(2):939--978, 2016.

\bibitem[BS18]{bs-legendrian}
John~A. Baldwin and Steven Sivek.
\newblock Invariants of {L}egendrian and transverse knots in monopole knot
  homology.
\newblock {\em J. Symplectic Geom.}, 16(4):959--1000, 2018.

\bibitem[BVV18]{bvv}
John Baldwin and David~Shea Vela-Vick.
\newblock A note on the knot {F}loer homology of fibered knots.
\newblock {\em Algebr. Geom. Topol.}, 18(6):3669--3690, 2018.

\bibitem[DG09]{ding-geiges-handles}
Fan Ding and Hansj\"org Geiges.
\newblock Handle moves in contact surgery diagrams.
\newblock {\em J. Topol.}, 2(1):105--122, 2009.

\bibitem[DK90]{donaldson-kronheimer}
S.~K. Donaldson and P.~B. Kronheimer.
\newblock {\em The geometry of four-manifolds}.
\newblock Oxford Mathematical Monographs. The Clarendon Press, Oxford
  University Press, New York, 1990.
\newblock Oxford Science Publications.

\bibitem[Don02]{donaldson-book}
S.~K. Donaldson.
\newblock {\em Floer homology groups in {Y}ang-{M}ills theory}, volume 147 of
  {\em Cambridge Tracts in Mathematics}.
\newblock Cambridge University Press, Cambridge, 2002.
\newblock With the assistance of M. Furuta and D. Kotschick.

\bibitem[EH01]{etnyre-honda}
John~B. Etnyre and Ko~Honda.
\newblock Knots and contact geometry. {I}. {T}orus knots and the figure eight
  knot.
\newblock {\em J. Symplectic Geom.}, 1(1):63--120, 2001.

\bibitem[Etn03]{etnyre-intro}
John~B. Etnyre.
\newblock Introductory lectures on contact geometry.
\newblock In {\em Topology and geometry of manifolds ({A}thens, {GA}, 2001)},
  volume~71 of {\em Proc. Sympos. Pure Math.}, pages 81--107. Amer. Math. Soc.,
  Providence, RI, 2003.

\bibitem[Etn06]{etnyre-ob}
John~B. Etnyre.
\newblock Lectures on open book decompositions and contact structures.
\newblock In {\em Floer homology, gauge theory, and low-dimensional topology},
  volume~5 of {\em Clay Math. Proc.}, pages 103--141. Amer. Math. Soc.,
  Providence, RI, 2006.

\bibitem[EVVZ17]{evvz}
John~B. Etnyre, David~Shea Vela-Vick, and Rumen Zarev.
\newblock Sutured {F}loer homology and invariants of {L}egendrian and
  transverse knots.
\newblock {\em Geom. Topol.}, 21(3):1469--1582, 2017.

\bibitem[Gei08]{geiges}
Hansj\"{o}rg Geiges.
\newblock {\em An introduction to contact topology}, volume 109 of {\em
  Cambridge Studies in Advanced Mathematics}.
\newblock Cambridge University Press, Cambridge, 2008.

\bibitem[Gir91]{giroux-convexite}
Emmanuel Giroux.
\newblock Convexit\'{e} en topologie de contact.
\newblock {\em Comment. Math. Helv.}, 66(4):637--677, 1991.

\bibitem[Gir00]{giroux-lens}
Emmanuel Giroux.
\newblock Structures de contact en dimension trois et bifurcations des
  feuilletages de surfaces.
\newblock {\em Invent. Math.}, 141(3):615--689, 2000.

\bibitem[HKM07]{hkm-rv}
Ko~Honda, William~H. Kazez, and Gordana Mati\'{c}.
\newblock Right-veering diffeomorphisms of compact surfaces with boundary.
\newblock {\em Invent. Math.}, 169(2):427--449, 2007.

\bibitem[HKM08]{hkm-tqft}
Ko~Honda, William~H. Kazez, and Gordana Mati{\'c}.
\newblock Contact structures, sutured {F}loer homology, and {TQFT}.
\newblock Preprint, \href{http://arxiv.org/abs/0807.2431}{\tt
  arXiv:math/0807.2431v1}[math.GT], 2008.

\bibitem[HKM09]{hkm-sutured}
Ko~Honda, William~H. Kazez, and Gordana Mati{\'c}.
\newblock The contact invariant in sutured {F}loer homology.
\newblock {\em Invent. Math.}, 176(3):637--676, 2009.

\bibitem[Hon00]{honda-lens}
Ko~Honda.
\newblock On the classification of tight contact structures. {I}.
\newblock {\em Geom. Topol.}, 4:309--368, 2000.

\bibitem[HW18]{hw}
Matthew Hedden and Liam Watson.
\newblock On the geography and botany of knot {F}loer homology.
\newblock {\em Selecta Math. (N.S.)}, 24(2):997--1037, 2018.

\bibitem[KM95]{km-gauge2}
P.~B. Kronheimer and T.~S. Mrowka.
\newblock Gauge theory for embedded surfaces. {II}.
\newblock {\em Topology}, 34(1):37--97, 1995.

\bibitem[KM10a]{km-alexander}
Peter Kronheimer and Tom Mrowka.
\newblock Instanton {F}loer homology and the {A}lexander polynomial.
\newblock {\em Algebr. Geom. Topol.}, 10(3):1715--1738, 2010.

\bibitem[KM10b]{km-excision}
Peter Kronheimer and Tomasz Mrowka.
\newblock Knots, sutures, and excision.
\newblock {\em J. Differential Geom.}, 84(2):301--364, 2010.

\bibitem[KM11]{km-khovanov}
P.~B. Kronheimer and T.~S. Mrowka.
\newblock Khovanov homology is an unknot-detector.
\newblock {\em Publ. Math. Inst. Hautes \'Etudes Sci.}, (113):97--208, 2011.

\bibitem[Krc15]{krcatovich}
David Krcatovich.
\newblock The reduced knot {F}loer complex.
\newblock {\em Topology Appl.}, 194:171--201, 2015.

\bibitem[LOSS09]{loss}
Paolo Lisca, Peter Ozsv\'ath, Andr\'as~I. Stipsicz, and Zolt\'an Szab\'o.
\newblock Heegaard {F}loer invariants of {L}egendrian knots in contact
  three-manifolds.
\newblock {\em J. Eur. Math. Soc. (JEMS)}, 11(6):1307--1363, 2009.

\bibitem[LS04]{lisca-stipsicz}
Paolo Lisca and Andr\'as~I. Stipsicz.
\newblock Ozsv\'ath-{S}zab\'o invariants and tight contact three-manifolds.
  {I}.
\newblock {\em Geom. Topol.}, 8:925--945, 2004.

\bibitem[Mas14]{massot}
Patrick Massot.
\newblock Topological methods in 3-dimensional contact geometry.
\newblock In {\em Contact and symplectic topology}, volume~26 of {\em Bolyai
  Soc. Math. Stud.}, pages 27--83. J\'{a}nos Bolyai Math. Soc., Budapest, 2014.

\bibitem[Mu{\~n}99]{munoz}
Vicente Mu{\~n}oz.
\newblock Ring structure of the {F}loer cohomology of {$\Sigma\times{\bf
  S}^1$}.
\newblock {\em Topology}, 38(3):517--528, 1999.

\bibitem[Ozb11]{ozbagci}
B.~Ozbagci.
\newblock Contact handle decompositions.
\newblock {\em Topology Appl.}, 158(5):718--727, 2011.

\bibitem[Sca15]{scaduto}
Christopher~W. Scaduto.
\newblock Instantons and odd {K}hovanov homology.
\newblock {\em J. Topol.}, 8(3):744--810, 2015.

\bibitem[Siv12]{sivek-legendrian}
Steven Sivek.
\newblock Monopole {F}loer homology and {L}egendrian knots.
\newblock {\em Geom. Topol.}, 16(2):751--779, 2012.

\bibitem[SV09]{stipsicz-vertesi}
Andr\'as~I. Stipsicz and Vera V\'ertesi.
\newblock On invariants for {L}egendrian knots.
\newblock {\em Pacific J. Math.}, 239(1):157--177, 2009.

\bibitem[VV11]{vv}
David~Shea Vela-Vick.
\newblock On the transverse invariant for bindings of open books.
\newblock {\em J. Differential Geom.}, 88(3):533--552, 2011.

\end{thebibliography}

\end{document}